\documentclass [12pt]{amsart}
\usepackage[T1]{fontenc}
\usepackage[utf8]{inputenc}
\usepackage[margin=1in]{geometry}
\usepackage{graphicx}
\usepackage{amssymb}
\usepackage{amsthm}
\usepackage{mathtools}
\usepackage{hyperref}
\usepackage{graphicx}
\usepackage{parskip}
\usepackage{multirow}
\usepackage{polynom}
\usepackage{centernot}
\usepackage[usenames,dvipsnames]{color}
\usepackage{chngcntr}
\counterwithout{figure}{section}

\usepackage{amsmath,amssymb,amsthm,amsfonts}
\usepackage{mathtools}%

\usepackage{comment}
\usepackage{tikz,graphicx,color}
\usepackage{tikz-cd}%
\usepackage{tikz-3dplot}%
\usetikzlibrary{calc}%
\usetikzlibrary{arrows}%
\usetikzlibrary{shapes}%
\usetikzlibrary{patterns}%
\usetikzlibrary{positioning}%
\usetikzlibrary{arrows.meta}
\usetikzlibrary{decorations.markings}
\usepackage{epstopdf}%
\pdfpageattr{/Group <</S /Transparency /I true /CS /DeviceRGB>>}

\usepackage{array}

\usepackage[margin=1in]{geometry}

\usepackage{soul}
\usepackage{accents}

\usepackage{xspace}

\usetikzlibrary{knots}
\usepackage{enumitem}

\usepackage{letltxmacro}
\usepackage{thmtools,etoolbox}

\def\myarabic#1{\normalfont(\roman{#1})}
\newlist{theoremlist}{enumerate}{1}
\setlist[theoremlist]{label=\myarabic{theoremlisti},ref={\myarabic{theoremlisti}},itemindent=0pt,labelindent=0pt,
leftmargin=*,noitemsep}

\makeatletter
\renewcommand{\p@theoremlisti}{\perh@ps{\thetheorem}}
\protected\def\perh@ps#1#2{\textup{#1#2}}
\newcommand{\itemrefperh@ps}[2]{\textup{#2}}
\newcommand{\itemref}[1]{\begingroup\let\perh@ps\itemrefperh@ps\ref{#1}\endgroup}
\makeatother

\usepackage{nameref,hyperref}
\usepackage[capitalize]{cleveref}

\protected\def\ignorethis#1\endignorethis{}
\let\endignorethis\relax

\addtotheorempostheadhook[theorem]{\crefalias{theoremlisti}{theorem}}
\addtotheorempostheadhook[lemma]{\crefalias{theoremlisti}{lemma}}
\addtotheorempostheadhook[proposition]{\crefalias{theoremlisti}{proposition}}
\addtotheorempostheadhook[corollary]{\crefalias{theoremlisti}{corollary}}

\def\Acal{\mathcal{A}}\def\Dcal{\mathcal{D}}\def\Ical{\mathcal{I}}\def\Mcal{\mathcal{M}}\def\Rcal{\mathcal{R}}\def\Xcal{\mathcal{X}}

\def\C{\mathbb{C}}
\def\R{\mathbb{R}}

\def\Z{\mathbb{Z}}

\newcommand\parr[1]{{({#1})}}
\def\<{{\langle}}
\def\>{{\rangle}}

\def\la{{\lambda}}

\def\multiset#1#2{\left(\!\left({#1\atopwithdelims..#2}\right)\!\right)}

\def\Gr{\operatorname{Gr}}
\def\Grkn{\Gr(k,n)}

\def\j{{\mathbf{j}}}

\def\Racc{R^\circ}
\def\Rich_#1^#2{\Racc_{#1,#2}}
\def\Richcl_#1^#2{R_{#1,#2}}

\def\bw{{\mathbf{w}}}
\def\bv{{\mathbf{v}}}

\def\F{{\mathbb{F}}}

\def\Fl{{\operatorname{Fl}}}

\def\Fl{{\operatorname{Fl}}}

\def\Pio{\Pi^\circ}

\def\Saff{\widetilde{S}_n}
\def\Saffk{\widetilde{S}_{n}^{(k)}}
\def\Saffz{\widetilde{S}_{n}^{(0)}}
\def\Anaff{\widetilde{\mathcal{A}_{n-1}}}

\def\BND{\mathbf{B}}

\numberwithin{equation}{section}

\def\Richafftp_#1^#2{{\Rcal_{#1,#2}^{>0}}}

\def\Richaff_#1^#2{\accentset{\circ}{\mathcal{R}}_{#1,#2}}

\def\Dyck{\operatorname{Dyck}}

\def\Dyckxx(#1,#2){\Dyck_{#1,#2}}

\def\Pit_#1{\Xcal^\circ_{#1}}

\def\Ical{\mathcal{I}}

\def\Povar_#1{\accentset{\circ}{\Pi}_{#1}}
\def\Povarcl_#1{\Pi_{#1}}
\def\RPovar_#1{\accentset{\circ}{\Pi}^\R_{#1}}
\def\RPovarcl_#1{\Pi^\R_{#1}}
\def\Povtp_#1{\Pi_{#1}^{>0}}
\def\Povtnn_#1{\Pi_{#1}^{\geq0}}

\def\sGr{\overleftarrow{\Ical}}
\def\sGrI_#1{\overleftarrow{I_{#1}}}
\def\tGr{\overrightarrow{\Ical}}
\def\tGrI_#1{\overrightarrow{I_{#1}}}

\def\GL{\operatorname{GL}}
\def\Fl{\operatorname{Fl}}

\def\Gr{\operatorname{Gr}}

\def\KLR_#1^#2{R_{#1,#2}(q)}

\def\sclbx{1}
\def\sclbxbigdeo{0.9}
\def\ybscl{0.9}

\def\op{1}

\def\gridop{0.4}
\newcommand{\be}{\begin{enumerate}[a)]}
\newcommand{\ee}{\end{enumerate}}
\def\bal#1\eal{\begin{align*}#1\end{align*}}
\allowdisplaybreaks

\newcommand{\floor}[1]{\left\lfloor #1 \right\rfloor}

\newtheorem*{thm*}{Theorem}
\newtheorem*{conj*}{Conjecture}

\newtheorem{theorem}{Theorem}[section]

\newtheorem{problem}[theorem]{Problem}
\newtheorem{lemma}[theorem]{Lemma}

\newtheorem{proposition}[theorem]{Proposition}
\newtheorem{remark}[theorem]{Remark}

\newtheorem{definition}[theorem]{Definition}

\newtheorem{corollary}[theorem]{Corollary}

\newcommand{\pn}[1]{\left( #1 \right)}
\newcommand{\lra}[1]{\langle #1 \rangle}

\def\multiset#1#2{\ensuremath{\left(\kern-.3em\left(\genfrac{}{}{0pt}{}{#1}{#2}\right)\kern-.3em\right)}}

\newcommand{\backvec}[1]{\reflectbox{$\vec{\reflectbox{\!$#1$}}$}}

\makeatletter
\renewcommand*\env@matrix[1][*\c@MaxMatrixCols c]{%
  \hskip -\arraycolsep
  \let\@ifnextchar\new@ifnextchar
  \array{#1}}
\makeatother

\newcommand*\Heq{\ensuremath{\overset{\kern2pt H}{=}}}

\def\Inv{\operatorname{inv}}

\newcommand{\Bkn}{\mathbf{B}_{k,n}}

\newcommand{\affdeomax}[2]{\operatorname{Deo}^{\max}_{#1,#2}}
\newcommand{\affdeo}[2]{\operatorname{Deo}_{#1,#2}}

\newcommand{\affdeop}[1]{\operatorname{Deo}_{#1}}
\newcommand{\deomax}[1]{\operatorname{Deo}^{\max}_{#1}}
\newcommand{\deo}[1]{\operatorname{Deo}_{#1}}

\newcommand{\antiaffdeo}[2]{\widetilde{\operatorname{Deo}_{#1,#2}}}

\newcommand{\antideomax}[1]{\widetilde{\operatorname{Deo}^{\max}_{#1}}}
\newcommand{\antideo}[1]{\widetilde{\operatorname{Deo}_{#1}}}
\newcommand{\Pkn}{\mathcal{P}_{k,n}}

\newcommand{\av}{\operatorname{av}}

\def\drawbox(#1,#2){
\draw[black!50,dashed] (#1-0.5,#2-0.5) rectangle (#1+0.5,#2+0.5);
}

\def\drawgrid#1#2#3#4{
\foreach\i in {#1,...,#3}{
\foreach \j in {#2,...,#4}{
\drawbox(\i,\j)
}
}
}

\def\crossing{\scalebox{0.4}{\begin{tikzpicture}[baseline=(ZUZU.base)]
\coordinate(ZUZU) at (0,-0.5);\drawgrid{0}{0}{0}{0}\draw[line width=3pt, rounded corners=12] (0.00,-0.50)--(0.00,0.50);\draw[line width=3pt, rounded corners=12] (-0.50,0.00)--(0.50,0.00);
\end{tikzpicture}}\xspace} 
\def\elbow{\scalebox{0.4}{\begin{tikzpicture}[baseline=(ZUZU.base),xscale=-1]
\coordinate(ZUZU) at (0,-0.5);\drawgrid{0}{0}{0}{0}\draw[line width=3pt, rounded corners=12] (0.00,-0.50)--(0.00,0.00)--(0.50,0.00);\draw[line width=3pt, rounded corners=12] (-0.50,0.00)--(0.00,0.00)--(0.00,0.50);
\end{tikzpicture}}\xspace}

\def\labelelbow{\scalebox{0.4}{\begin{tikzpicture}[baseline=(ZUZU.base),xscale=-1]
\coordinate(ZUZU) at (0,-0.5);
\drawgrid{0}{0}{0}{0}\draw[line width=3pt, rounded corners=12] (0.00,-0.50)--(0.00,0.00)--(0.50,0.00);
\draw[line width=3pt, rounded corners=12] (-0.50,0.00)--(0.00,0.00)--(0.00,0.50);
\node[anchor=west,inner sep=1pt] at (1,0) {$j_1$}; 
\node[anchor=north,inner sep=1pt] at (0,1) {$j_2$}; 
\end{tikzpicture}}\xspace}

\newcommand{\bottomcross}{
    \draw[line width=2pt,opacity=\op] (-0.5,0) -- (-0.5,-0.1) ..controls(-0.5,-0.5) and (0.5,-0.5).. (0.5,-0.9) -- (0.5,-1);
    \draw[line width=2pt,opacity=\op] (0.5,0) -- (0.5,-0.1) ..controls(0.5,-0.5) and (-0.5,-0.5).. (-0.5,-0.9) -- (-0.5,-1);
}

\newcommand{\crosscross}{
\draw[line width=2pt, opacity=\op] (-0.5,0) -- (-0.5,1);
\draw[line width=2pt, opacity=\op] (0.5,0) -- (0.5,1);
\draw[line width=2pt, opacity=\op] (-1,0.5) -- (0,0.5);
\draw[line width=2pt, opacity=\op] (0,0.5) -- (1,0.5);
}

\newcommand{\topcross}{
\draw[line width=2pt,opacity=\op] (-0.5,1) -- (-0.5,1.1) ..controls(-0.5,1.5) and (0.5,1.5).. (0.5,1.9) -- (0.5,2);
\draw[line width=2pt,opacity=\op] (0.5,1) -- (0.5,1.1) ..controls(0.5,1.5) and (-0.5,1.5).. (-0.5,1.9) -- (-0.5,2);
}

\newcommand{\elbowcross}{
\draw[line width=2pt, opacity=\op] (0,0.5) -- (1,0.5);
\draw[line width=2pt, opacity=\op] (0.5,0) -- (0.5,1);
\draw[line width=2pt,opacity=\op] (-0.5,0) to[out=90,in=0] (-1,0.5);
\draw[line width=2pt,opacity=\op] (-0.5,1) to[out=-90,in=180] (0,0.5);
}

\newcommand{\crosselbow}{
\draw[line width=2pt, opacity=\op] (0,0.5) -- (-1,0.5);
\draw[line width=2pt, opacity=\op] (-0.5,0) -- (-0.5,1);
\draw[line width=2pt,opacity=\op] (0.5,0) to[out=90,in=0] (0,0.5);
\draw[line width=2pt,opacity=\op] (0.5,1) to[out=-90,in=180] (1,0.5);
}

\newcommand{\elbowelbow}{
\draw[line width=2pt,opacity=\op] (0.5,0) to[out=90,in=0] (0,0.5);
\draw[line width=2pt,opacity=\op] (0.5,1) to[out=-90,in=180] (1,0.5);
\draw[line width=2pt,opacity=\op] (-0.5,0) to[out=90,in=0] (-1,0.5);
\draw[line width=2pt,opacity=\op] (-0.5,1) to[out=-90,in=180] (0,0.5);
}

\newcommand{\bottomelbow}{
\draw[line width=2pt,opacity=\op] (-0.5,0) -- (-0.5,-0.1) ..controls(-0.5, -0.25) and (-0.1,-0.25).. (-0.1, -0.5);
\draw[line width=2pt,opacity=\op]  (-0.5,-1) -- (-0.5, -0.9) ..controls(-0.5, -0.75) and (-0.1,-0.75).. (-0.1, -0.5);

\draw[line width=2pt,opacity=\op] (0.5,0) -- (0.5,-0.1) ..controls(0.5, -0.25) and (0.1,-0.25).. (0.1, -0.5);
\draw[line width=2pt,opacity=\op]  (0.5,-1) -- (0.5, -0.9) ..controls(0.5, -0.75) and (0.1,-0.75).. (0.1, -0.5);
}

\newcommand{\topelbow}{
\draw[line width=2pt,opacity=\op] (-0.5,2) -- (-0.5,1.9) ..controls(-0.5, 1.75) and (-0.1,1.75).. (-0.1, 1.5);
\draw[line width=2pt,opacity=\op]  (-0.5,1) -- (-0.5, 1.1) ..controls(-0.5, 1.25) and (-0.1,1.25).. (-0.1, 1.5);

\draw[line width=2pt,opacity=\op] (0.5,2) -- (0.5,1.9) ..controls(0.5, 1.75) and (0.1,1.75).. (0.1, 1.5);
\draw[line width=2pt,opacity=\op]  (0.5,1) -- (0.5, 1.1) ..controls(0.5, 1.25) and (0.1,1.25).. (0.1, 1.5);
}

\def\deog#1#2#3#4#5{
\scalebox{\sclbx}{
\begin{tikzpicture}[xscale=0.5,yscale=0.5,baseline ={(0,.67)}]
\foreach \x/\y in {#1}{
  \draw[line width=1pt] (\x,\y-0.035) -- (\x,\y+1.035);
}
\foreach \x/\y in {#2}{
  \draw[line width=1pt] (\x,\y) -- (\x+1,\y);
}
\foreach \x/\y/\z in {#3}{
  \node[scale=0.50,inner sep=1pt] at (\x,\y) {$\z$};
}
\foreach \x/\y in {#4}{
  \draw[line width=0.5pt,dashed,opacity=\gridop] (\x,\y) -- (\x+1,\y);
  \draw[line width=0.5pt,dashed,opacity=\gridop] (\x+1,\y) -- (\x+1,\y+1);
  \draw[line width=2pt,opacity=\op] (\x+0.5,\y) to[out=90,in=0] (\x,\y+0.5);
  \draw[line width=2pt,opacity=\op] (\x+0.5,\y+1) to[out=-90,in=180] (\x+1,\y+0.5);
}
\foreach \x/\y in {#5}{
  \draw[line width=0.5pt,dashed,opacity=\gridop] (\x,\y) -- (\x+1,\y);
  \draw[line width=0.5pt,dashed,opacity=\gridop] (\x+1,\y) -- (\x+1,\y+1);
  \draw[line width=2pt,opacity=\op] (\x+0.5,\y) -- (\x+0.5,\y+1);
  \draw[line width=2pt,opacity=\op] (\x,\y+0.5) -- (\x+1,\y+0.5);
}
\end{tikzpicture}
}
}

\def\deogcolored#1#2#3#4#5#6{
\scalebox{\sclbx}{
\begin{tikzpicture}[xscale=0.5,yscale=0.5,baseline ={(0,.66)}]
\foreach \x/\y in {#1}{
  \draw[line width=1pt] (\x,\y-0.035) -- (\x,\y+1.035);
}
\foreach \x/\y in {#2}{
  \draw[line width=1pt] (\x,\y) -- (\x+1,\y);
}
\foreach \x/\y/\z in {#3}{
  \node[scale=0.50,inner sep=1pt] at (\x,\y) {$\z$};
}
\foreach \x/\y in {#4}{
  \draw[line width=0.5pt,dashed,opacity=\gridop] (\x,\y) -- (\x+1,\y);
  \draw[line width=0.5pt,dashed,opacity=\gridop] (\x+1,\y) -- (\x+1,\y+1);
  \draw[line width=2pt,opacity=\op] (\x+0.5,\y) to[out=90,in=0] (\x,\y+0.5);
  \draw[line width=2pt,opacity=\op] (\x+0.5,\y+1) to[out=-90,in=180] (\x+1,\y+0.5);
}
\foreach \x/\y in {#5}{
  \draw[line width=0.5pt,dashed,opacity=\gridop] (\x,\y) -- (\x+1,\y);
  \draw[line width=0.5pt,dashed,opacity=\gridop] (\x+1,\y) -- (\x+1,\y+1);
  \draw[line width=2pt,opacity=\op] (\x+0.5,\y) -- (\x+0.5,\y+1);
  \draw[line width=2pt,opacity=\op] (\x,\y+0.5) -- (\x+1,\y+0.5);
}
\foreach \x/\y in {#6}{
    \draw[line width=0pt,fill=red!40,opacity=\gridop] (\x,\y) rectangle (1+\x,1+\y);
}
\end{tikzpicture}
}
}

\def\bigaffdeog#1#2#3#4#5#6#7{
\scalebox{\sclbxbigdeo}{
\begin{tikzpicture}[xscale=0.45,yscale=0.45,baseline={(0,0.5)}]
\foreach \x/\y in {#1}{
  \draw[line width=1pt] (\x,\y-0.035) -- (\x,\y+1.035);
}
\foreach \x/\y in {#2}{
  \draw[line width=1pt] (\x,\y) -- (\x+1,\y);
}
\foreach \x/\y/\z in {#3}{
  \node[scale=0.40,inner sep=1pt] at (\x,\y) {$\z$};
}
\foreach \x/\y in {#4}{
  \draw[line width=0.5pt,dashed,opacity=\gridop] (\x,\y) -- (\x+1,\y);
  \draw[line width=0.5pt,dashed,opacity=\gridop] (\x+1,\y) -- (\x+1,\y+1);
  \draw[line width=2pt,opacity=\op] (\x+0.5,\y) to[out=90,in=0] (\x,\y+0.5);
  \draw[line width=2pt,opacity=\op] (\x+0.5,\y+1) to[out=-90,in=180] (\x+1,\y+0.5);
}
\foreach \x/\y in {#5}{
  \draw[line width=0.5pt,dashed,opacity=\gridop] (\x,\y) -- (\x+1,\y);
  \draw[line width=0.5pt,dashed,opacity=\gridop] (\x+1,\y) -- (\x+1,\y+1);
  \draw[line width=2pt,opacity=\op] (\x+0.5,\y) -- (\x+0.5,\y+1);
  \draw[line width=2pt,opacity=\op] (\x,\y+0.5) -- (\x+1,\y+0.5);
}
\foreach \x/\y in {#6}{
    \draw[line width=0pt,fill=blue!40,opacity=\gridop] (\x,\y) rectangle (1+\x,1+\y);
  \draw[line width=0.5pt,dashed,opacity=\gridop] (\x,\y) -- (\x+1,\y);
  \draw[line width=0.5pt,dashed,opacity=\gridop] (\x+1,\y) -- (\x+1,\y+1);
  \draw[line width=2pt,opacity=\op] (\x+0.5,\y) to[out=90,in=0] (\x,\y+0.5);
  \draw[line width=2pt,opacity=\op] (\x+0.5,\y+1) to[out=-90,in=180] (\x+1,\y+0.5);
}
\foreach \x/\y in {#7}{
 \draw[line width=0pt,fill=blue!40,opacity=\gridop] (\x,\y) rectangle (1+\x,1+\y);
  \draw[line width=0.5pt,dashed,opacity=\gridop] (\x,\y) -- (\x+1,\y);
  \draw[line width=0.5pt,dashed,opacity=\gridop] (\x+1,\y) -- (\x+1,\y+1);
  \draw[line width=2pt,opacity=\op] (\x+0.5,\y) -- (\x+0.5,\y+1);
  \draw[line width=2pt,opacity=\op] (\x,\y+0.5) -- (\x+1,\y+0.5);
}
\end{tikzpicture}
}
}

\def\bigaffdeogcolored#1#2#3#4#5#6#7#8#9{
\scalebox{\sclbxbigdeo}{
\begin{tikzpicture}[xscale=0.45,yscale=0.45,baseline={(0,0.5)}]
\foreach \x/\y in {#1}{
  \draw[line width=1pt] (\x,\y-0.035) -- (\x,\y+1.035);
}
\foreach \x/\y in {#2}{
  \draw[line width=1pt] (\x,\y) -- (\x+1,\y);
}
\foreach \x/\y/\z in {#3}{
  \node[scale=0.40,inner sep=1pt] at (\x,\y) {$\z$};
}
\foreach \x/\y in {#4}{
  \draw[line width=0.5pt,dashed,opacity=\gridop] (\x,\y) -- (\x+1,\y);
  \draw[line width=0.5pt,dashed,opacity=\gridop] (\x+1,\y) -- (\x+1,\y+1);
  \draw[line width=2pt,opacity=\op] (\x+0.5,\y) to[out=90,in=0] (\x,\y+0.5);
  \draw[line width=2pt,opacity=\op] (\x+0.5,\y+1) to[out=-90,in=180] (\x+1,\y+0.5);
}
\foreach \x/\y in {#5}{
  \draw[line width=0.5pt,dashed,opacity=\gridop] (\x,\y) -- (\x+1,\y);
  \draw[line width=0.5pt,dashed,opacity=\gridop] (\x+1,\y) -- (\x+1,\y+1);
  \draw[line width=2pt,opacity=\op] (\x+0.5,\y) -- (\x+0.5,\y+1);
  \draw[line width=2pt,opacity=\op] (\x,\y+0.5) -- (\x+1,\y+0.5);
}
\foreach \x/\y in {#6}{
  \draw[line width=0.5pt,dashed,opacity=\gridop] (\x,\y) -- (\x+1,\y);
  \draw[line width=0.5pt,dashed,opacity=\gridop] (\x+1,\y) -- (\x+1,\y+1);
  \draw[red,line width=2pt,opacity=\op] (\x+0.5,\y) -- (\x+0.5,\y+1);
  \draw[red,line width=2pt,opacity=\op] (\x,\y+0.5) -- (\x+1,\y+0.5);
}
\foreach \x/\y in {#7}{
  \draw[line width=0.5pt,dashed,opacity=\gridop] (\x,\y) -- (\x+1,\y);
  \draw[line width=0.5pt,dashed,opacity=\gridop] (\x+1,\y) -- (\x+1,\y+1);
  \draw[red,line width=2pt,opacity=\op] (\x+0.5,\y) to[out=90,in=0] (\x,\y+0.5);
  \draw[line width=2pt,opacity=\op] (\x+0.5,\y+1) to[out=-90,in=180] (\x+1,\y+0.5);
}
\foreach \x/\y in {#8}{
  \draw[line width=0.5pt,dashed,opacity=\gridop] (\x,\y) -- (\x+1,\y);
  \draw[line width=0.5pt,dashed,opacity=\gridop] (\x+1,\y) -- (\x+1,\y+1);
  \draw[line width=2pt,opacity=\op] (\x+0.5,\y) to[out=90,in=0] (\x,\y+0.5);
  \draw[red,line width=2pt,opacity=\op] (\x+0.5,\y+1) to[out=-90,in=180] (\x+1,\y+0.5);
}
\foreach \x/\y in {#9}{
  \draw[line width=0pt,fill=red!40,opacity=\gridop] (\x,\y) rectangle (1+\x,1+\y);
  \draw[red,line width=2pt,opacity=\op] (\x+0.5,\y) to[out=90,in=0] (\x,\y+0.5);
  \draw[red,line width=2pt,opacity=\op] (\x+0.5,\y+1) to[out=-90,in=180] (\x+1,\y+0.5);
}
\end{tikzpicture}
}
}

\title{The Combinatorics of Affine Deodhar Diagrams}

\author{Thomas C. Martinez}
\address{Department of Mathematics, University of California, Los Angeles, 520 Portola Plaza,
Los Angeles, CA 90025, USA}
\email{\href{mailto:thomasmart@ucla.edu}{thomasmart@ucla.edu}}

\thanks{}

\date\today

\begin{document}

\makeatletter
\@namedef{subjclassname@2020}{%
  \textup{2020} Mathematics Subject Classification}
\makeatother



\begin{abstract}
Deodhar diagrams give a combinatorial way to compute point counts of open positroid varieties over finite fields. We introduce affine Deodhar diagrams, which extend this construction to affine patches of open positroid varieties. These diagrams are indexed by a bounded affine permutation together with a lattice path, and this extra flexibility makes them especially useful for recursive bijections.

Motivated by connections with Dyck paths, open positroid varieties, and their cluster structure, we construct bijections between several classes of affine Deograms. These bijections give combinatorial proofs of point-count identities that were previously known from geometric isomorphisms.
\end{abstract}
\maketitle
\setcounter{tocdepth}{1}
\tableofcontents

\section{Introduction}\label{sec:intro}

\textbf{Open positroid varieties}, $\Pio_f\subset\Grkn$, defined by Knutson--Lam--Speyer \cite{KLS13}, are smooth irreducible subvarieties of the Grassmannian $\operatorname{Gr}(k,n)$ isomorphic to certain Richardson varieties in the flag variety. Since their inception, positroid varieties have been studied extensively due to their deep connections with total positivity \cite{kodwil,marsh}, cluster algebras \cite{galashin2023positroid,muller,muller2017twist,serhiyenko2019cluster}, and Catalan combinatorics \cite{positroid,gallam2024}, to name a few. 

Open positroid varieties $\Pio_f$ are indexed by $(k,n)$-bounded affine permutations $f:\Z\to\Z$, which satisfy $\sum_{i=1}^n (f(i)-i) = kn$ and for all $i\in\Z$, $f(i+n)=f(i)+n$ and $i\leq f(i)\leq i+n$. We write $\Bkn$ for the set of such permutations. We also write $\overline f\in S_n$ for the finite permutation obtained by reducing values modulo $n$, and $c(\overline f)$ for its number of cycles. The length $\ell(f)$ is the Coxeter length in the affine Coxeter group of type $\Anaff$, after the standard identification of the average-$k$ component with the affine Weyl group; see Section~\ref{sec:background}.

For $I\in\binom{[n]}k$, let $\Delta_I$ be the corresponding Pl\"ucker coordinate. We consider the affine patch
\[
    \Pio_{f,I}:=\Pio_f\cap\{\Delta_I\neq0\}.
\]
By a theorem of Snider \cite{snider}, when $\Pio_{f,I}$ is non-empty, it is isomorphic to an open Richardson variety in the affine flag variety. The index set $I$ may equivalently be encoded by a lattice path $P$ from $(0,0)$ to $(n-k,k)$.

We are now ready to state the main object of the paper. For each lattice path $P$, periodically extend it to an infinite up/right lattice path in $\Z^2$, invariant under translation by $(n-k,k)$. An \textbf{$(f,P)$-affine Deogram} is a periodic filling of the diagonal strip between $P$ and $P+(0,k)$ by crossings \crossing and elbows \elbow such that:
\begin{enumerate}[label=(\alph*)]
\item \label{cond1} the resulting strand permutation equals $f$. That is, the strands connect the index $i$ on the northern boundary to the index $f(i)$ on the southern boundary, and
\item \label{disting} counting from the northern boundary, if a pair of strands has crossed an odd number of times, they cannot form an elbow until they cross again. 
\end{enumerate}
See Figure~\ref{fig:bigaffdeo} for an example. Condition \ref{disting} is Deodhar's distinguished condition \cite{deodhar}.

\begin{figure}[t]
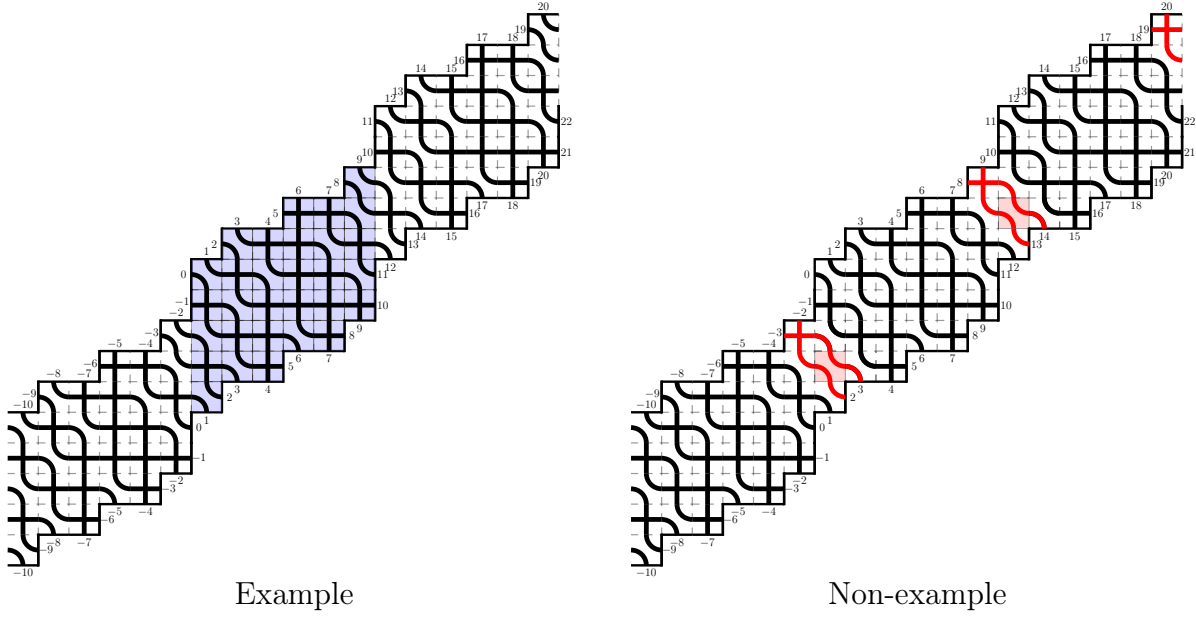

\centering
 
\begin{tabular}{cc}
\bigaffdeog{-5/-5,7/5,1/5,-5/0,7/10,3/1,-3/-4,9/6,3/6,-3/1,9/11,5/2,-1/-3,11/7,5/7,-1/2,11/12,6/3,0/-2,12/8,6/8,0/3,6/4,0/-1,12/9,6/9,0/4,1/0,-5/-5,7/5,1/5,-5/0,7/10,3/1,-3/-4,9/6,3/6,-3/1,9/11,5/2,-1/-3,11/7,5/7,-1/2,11/12,6/3,0/-2,12/8,6/8,0/3,6/4,0/-1,12/9,6/9,0/4}{0/0,-6/-5,6/5,0/5,-6/0,6/10,1/1,-5/-4,7/6,1/6,-5/1,7/11,2/1,-4/-4,8/6,2/6,-4/1,8/11,3/2,-3/-3,9/7,3/7,-3/2,9/12,4/2,-2/-3,10/7,4/7,-2/2,10/12,5/3,-1/-2,11/8,5/8,-1/3,11/13}{0.500000000000000/-0.250000000000000/1,0.500000000000000/5.25000000000000/1,-5.50000000000000/-5.25000000000000/-10,-5.50000000000000/0.250000000000000/-10,6.50000000000000/4.75000000000000/12,6.50000000000000/10.2500000000000/12,1.25000000000000/0.500000000000000/2,0.750000000000000/5.50000000000000/2,7.25000000000000/5.50000000000000/13,6.75000000000000/10.5000000000000/13,1.50000000000000/0.750000000000000/3,1.50000000000000/6.25000000000000/3,-4.50000000000000/-4.25000000000000/-8,-4.50000000000000/1.25000000000000/-8,7.50000000000000/5.75000000000000/14,7.50000000000000/11.2500000000000/14,2.50000000000000/0.750000000000000/4,2.50000000000000/6.25000000000000/4,-3.50000000000000/-4.25000000000000/-7,-3.50000000000000/1.25000000000000/-7,8.50000000000000/5.75000000000000/15,8.50000000000000/11.2500000000000/15,3.25000000000000/1.50000000000000/5,2.75000000000000/6.50000000000000/5,-2.75000000000000/-3.50000000000000/-6,-3.3000000000000/1.60000000000000/-6,-4.75000000000000/-4.50000000000000/-9,-5.3000000000000/0.60000000000000/-9,9.25000000000000/6.50000000000000/16,8.75000000000000/11.5000000000000/16,3.50000000000000/1.75000000000000/6,3.50000000000000/7.25000000000000/6,-2.50000000000000/-3.25000000000000/-5,-2.50000000000000/2.25000000000000/-5,9.50000000000000/6.75000000000000/17,9.50000000000000/12.2500000000000/17,4.50000000000000/1.75000000000000/7,4.50000000000000/7.25000000000000/7,-1.50000000000000/-3.25000000000000/-4,-1.50000000000000/2.25000000000000/-4,10.5000000000000/6.75000000000000/18,10.5000000000000/12.2500000000000/18,5.25000000000000/2.50000000000000/8,4.75000000000000/7.50000000000000/8,-0.750000000000000/-2.50000000000000/-3,-1.3000000000000/2.60000000000000/-3,11.2500000000000/7.50000000000000/19,10.7500000000000/12.5000000000000/19,5.50000000000000/2.75000000000000/9,5.50000000000000/8.25000000000000/9,-0.500000000000000/-2.25000000000000/-2,-0.500000000000000/3.25000000000000/-2,11.5000000000000/7.75000000000000/20,11.5000000000000/13.2500000000000/20,6.25000000000000/3.50000000000000/10,5.75000000000000/8.50000000000000/10,0.250000000000000/-1.50000000000000/-1,-0.30000000000000/3.60000000000000/-1,12.2500000000000/8.50000000000000/21,6.25000000000000/4.50000000000000/11,5.75000000000000/9.50000000000000/11,0.250000000000000/-0.500000000000000/0,-0.250000000000000/4.50000000000000/0,12.2500000000000/9.50000000000000/22}{-6/-1,6/9,-4/-1,8/9,-1/-1,11/9,-5/-2,7/8,-6/-3,6/7,-3/-3,9/7,-5/-4,7/6,-6/-5,6/5,11/11,11/12,10/10,-1/1,-1/2,-2/0,-5/0,7/10}{-2/-2,10/8,-6/-4,6/6,-5/-1,7/9,-3/-1,9/9,-2/-1,-6/-2,6/8,-4/-2,8/8,-3/-2,9/8,-1/-2,11/8,-5/-3,7/7,-4/-3,8/7,-2/-3,10/7,-4/-4,8/6,-3/1,9/11,-2/1,10/11,-4/0,8/10,-3/0,9/10,10/9,-1/0,11/10}{0/4,2/4,5/4,1/3,0/2,3/2,1/1,0/0,4/5,5/6,5/7,1/5}{0/1,1/4,3/4,4/4,0/3,2/3,3/3,5/3,1/2,2/2,4/2,2/1,3/6,4/6,2/5,3/5,5/5,4/3}
 
&
 
\bigaffdeogcolored{1/0,-5/-5,7/5,1/5,-5/0,7/10,3/1,-3/-4,9/6,3/6,-3/1,9/11,5/2,-1/-3,11/7,5/7,-1/2,11/12,6/3,0/-2,12/8,6/8,0/3,6/4,0/-1,12/9,6/9,0/4,1/0,-5/-5,7/5,1/5,-5/0,7/10,3/1,-3/-4,9/6,3/6,-3/1,9/11,5/2,-1/-3,11/7,5/7,-1/2,11/12,6/3,0/-2,12/8,6/8,0/3,6/4,0/-1,12/9,6/9}{0/0,-6/-5,6/5,0/5,-6/0,6/10,1/1,-5/-4,7/6,1/6,-5/1,7/11,2/1,-4/-4,8/6,2/6,-4/1,8/11,3/2,-3/-3,9/7,3/7,-3/2,9/12,4/2,-2/-3,10/7,4/7,-2/2,10/12,5/3,-1/-2,11/8,5/8,-1/3,11/13}{0.500000000000000/-0.250000000000000/1,0.500000000000000/5.25000000000000/1,-5.50000000000000/-5.25000000000000/-10,-5.50000000000000/0.250000000000000/-10,6.50000000000000/4.75000000000000/12,6.50000000000000/10.2500000000000/12,1.25000000000000/0.500000000000000/2,0.750000000000000/5.50000000000000/2,7.25000000000000/5.50000000000000/13,6.75000000000000/10.5000000000000/13,1.50000000000000/0.750000000000000/3,1.50000000000000/6.25000000000000/3,-4.50000000000000/-4.25000000000000/-8,-4.50000000000000/1.25000000000000/-8,7.50000000000000/5.75000000000000/14,7.50000000000000/11.2500000000000/14,2.50000000000000/0.750000000000000/4,2.50000000000000/6.25000000000000/4,-3.50000000000000/-4.25000000000000/-7,-3.50000000000000/1.25000000000000/-7,8.50000000000000/5.75000000000000/15,8.50000000000000/11.2500000000000/15,3.25000000000000/1.50000000000000/5,2.75000000000000/6.50000000000000/5,-2.75000000000000/-3.50000000000000/-6,-3.3000000000000/1.60000000000000/-6,-4.75000000000000/-4.50000000000000/-9,-5.3000000000000/0.60000000000000/-9,9.25000000000000/6.50000000000000/16,8.75000000000000/11.5000000000000/16,3.50000000000000/1.75000000000000/6,3.50000000000000/7.25000000000000/6,-2.50000000000000/-3.25000000000000/-5,-2.50000000000000/2.25000000000000/-5,9.50000000000000/6.75000000000000/17,9.50000000000000/12.2500000000000/17,4.50000000000000/1.75000000000000/7,4.50000000000000/7.25000000000000/7,-1.50000000000000/-3.25000000000000/-4,-1.50000000000000/2.25000000000000/-4,10.5000000000000/6.75000000000000/18,10.5000000000000/12.2500000000000/18,5.25000000000000/2.50000000000000/8,4.75000000000000/7.50000000000000/8,-0.750000000000000/-2.50000000000000/-3,-1.3000000000000/2.60000000000000/-3,11.2500000000000/7.50000000000000/19,10.7500000000000/12.5000000000000/19,5.50000000000000/2.75000000000000/9,5.50000000000000/8.25000000000000/9,-0.500000000000000/-2.25000000000000/-2,-0.500000000000000/3.25000000000000/-2,11.5000000000000/7.75000000000000/20,11.5000000000000/13.2500000000000/20,6.25000000000000/3.50000000000000/10,5.75000000000000/8.50000000000000/10,0.250000000000000/-1.50000000000000/-1,-0.30000000000000/3.60000000000000/-1,12.2500000000000/8.50000000000000/21,6.25000000000000/4.50000000000000/11,5.75000000000000/9.50000000000000/11,0.250000000000000/-0.500000000000000/0,-0.250000000000000/4.50000000000000/0,12.2500000000000/9.50000000000000/22}{1/5,-6/-1,6/9,-4/-1,8/9,-1/-1,11/9,-5/-2,7/8,-6/-3,-3/-3,9/7,-5/-4,7/6,-6/-5,10/10,-2/0,0/4,2/4,5/4,1/3,0/2,3/2,1/1,4/5,-5/0,7/10}{-2/-2,4/3,10/8,-6/-4,-5/-1,7/9,-3/-1,9/9,-2/-1,-6/-2,6/8,-4/-2,8/8,-3/-2,9/8,-1/-2,11/8,-5/-3,7/7,-4/-3,8/7,-2/-3,10/7,-4/-4,8/6,-3/1,9/11,-2/1,10/11,-4/0,8/10,-3/0,9/10,10/9,-1/0,11/10,1/4,3/4,4/4,2/3,3/3,1/2,2/2,4/2,2/1,3/6,4/6,2/5,3/5,5/5,5/3,0/3}{5/7,11/12,-1/2}{0/2,1/1,6/7,7/6}{5/6,6/5,11/11,-1/1,0/0}{6/6,0/1}
 
\\
Example & Non-example
\end{tabular}
 
\caption{Left: an $(f,P)$-affine Deogram for $P=\text{RURRURRURUU}$ and $f=[4,10,7,9,11,8,12,14,13,16,17]$. The blue region is a representative region; in later figures we display only this region. Right: a periodic filling which fails condition~\ref{disting} at the highlighted box.}
\label{fig:bigaffdeo}
\end{figure}

We also consider \textbf{anti-distinguished} affine Deograms, where condition~\ref{disting} is imposed by reading from the southern boundary instead. Specializing $P$ to a canonical lattice path associated to $f$ recovers usual $f$-\textbf{Deograms}, as further explained in Section~\ref{sec:affdeo}. The Deodhar decomposition \cite{deodhar} expresses the point count of $\Pio_f$ over $\F_q$ as a weighted sum over the set $\deo{f}$ of ordinary $f$-Deograms. This allows us to interpret geometric isomorphisms between open positroid varieties as equalities between different sets of Deograms.

\subsection{Main Results} Let $f_{k,n}:\Z\to\Z$ denote the affine permutation $f_{k,n}(i)=i+k$. A (affine) Deogram is \textbf{maximal} if it has precisely $n-c(\overline{f})$ many elbows (in a representative region). Galashin--Lam~\cite{positroid} showed that, when $\gcd(k,n)=1$, the set of maximal $f_{k,n}$-Deograms $\deomax{f_{k,n}}$ has cardinality equal to the rational Catalan number $C_{k,n-k}=\frac1n\binom nk$. This is also the cardinality of the set $\Dyck_{k,n}$ of rational Dyck paths in a $k\times(n-k)$ rectangle; see Figure~\ref{fig:deodyckex}.

The following identities all admit known geometric proofs.

\begin{figure}[t]
\centering
\begin{tabular}{cccc}  
\setlength{\tabcolsep}{0pt}
\deog{0/0,0/1,0/2,5/0,5/1,5/2,0/0,0/1,0/2,5/0,5/1,5/2}{0/3,1/3,2/3,3/3,4/3,0/0,1/0,2/0,3/0,4/0}{-0.350000000000000/0.500000000000000/1,-0.350000000000000/1.50000000000000/2,-0.350000000000000/2.50000000000000/3,0.500000000000000/3.35000000000000/4,1.50000000000000/3.35000000000000/5,2.50000000000000/3.35000000000000/6,3.50000000000000/3.35000000000000/7,4.50000000000000/3.35000000000000/8,0.500000000000000/-0.350000000000000/1,1.50000000000000/-0.350000000000000/2,2.50000000000000/-0.350000000000000/3,3.50000000000000/-0.350000000000000/4,4.50000000000000/-0.350000000000000/5,5.35000000000000/0.500000000000000/6,5.35000000000000/1.50000000000000/7,5.35000000000000/2.50000000000000/8}{2/2,4/2,1/1,3/1,0/0,3/0,4/0}{0/2,1/2,3/2,0/1,2/1,4/1,1/0,2/0}&\deog{0/0,0/1,0/2,5/0,5/1,5/2,0/0,0/1,0/2,5/0,5/1,5/2}{0/3,1/3,2/3,3/3,4/3,0/0,1/0,2/0,3/0,4/0}{-0.350000000000000/0.500000000000000/1,-0.350000000000000/1.50000000000000/2,-0.350000000000000/2.50000000000000/3,0.500000000000000/3.35000000000000/4,1.50000000000000/3.35000000000000/5,2.50000000000000/3.35000000000000/6,3.50000000000000/3.35000000000000/7,4.50000000000000/3.35000000000000/8,0.500000000000000/-0.350000000000000/1,1.50000000000000/-0.350000000000000/2,2.50000000000000/-0.350000000000000/3,3.50000000000000/-0.350000000000000/4,4.50000000000000/-0.350000000000000/5,5.35000000000000/0.500000000000000/6,5.35000000000000/1.50000000000000/7,5.35000000000000/2.50000000000000/8}{2/2,4/2,0/1,3/1,0/0,1/0,4/0}{0/2,1/2,3/2,1/1,2/1,4/1,2/0,3/0}&\deog{0/0,0/1,0/2,5/0,5/1,5/2,0/0,0/1,0/2,5/0,5/1,5/2}{0/3,1/3,2/3,3/3,4/3,0/0,1/0,2/0,3/0,4/0}{-0.350000000000000/0.500000000000000/1,-0.350000000000000/1.50000000000000/2,-0.350000000000000/2.50000000000000/3,0.500000000000000/3.35000000000000/4,1.50000000000000/3.35000000000000/5,2.50000000000000/3.35000000000000/6,3.50000000000000/3.35000000000000/7,4.50000000000000/3.35000000000000/8,0.500000000000000/-0.350000000000000/1,1.50000000000000/-0.350000000000000/2,2.50000000000000/-0.350000000000000/3,3.50000000000000/-0.350000000000000/4,4.50000000000000/-0.350000000000000/5,5.35000000000000/0.500000000000000/6,5.35000000000000/1.50000000000000/7,5.35000000000000/2.50000000000000/8}{2/2,3/2,4/2,1/1,4/1,0/0,3/0}{0/2,1/2,0/1,2/1,3/1,1/0,2/0,4/0}&\deog{0/0,0/1,0/2,5/0,5/1,5/2,0/0,0/1,0/2,5/0,5/1,5/2}{0/3,1/3,2/3,3/3,4/3,0/0,1/0,2/0,3/0,4/0}{-0.350000000000000/0.500000000000000/1,-0.350000000000000/1.50000000000000/2,-0.350000000000000/2.50000000000000/3,0.500000000000000/3.35000000000000/4,1.50000000000000/3.35000000000000/5,2.50000000000000/3.35000000000000/6,3.50000000000000/3.35000000000000/7,4.50000000000000/3.35000000000000/8,0.500000000000000/-0.350000000000000/1,1.50000000000000/-0.350000000000000/2,2.50000000000000/-0.350000000000000/3,3.50000000000000/-0.350000000000000/4,4.50000000000000/-0.350000000000000/5,5.35000000000000/0.500000000000000/6,5.35000000000000/1.50000000000000/7,5.35000000000000/2.50000000000000/8}{1/2,4/2,1/1,2/1,3/1,0/0,3/0}{0/2,2/2,3/2,0/1,4/1,1/0,2/0,4/0}\\
\end{tabular}
 
\begin{tabular}{ccc}
\deog{0/0,0/1,0/2,5/0,5/1,5/2,0/0,0/1,0/2,5/0,5/1,5/2}{0/3,1/3,2/3,3/3,4/3,0/0,1/0,2/0,3/0,4/0}{-0.350000000000000/0.500000000000000/1,-0.350000000000000/1.50000000000000/2,-0.350000000000000/2.50000000000000/3,0.500000000000000/3.35000000000000/4,1.50000000000000/3.35000000000000/5,2.50000000000000/3.35000000000000/6,3.50000000000000/3.35000000000000/7,4.50000000000000/3.35000000000000/8,0.500000000000000/-0.350000000000000/1,1.50000000000000/-0.350000000000000/2,2.50000000000000/-0.350000000000000/3,3.50000000000000/-0.350000000000000/4,4.50000000000000/-0.350000000000000/5,5.35000000000000/0.500000000000000/6,5.35000000000000/1.50000000000000/7,5.35000000000000/2.50000000000000/8}{1/2,4/2,0/1,3/1,0/0,1/0,2/0}{0/2,2/2,3/2,1/1,2/1,4/1,3/0,4/0}&\deog{0/0,0/1,0/2,5/0,5/1,5/2,0/0,0/1,0/2,5/0,5/1,5/2}{0/3,1/3,2/3,3/3,4/3,0/0,1/0,2/0,3/0,4/0}{-0.350000000000000/0.500000000000000/1,-0.350000000000000/1.50000000000000/2,-0.350000000000000/2.50000000000000/3,0.500000000000000/3.35000000000000/4,1.50000000000000/3.35000000000000/5,2.50000000000000/3.35000000000000/6,3.50000000000000/3.35000000000000/7,4.50000000000000/3.35000000000000/8,0.500000000000000/-0.350000000000000/1,1.50000000000000/-0.350000000000000/2,2.50000000000000/-0.350000000000000/3,3.50000000000000/-0.350000000000000/4,4.50000000000000/-0.350000000000000/5,5.35000000000000/0.500000000000000/6,5.35000000000000/1.50000000000000/7,5.35000000000000/2.50000000000000/8}{0/2,3/2,4/2,1/1,4/1,0/0,2/0}{1/2,2/2,0/1,2/1,3/1,1/0,3/0,4/0}&\deog{0/0,0/1,0/2,5/0,5/1,5/2,0/0,0/1,0/2,5/0,5/1,5/2}{0/3,1/3,2/3,3/3,4/3,0/0,1/0,2/0,3/0,4/0}{-0.350000000000000/0.500000000000000/1,-0.350000000000000/1.50000000000000/2,-0.350000000000000/2.50000000000000/3,0.500000000000000/3.35000000000000/4,1.50000000000000/3.35000000000000/5,2.50000000000000/3.35000000000000/6,3.50000000000000/3.35000000000000/7,4.50000000000000/3.35000000000000/8,0.500000000000000/-0.350000000000000/1,1.50000000000000/-0.350000000000000/2,2.50000000000000/-0.350000000000000/3,3.50000000000000/-0.350000000000000/4,4.50000000000000/-0.350000000000000/5,5.35000000000000/0.500000000000000/6,5.35000000000000/1.50000000000000/7,5.35000000000000/2.50000000000000/8}{0/2,1/2,4/2,1/1,3/1,0/0,2/0}{2/2,3/2,0/1,2/1,4/1,1/0,3/0,4/0}\\
\end{tabular}
\vspace{2mm}
 
\begin{tabular}{cccc}
\scalebox{\sclbx}{
\begin{tikzpicture}[xscale=0.5,yscale=0.5]
 
    \draw[black] (0,0) grid (5,3);
 
    \draw[red,line width=1.5pt] (0,0) -- (5,3);
    
    \draw[orange, thick,line width=3pt](0,0)--(0,3)--(5,3);
    
\end{tikzpicture}
}
&
\scalebox{\sclbx}{
\begin{tikzpicture}[xscale=0.5,yscale=0.5]
 
    \draw[black] (0,0) grid (5,3);
 
    \draw[red,line width=1.5pt] (0,0) -- (5,3);
    
    \draw[orange, thick,line width=3pt](0,0)--(0,2)--(1,2)--(1,3)--(5,3);
    
\end{tikzpicture}
}
&
\scalebox{\sclbx}{
\begin{tikzpicture}[xscale=0.5,yscale=0.5]
 
    \draw[black] (0,0) grid (5,3);
 
    \draw[red,line width=1.5pt] (0,0) -- (5,3);
    
    \draw[orange, thick,line width=3pt](0,0)--(0,2)--(2,2)--(2,3)--(5,3);
    
\end{tikzpicture}
}
 
&
\scalebox{\sclbx}{
\begin{tikzpicture}[xscale=0.5,yscale=0.5]
 
    \draw[black] (0,0) grid (5,3);
 
    \draw[red,line width=1.5pt] (0,0) -- (5,3);
    
    \draw[orange, thick,line width=3pt](0,0)--(0,2)--(3,2)--(3,3)--(5,3);
    
\end{tikzpicture}
}
\end{tabular}
 
\begin{tabular}{ccc}
\scalebox{\sclbx}{
\begin{tikzpicture}[xscale=0.5,yscale=0.5]
 
    \draw[black] (0,0) grid (5,3);
 
    \draw[red,line width=1.5pt] (0,0) -- (5,3);
    
    \draw[orange, thick,line width=3pt](0,0)--(0,1)--(1,1)--(1,3)--(5,3);
    
\end{tikzpicture}
}
&
\scalebox{\sclbx}{
\begin{tikzpicture}[xscale=0.5,yscale=0.5]
 
    \draw[black] (0,0) grid (5,3);
 
    \draw[red,line width=1.5pt] (0,0) -- (5,3);
    
    \draw[orange, thick,line width=3pt](0,0)--(0,1)--(1,1)--(1,2)--(2,2)--(2,3)--(5,3);
    
\end{tikzpicture}
}
&
\scalebox{\sclbx}{
\begin{tikzpicture}[xscale=0.5,yscale=0.5]
 
    \draw[black] (0,0) grid (5,3);
 
    \draw[red,line width=1.5pt] (0,0) -- (5,3);
    
    \draw[orange, thick,line width=3pt](0,0)--(0,1)--(1,1)--(1,2)--(3,2)--(3,3)--(5,3);
    
\end{tikzpicture}
}
\end{tabular}
\caption{The sets $\Dyck_{k,n}$ and $\deomax{f_{k,n}}$ for $k=3$, $n=8$.}
\label{fig:deodyckex}
\end{figure}
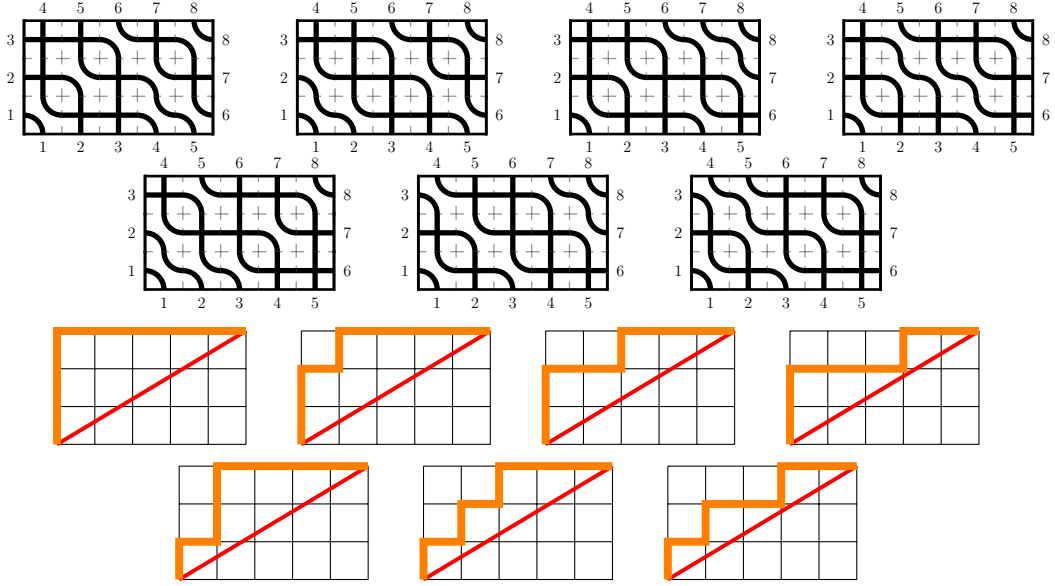
 
\begin{theorem}\label{thm:geometric}
\begin{enumerate}[label=(\arabic*)]
\item \label{part1} \cite{positroid} Let $\gcd(k,n)=1$. Then the sets $\deomax{f_{k,n}}$ and $\Dyck_{k,n}$ have the same cardinality. They are both equal to $C_{k,n-k} = \frac1n \binom{n}{k}$, a rational Catalan number.
\item \label{decouplepart} \cite{gallam2024} Suppose $\overline{f}=(a^\parr1_1\cdots a^\parr1_{n_1})(a^\parr2_1\cdots a^\parr2_{n_2})\cdots(a^\parr r_1\cdots a^\parr r_{n_r})$
is the cycle decomposition of $\overline f$. For each $j\in[r]$, let $S_j$ be the set of all integers congruent modulo $n$ to one of $a^\parr j_1,\dots,a^\parr j_{n_j}$, and let $f|_{S_j}\in\BND_{k_j,n_j}$ be the restriction of $f$ to $S_j$. Then,
\[
	\#\deomax{f} = \prod_{j=1}^r \#\deomax{f|_{S_j}}.
\] 
\item \label{part2} \cite{KLS13} Let $\sigma f:\Z\to\Z$ denote the cyclic shift of $f$, given by $(\sigma f)(i)=f(i-1)+1$ for all $i\in\Z$. Then $\Pi_f^{\circ}\cong\Pi_{\sigma f}^{\circ}$. In particular,
\begin{equation}\label{eq:cyclic}
\#\deo{f}=\#\deo{\sigma f}.
\end{equation}
\item \label{part3} \cite{galashin2022twist,muller2017twist} For any $f\in\Bkn$, let $\antideo{f}$ denote the set of anti-distinguished $f$-Deograms. Then, 
\begin{equation}\label{eq:anti}
\#\deo{f}=\#\antideo{f}.
\end{equation} 
\item \label{part4} \cite{fraser2022positroid,muller} If $\ell(f)=\ell(s_ifs_i)$, then $\Pi_f^{\circ} \cong \Pi_{s_ifs_i}^{\circ}$. In particular,
\begin{equation}\label{eq:ceq}
\#\deo{f} = \#\deo{s_ifs_i}.
\end{equation}
\item \label{part5} \cite{muller} If $\ell(f)+2 = \ell(s_ifs_i)$, $\Pio_f$ decomposes into a closed subvariety isomorphic to $\Pio_{s_ifs_i}\times\C$ and its open complement, isomorphic to $\Pio_{s_if}\times\C^*$. In particular,
\begin{equation}\label{eq:double}
	\#\deo{f} = \#\deo{s_if} + \,\#\deo{s_ifs_i}.
\end{equation}
\end{enumerate}
\end{theorem}

The main results of this paper present combinatorial proofs of \Cref{thm:geometric}~\labelcref{decouplepart,part3,part5} and cases of \Cref{thm:geometric}~\ref{part4}. Additionally, we reduce the problem of a combinatorial proof of \Cref{thm:geometric}~\ref{part1} to one of \Cref{thm:geometric}~\ref{part2}; see Theorem \ref{thm:problem}.

\begin{figure}

\centering
\setlength{\tabcolsep}{6pt}
\renewcommand{\arraystretch}{1.3}
\begin{tabular}{r@{\hskip 12pt}ccc}
$\deomax{f}:$ &
\deog{0/0,0/1,0/2,3/0,4/1,4/2,0/0,0/1,0/2,3/0,4/1,4/2}{0/3,1/3,2/3,3/3,0/0,1/0,2/0,3/1}{-0.35/0.5/1,-0.35/1.5/2,-0.35/2.5/3,0.5/3.35/4,1.5/3.35/5,2.5/3.35/6,3.5/3.35/7,0.5/-0.3/1,1.5/-0.3/2,2.5/-0.3/3,3.3/0.5/4,3.5/0.7/5,4.3/1.5/6,4.3/2.5/7}{1/2,3/2,0/1,2/1,0/0,1/0}{0/2,2/2,1/1,3/1,2/0}
&
\deog{0/0,0/1,0/2,3/0,4/1,4/2,0/0,0/1,0/2,3/0,4/1,4/2}{0/3,1/3,2/3,3/3,0/0,1/0,2/0,3/1}{-0.35/0.5/1,-0.35/1.5/2,-0.35/2.5/3,0.5/3.35/4,1.5/3.35/5,2.5/3.35/6,3.5/3.35/7,0.5/-0.3/1,1.5/-0.3/2,2.5/-0.3/3,3.3/0.5/4,3.5/0.7/5,4.3/1.5/6,4.3/2.5/7}{0/2,2/2,3/2,1/1,3/1,0/0}{1/2,0/1,2/1,1/0,2/0}
&
\deog{0/0,0/1,0/2,3/0,4/1,4/2,0/0,0/1,0/2,3/0,4/1,4/2}{0/3,1/3,2/3,3/3,0/0,1/0,2/0,3/1}{-0.35/0.5/1,-0.35/1.5/2,-0.35/2.5/3,0.5/3.35/4,1.5/3.35/5,2.5/3.35/6,3.5/3.35/7,0.5/-0.3/1,1.5/-0.3/2,2.5/-0.3/3,3.3/0.5/4,3.5/0.7/5,4.3/1.5/6,4.3/2.5/7}{0/2,1/2,3/2,1/1,2/1,0/0}{2/2,0/1,3/1,1/0,2/0}
\\[10pt]
$\deomax{\sigma f}:$ &
\deog{0/0,0/1,0/2,4/0,4/1,4/2,0/0,0/1,0/2,4/0,4/1,4/2}{0/3,1/3,2/3,3/3,0/0,1/0,2/0,3/0}{-0.35/0.5/1,-0.35/1.5/2,-0.35/2.5/3,0.5/3.35/4,1.5/3.35/5,2.5/3.35/6,3.5/3.35/7,0.5/-0.3/1,1.5/-0.3/2,2.5/-0.3/3,3.5/-0.3/4,4.3/0.5/5,4.3/1.5/6,4.3/2.5/7}{1/2,2/2,2/1,3/1,0/0,3/0}{0/2,3/2,0/1,1/1,1/0,2/0}
&
\deog{0/0,0/1,0/2,4/0,4/1,4/2,0/0,0/1,0/2,4/0,4/1,4/2}{0/3,1/3,2/3,3/3,0/0,1/0,2/0,3/0}{-0.35/0.5/1,-0.35/1.5/2,-0.35/2.5/3,0.5/3.35/4,1.5/3.35/5,2.5/3.35/6,3.5/3.35/7,0.5/-0.3/1,1.5/-0.3/2,2.5/-0.3/3,3.5/-0.3/4,4.3/0.5/5,4.3/1.5/6,4.3/2.5/7}{1/2,2/2,0/1,3/1,0/0,2/0}{0/2,3/2,1/1,2/1,1/0,3/0}
&
\deog{0/0,0/1,0/2,4/0,4/1,4/2,0/0,0/1,0/2,4/0,4/1,4/2}{0/3,1/3,2/3,3/3,0/0,1/0,2/0,3/0}{-0.35/0.5/1,-0.35/1.5/2,-0.35/2.5/3,0.5/3.35/4,1.5/3.35/5,2.5/3.35/6,3.5/3.35/7,0.5/-0.3/1,1.5/-0.3/2,2.5/-0.3/3,3.5/-0.3/4,4.3/0.5/5,4.3/1.5/6,4.3/2.5/7}{0/2,2/2,2/1,3/1,0/0,1/0}{1/2,3/2,0/1,1/1,2/0,3/0}
\\[10pt]
$\antideomax{f}:$ &
\deog{0/0,0/1,0/2,3/0,4/1,4/2,0/0,0/1,0/2,3/0,4/1,4/2}{0/3,1/3,2/3,3/3,0/0,1/0,2/0,3/1}{-0.35/0.5/1,-0.35/1.5/2,-0.35/2.5/3,0.5/3.35/4,1.5/3.35/5,2.5/3.35/6,3.5/3.35/7,0.5/-0.3/1,1.5/-0.3/2,2.5/-0.3/3,3.3/0.5/4,3.5/0.7/5,4.3/1.5/6,4.3/2.5/7}{1/2,3/2,0/1,2/1,0/0,1/0}{0/2,2/2,1/1,3/1,2/0}
&
\deog{0/0,0/1,0/2,3/0,4/1,4/2,0/0,0/1,0/2,3/0,4/1,4/2}{0/3,1/3,2/3,3/3,0/0,1/0,2/0,3/1}{-0.35/0.5/1,-0.35/1.5/2,-0.35/2.5/3,0.5/3.35/4,1.5/3.35/5,2.5/3.35/6,3.5/3.35/7,0.5/-0.3/1,1.5/-0.3/2,2.5/-0.3/3,3.3/0.5/4,3.5/0.7/5,4.3/1.5/6,4.3/2.5/7}{2/2,3/2,1/1,3/1,0/0,2/0}{1/2,0/1,2/1,1/0,0/2}
&
\deog{0/0,0/1,0/2,3/0,4/1,4/2,0/0,0/1,0/2,3/0,4/1,4/2}{0/3,1/3,2/3,3/3,0/0,1/0,2/0,3/1}{-0.35/0.5/1,-0.35/1.5/2,-0.35/2.5/3,0.5/3.35/4,1.5/3.35/5,2.5/3.35/6,3.5/3.35/7,0.5/-0.3/1,1.5/-0.3/2,2.5/-0.3/3,3.3/0.5/4,3.5/0.7/5,4.3/1.5/6,4.3/2.5/7}{2/0,1/2,3/2,1/1,2/1,0/0}{2/2,0/1,3/1,1/0,0/2}
\\
\end{tabular}

\caption{Examples of \Cref{thm:geometric}~\labelcref{part2,part3} for $f = [3,5,6,7,8,9,10]$.}
\end{figure}

\begin{theorem}\label{thm:decoupling_intro}
Let $f\in\Bkn$ and $P\in\Pkn$. Suppose
\[
\overline{f}=(a^\parr1_1\cdots a^\parr1_{n_1})(a^\parr2_1\cdots a^\parr2_{n_2})\cdots(a^\parr r_1\cdots a^\parr r_{n_r})
\]
is the cycle decomposition of $\overline f$. For each $j\in[r]$, let $S_j$ be the set of all integers congruent modulo $n$ to one of $a^\parr j_1,\dots,a^\parr j_{n_j}$, and let
\[
	f_j:=f|_{S_j},\qquad P_j:=P|_{S_j},\qquad k_j:=\av(f_j),
\]
where $f_j$ is the restriction of $f$ to $S_j$, viewed as an $n_j$-periodic affine permutation via the order-preserving identification $S_j\cong\Z$, and $P_j$ is the path obtained by restricting $P$ to the steps indexed by $S_j$. Then $f_j\in\BND_{k_j,n_j}$, and:
\begin{enumerate}
\item if $P_j\notin\mathcal P_{k_j,n_j}$ for some $j\in[r]$, then $\affdeomax{f}{P}=\varnothing$;
\item otherwise, there is a bijection
\[
    \affdeomax{f}{P}
    \longrightarrow
    \prod_{j=1}^r \affdeomax{f_j}{P_j}.
\]
\end{enumerate}
\end{theorem}

Theorem~\ref{thm:decoupling_intro} is proved in Section~\ref{sec:affdeo}. Using the correspondence between ordinary Deograms and affine Deograms highlighted in Section~\ref{sec:affdeo}, this recovers \Cref{thm:geometric}~\labelcref{decouplepart}. As a consequence, the enumeration of maximal affine Deograms reduces to the case where $\overline{f}$ is a single cycle. 

\begin{theorem}\label{thm:dist_to_anti}
For any $f\in\Bkn$, $P\in\Pkn$ and a reduced word $\mathbf{f}$ for $f$, we have a bijection
\[
	\phi_{{\mathbf{f}}}:  \antiaffdeo{f}{P} \to \affdeo{f}{P}.
\]
\end{theorem}
 
Theorem \ref{thm:dist_to_anti} provides a combinatorial proof of \Cref{thm:geometric}~\labelcref{part3}. We discuss the combinatorial ideas behind Theorem \ref{thm:dist_to_anti} in Section \ref{sec:antitodist} and the geometric motivation in Section \ref{sec:geometry}.
 
\begin{theorem}\label{thm:eqmoves}
Let $f\in\Bkn$, $P\in\Pkn$ and suppose $\ell(s_ifs_i) = \ell(f)+2$. If $s_iP=P$ or $P_i=R,P_{i+1}=U$, then we have a bijection
\[
	\affdeo{f}{P} \to  \affdeo{s_if}{s_iP}\sqcup \affdeo{s_ifs_i}{s_iP}.
\]
Otherwise, if $P_i = U,P_{i+1}=R$, then we instead have a bijection
\[
	\affdeo{f}{P} \to \affdeo{fs_i}{s_iP}\sqcup \affdeo{s_ifs_i}{s_iP}.
\]
Suppose $\ell(f)=\ell(s_ifs_i)$ with $fs_i>f$ and $s_if<f$. If $s_iP=P$ or $P_i=R,P_{i+1}=U$, then we have a bijection
\[
	\affdeo{f}{P}\to \affdeo{s_ifs_i}{s_iP}.
\]
\end{theorem}
 
Theorem \ref{thm:eqmoves} provides a combinatorial proof of \Cref{thm:geometric}~\labelcref{part5} and cases of \Cref{thm:geometric}~\labelcref{part4}. We prove Theorem \ref{thm:eqmoves} in Section \ref{sec:affdeo}. The remaining length-preserving case of item~\labelcref{part4} is not covered by Theorem~\ref{thm:eqmoves}; it is related to the cluster structure of open positroid varieties and appears in item (5') of Theorem \ref{thm:problem} below. We highlight this phenomenon in Section \ref{sec:geometry}, where we also prove explicit isomorphisms of \textbf{affine patches of open positroid varieties}, which induce the equalities stated in Theorems \ref{thm:eqmoves} and \ref{thm:problem}.

For $i\in[n]$, let $\sGrI_i(f) = \{ a\,|\,a < i, f(a) \geq i\} \bmod n$ denote the $i$th element of the \textbf{source Grassmann necklace} of $f\in \Bkn$.

\begin{theorem}\label{thm:problem}
Constructing a bijection realizing either of the following cardinality equalities would give a bijective proof of $\#\deomax{f_{k,n}}=\#\Dyck_{k,n}$.
\begin{enumerate}
\item[(2')] \label{part2prob} For any $f\in\Bkn$, $\#\affdeo{f}{P_{\sGrI_i(f)}} = \#\affdeo{f}{P_{\sGrI_{i+1}(f)}}$.
\item[(5')] \label{part4prob} For any $f,s_ifs_i\in\mathbf{B}_{k,n}$ with $\ell(f)=\ell(s_ifs_i)$,
\[
	\#\affdeo{f}{P_{\sGrI_{i+1}(f)}} = \#\affdeo{s_ifs_i}{P_{\sGrI_{i+1}(s_ifs_i)}}.
\]
\end{enumerate}
\end{theorem}

We prove Theorem \ref{thm:problem} in Section \ref{sec:problems}. We also provide a bijection between $\deomax{f_{k,n}}$ and $\Dyck_{k,n}$ for $k=2$ and $k=n-2$ in Section \ref{sec:problems}.

Finally, in Section~\ref{sec:invariance}, we give a purely combinatorial proof that the point count of the Pl\"ucker chart $C_I=\{x\in\Grkn\,|\,\Delta_I(x)\neq0\}$ is independent of $I$. While the identity $\#C_I(\F_q)=q^{k(n-k)}$ follows from the geometry of the Grassmannian, it specializes to a non-trivial identity between affine Deodhar diagrams.

\subsection{Acknowledgements}

The author would like to thank Pavel Galashin for introducing the problem and for their guidance and Olha Shevchenko for helpful comments.

\section{Background}\label{sec:background}

\subsection{Bounded Affine Permutations}

An ($n$-periodic) affine permutation is a bijection $f:\Z\to\Z$ which is periodic, i.e., $f(i+n)=f(i)+n$. Let $\Saff$ denote the group of all $n$-periodic affine permutations, typically referred to as the extended affine Weyl group of type $A$. We may define the set of inversions of $f$ as $\Inv(f):=\{(i,j)\,|\,i<j,f(i)>f(j),i\in[n],j\in\Z\}$, and the length $\ell(f):=\#\Inv(f)$. Throughout the paper, products of affine permutations are read from left to right: $(fg)(a)=g(f(a))$. 

Define
\[
	\Saffk = \{f \in \Saff\,|\,\av(f)=k\} \subset \Saff,
\]
where $\av(f)=\frac1n\sum_{i=1}^n(f(i)-i)$. Let $\Lambda\in \Saff$ be the affine permutation where $\Lambda(i)=i+1$ for all $i\in\Z$. We have a canonical isomorphism $\Saffk \cong \Saffz$, given by $f\mapsto f\circ \Lambda^{-k}$ where $\Saffz$ is simply the Weyl group of type $\Anaff$. This isomorphism allows us to use the Bruhat order on $\Saffk$, where for $f,g\in\Saffk$, we say $f\leq g$ if $f\circ \Lambda^{-k}\leq g \circ \Lambda^{-k}$.

A bounded affine permutation is an affine permutation $f\in\Saff$ additionally satisfying $i\leq f(i)\leq i+n$. We let $\Bkn\subset \Saffk$ denote the set of bounded affine permutations inside $\Saffk$, referred to as the set of $(k,n)$-bounded affine permutations. Given $f\in\Saff$, define the cyclic shift $\sigma f\in\Saff$ by
\begin{equation}\label{eq:cyclicshift}
(\sigma f)(i)=f(i-1)+1 \quad \text{ for all } i\in\Z.
\end{equation}
In other words, $\sigma f=\Lambda^{-1} f\Lambda$. Note that $\sigma$ preserves $\Saffk$ and $\Bkn$.

For $i\in\Z$, let $s_i\in\Saff$ be the simple transposition where $i\mapsto i+1$, $i+1\mapsto i$, and $j\mapsto j$ for $j\neq i,i+1$, where the indices are taken modulo $n$. For $f\in\Bkn$ and $i\in\Z$, we have $\ell(s_if)=\ell(f)\pm1$ and $\ell(fs_i)=\ell(f)\pm1$. We say $f<fs_i$ if $\ell(f)<\ell(fs_i)$.

\subsection{Open Positroid Varieties and Their Affine Patches}

For a $k\times n$ matrix $M$ and a $k$-subset $I\in\binom{[n]}k$, let $\Delta_I(M)$ denote the $I$th Pl\"ucker coordinate of $M$, i.e., the determinant of the $k\times k$ submatrix of $M$ in columns $I$. If two matrices $M$ and $M'$ have the same row span, then they are related by left multiplication by an element of $\GL_k$. Therefore, their Pl\"ucker coordinates coincide up to a common scalar multiple. 

Let $\Grkn$ denote the Grassmannian of $k$-dimensional subspaces of $\C^n$. The Pl\"ucker embedding $\Grkn\hookrightarrow \mathbb{P}^{\binom{n}k -1}$ realizes $\Grkn$ as a projective variety. This embedding sends a point $x\in\Grkn$ to the list of Pl\"ucker coordinates $(\Delta_I(x))_{I\in\binom{[n]}k}\in\mathbb{P}^{\binom{n}k-1}$, where $\Delta_I(x)$ is defined as the Pl\"ucker coordinate of any $k\times n$ matrix representative for $x$.

For every bounded affine permutation $f\in\Bkn$, we define the \textbf{source Grassmann necklace} $\sGr(f) = (\sGrI_1(f),\dots,\sGrI_n(f))$ where
\[
    \sGrI_i(f) = \{ a \,| \, a < i, f(a) \geq i\} \bmod n,
\]
and the \textbf{target Grassmann necklace} $\tGr(f) = (\tGrI_1(f),\dots,\tGrI_n(f))$ where
\[
    \tGrI_i(f) =  \{ f(a) \,| \, a < i, f(a) \geq i\} \bmod n.
\]
For $f\in\Bkn$, we have $\sGrI_i(f) ,\tGrI_i(f) \in\binom{[n]}k$ for all $i\in[n]$. For any $i\in[n]$, let $\leq_i$ denote the order on $[n]$ in which $i$ is smallest and $i-1$ is largest, i.e. $i <_i i+1 <_i\dots <_i i-1$. For a pair of subsets $I =\{ i_1 <_i\dots <_i i_k\}$, $J = \{j_1 <_i\dots <_i j_k\}$, we say that $I\leq_i J$ if $i_r \leq_i j_r$ for all $r$.

The \textbf{positroid envelope} of $f$, $\Mcal(f)$, is then read off from either the source Grassmann necklace or the target Grassmann necklace. This construction is known as Oh's Theorem \cite{oh2011positroids}:
\[
\Mcal(f) = \left\{I\in\binom{[n]}k \,\big|\,I \geq_i \tGrI_i \text{ for } i\in[n] \right\} = \left\{I\in\binom{[n]}k \,|\,I \leq_i \sGrI_i \text{ for } i\in[n] \right\}.
\]
The open positroid variety $\Pio_f\subset \Grkn$ is then:
\[
    \Pio_f = \{x\in\Grkn\,|\,\Delta_I(x) = 0 \text{ for all } I\not\in\Mcal(f) \text{ and } \Delta_I(x)\neq0 \text{ for all } I \in \tGr(f)\}.
\]
Equivalently (see, for example, \cite{fraser2022positroid}),
\[
    \Pio_f = \{x\in\Grkn\,|\,\Delta_I(x) = 0 \text{ for all } I\not\in\Mcal(f) \text{ and } \Delta_I(x)\neq0 \text{ for all } I \in \sGr(f)\}.
\]

This paper primarily concerns itself with affine patches of open positroid varieties:
\[
    \Pio_{f,I} = \Pio_f \cap\{M\in\Grkn\,|\,\Delta_I(M)\neq 0\}.
\]
For $I\in\binom{[n]}k$, let $t_I \in\Saffk$ be the bounded affine permutation where 
\[
    t_I(i) = \begin{cases} i+n, & i \bmod n \in I,\\ i, & i \bmod n\not\in I.\end{cases}
\]
The following theorem by Snider \cite{snider} shows that affine patches of open positroid varieties are isomorphic to open affine Richardson varieties.
\begin{theorem}[\cite{snider}]\label{thm:snideriso}
For $f\in\Bkn$ and $I\in\binom{[n]}{k}$, with $f\leq t_I$ in the Bruhat order, we have an isomorphism
\[
    \mathring{\Rcal}_f^{t_I} \xrightarrow{\sim}  \Pio_{f,I}.
\]
\end{theorem}

\subsection{Deograms}\label{subsec:deogram}

We say that a permutation $w \in S_n$ is $k$-Grassmannian if $w(1)<w(2)<\dots<w(k)$ and $w(k+1)<\dots<w(n)$. We denote by $\leq$ the Bruhat order on $S_n$. Let $\mathbf{Q}_{k,n}$ denote the set of pairs $(v, w)$ of permutations such that $v \leq w$ and $w$ is $k$-Grassmannian. We state a few results about open positroid varieties useful for our discussion of $f$-Deograms. For $f\in\Bkn$, let $\overline{f}\in S_n$ be the permutation $\overline{f}(i) = f(i) \bmod n$.

\begin{proposition}[{\cite[Proposition 3.15]{KLS13}}]\label{prop:v_w}
There exists a bijection $(v,w)\mapsto f_{v,w}$ between $\mathbf{Q}_{k,n}$ and $\mathbf{B}_{k,n}$ such that for every $f=f_{v,w}\in \mathbf{B}_{k,n}$, we have $\overline{f}=w^{-1}v$ and $f=w^{-1}\tau_{k,n}v$.
\end{proposition}

Let $G = SL_n(\C)$. Let $B$ and $B_-$ be the subgroup of invertible upper (resp. lower) triangular matrices, and let $N\subset B$ and $N_-\subset B_-$, be the subset of upper (resp. lower) unitriangular matrices. We then define Richardson varieties, indexed by $v\leq w$, as:
\begin{equation}\label{eq:richardson}
	\mathring{R}_{v,w} = (B_-wB_- \cap Bv B_-) / B_- \subset G/B_- = \Fl(n).
\end{equation}

\begin{proposition}[{\cite[Theorem 5.9]{KLS13}}]\label{prop:richardson}
For each $f=f_{v,w}\in\Bkn$, the natural projection map $\Fl(n)\to\Grkn$ restricts to an isomorphism $\mathring{R}_{v,w}\cong\Pio_f$. Thus, open positroid varieties are special cases of open Richardson varieties.
\end{proposition}

The set of $k$-Grassmannian permutations in $S_n$ is well known to be in bijection with the set of Young diagrams that fit inside a $k\times (n-k)$-rectangle. Let $\la$ be such a Young diagram. We are going to consider fillings of boxes of $\la$ with \emph{crossings}~\crossing and \emph{elbows}~\elbow. Each such filling $D$  gives rise to a permutation $w_D$, obtained as follows. Consider paths labeled by $1,2,\dots,n$ entering from the northwest boundary of $\la$, where the labels increase in the northeast direction. The paths follow crossings and elbows until they exit through the southeast boundary of $\la$. Recording the positions of outgoing edges, one obtains the permutation $w_D$. Observe that when a Deogram $D$ of shape $\lambda$ consists entirely of crossings, the permutation $w_D = w$ indeed is $k$-Grassmannian: we have $w(1)<w(2)<\dots<w(k)$ and $w(k+1)<\dots<w(n)$. We denote this correspondence by $\lambda_w :=\la$.

\begin{definition}\label{def:deogram}
For any $f=f_{v,w}\in\Bkn$, we define an $f$-Deogram as a filling of the shape $\la_w$ with crossings and elbows such that, starting from the northwest boundary,
\begin{enumerate}
\item the resulting strand permutation equals $v$,
\item (distinguished) if two strands have crossed an odd number of times, they cannot form an elbow.
\end{enumerate}

We let $\deo{f}$ denote the set of all $f$-Deograms. An $f$-Deogram is maximal if it contains precisely $n-c(\overline{f})$ elbows, where $c(\overline{f})$ is the number of cycles of $\overline{f}$. We let $\deomax{f}$ denote the set of all maximal $f$-Deograms. 
\end{definition}

\begin{remark}\label{rem:maxideo}
Any $f$-Deogram must have at least $n-c(\overline{f})$ elbows. One can see this by iteratively replacing bottom-rightmost elbows with crossings to obtain a Deogram with all crossings, which only has fixed points. Each elbow then corresponds to resolving a crossing, which either adds or removes a cycle. So, to obtain $c(\overline{f})$ many cycles, we require at least $n-c(\overline{f})$ many elbows.
\end{remark}

One of the key uses of $f$-Deograms is that they readily provide a decomposition of open positroid varieties (more generally, open Richardson varieties) into simpler pieces.

\begin{theorem}[{\cite{deodhar}}]\label{thm:deodharorig}
For $f=f_{v,w}\in\Bkn$, 
\[
\Pio_f = \bigsqcup_{D\in \deo{f}} (\F^*)^{\#\text{elbows}(D)}\times \F^{(\#\text{crossings}(D)-\ell(v))/2}.
\]
\end{theorem}

In particular, the point count of $\Pio_f$ over $\F_q$ is the $R$-polynomial
\[
    R_f(q):=\#\Pio_f(\F_q)=\sum_{D\in\deo{f}}(q-1)^{\#\text{elbows}(D)}\,q^{(\#\text{crossings}(D)-\ell(v))/2},
\]
a Kazhdan--Lusztig $R$-polynomial in the sense of \cite{kl1}. 

\subsection{Words, Subwords, and Twists}\label{subsec:words}

A \textit{word} is any finite sequence $\bw=(s_{i_1},s_{i_2},\dots,s_{i_m})$ of elements in $T_n := \{s_1,\dots,s_{n-1}\}$, with repetition allowed. If $w=s_{i_1}s_{i_2}\dots s_{i_m}$, we call $\bw$ a $w$-word, and if $m=\ell(w)$, we say it is reduced. A subword of $\bw$ is a sequence $\bv=(v_1,v_2,\dots,v_m)$ where $v_j\in\{s_{i_j},e\}$ for all $j$. For any such sequence, we set $\bv_{(j)}=v_1v_2\dots v_j$. If $\bv_{(m)}=v$, we refer to $\bv$ as a $v$-subword of $\bw$.

\begin{definition}{\cite{deodhar,marsh}}
For $v\in S_n$, we say that a $v$-subword $\bv$ of $\bw$ is \textit{distinguished} if $\bv_{(j)}\leq \bv_{(j-1)}s_{i_j}$ for all $j$. We write $\Dcal_{v,\bw}$ for the set of distinguished $v$-subwords of $\bw$. Additionally, for any $v$-subword $\bv$ of $\bw$, we write
\[
	e_{\bv}:=\#\{j\in[m]\,|\,v_j=e\}.
\]
We write $\Dcal^r_{v,\bw}:=\{\bv\in\Dcal_{v,\bw}\,|\,e_{\bv}=r\}$.
\end{definition}

To each Deogram $D\in \deo{f}=\deo{f_{v,w}}$, we may associate a distinguished subword $\mathbf{v}\in \Dcal_{v,\bw}$, so that the indices $j$ such that $v_j=s_{i_j}$ correspond to the crossings in $D$. This correspondence is easily seen to be bijective, and we note that $e_{\bv}$ equals the number of elbows in our Deogram $D$. We also define \textit{twisted} words.

\begin{definition}{\cite{nathan}}
We say that a subword $\mathbf{v}$ of a word $\mathbf{w}$ is $u$-distinguished if $u\bv_{(j)}\leq u\bv_{(j-1)}s_{i_j}$ for each $j\in[m]$. We write $\Dcal_{v,\bw}^{(u)}$ for the set of $u$-distinguished $v$-subwords of $\mathbf{w}$. Similarly to the previous case, we write $\Dcal_{v,\bw}^{(u),r}\subseteq \Dcal_{v,\bw}^{(u)}$ for the subset of elements $\mathbf{v}$ such that $e_{\mathbf{v}}=r$.
\end{definition}

\section{Affine Deograms}\label{sec:affdeo}

\begin{definition} Let $\Pkn$ denote the set of lattice paths from $(0,0)$ to $(n-k,k)$. For $P\in \Pkn$, let $\tau_P$ denote the bounded affine permutation obtained by filling the (periodic) region cut out by $P$ entirely with crossings.

For any $P\in\Pkn$, we let $P_i\in\{\text{R},\text{U}\}$ denote whether the $i$th step of $P$ is a right or up-step.

Define $s_iP$ as the path obtained by swapping $P_i$ and $P_{i+1}$, where the indices are taken mod $n$. Let $P_U  = \{i\in[n]\,|\,P_i=\text{U}\}$. Finally, for $I\in\binom{[n]}{k}$, let $P_I\in\Pkn$ denote the unique lattice path with $(P_I)_U=I$; thus $\tau_{P_I}=t_I$.
\end{definition}

\begin{definition}
Recall the definition of an $(f,P)$-affine Deogram, stated in Section~\ref{sec:intro}; see Figure \ref{fig:bigaffdeo} for an example. We let $\affdeo{f}{P}$ denote the set of all $(f,P)$-affine Deograms. An $(f,P)$-affine Deogram is \emph{maximal} if it contains precisely $n-c(\overline{f})$ elbows, counting just the columns with indices $1\leq i \leq n$. We let $\affdeomax{f}{P}$ denote the set of all maximal $(f,P)$-affine Deograms. 

Finally, we let $\affdeop{P}= \bigsqcup_{f\in \Bkn}\affdeo{f}{P}$.
\end{definition}

\begin{remark}\label{rem:maxiaffdeo}
As in the case of $f$-Deograms, every $(f,P)$-affine Deogram requires at least $n-c(\overline{f})$ elbows; the argument is identical to that of Remark \ref{rem:maxideo}. It may be the case that $\affdeo{f}{P}$ is non-empty but contains no $(f,P)$-affine Deograms with precisely $n-c(\overline{f})$ elbows.
\end{remark}

Generalizing Theorem \ref{thm:deodharorig} of \cite{deodhar}, affine Deograms decompose affine patches of open positroid varieties into simpler pieces. This generalization is due to Billig--Dyer \cite{billig}.

\begin{theorem}[\cite{billig}]
\label{thm:billig}
For $f\in\Bkn$ and $P\in\Pkn$ with $f\leq \tau_P$, 
\[
\Pio_{f,P_U} = \bigsqcup_{A\in \affdeo{f}{P}} (\F^*)^{\#\text{elbows}(A)}\times \F^{(\#\text{crossings}(A)-\ell(f))/2}.
\]
\end{theorem}
This implies the following corollary.
\begin{corollary}
For $f\in\Bkn$ and $P\in\Pkn$ with $f\leq \tau_P$ and $q$ a prime power, the point count of $\Pio_{f,P_U}$ over the finite field $\F_q$ is computed as
\begin{equation}\label{eq:rftau}
R_{f,\tau_P}(q) := \#\Pio_{f,P_U}(\F_q) = \sum_{A\in \affdeo{f}{P}} (q-1)^{\#\text{elbows}(A)} q^{(\#\text{crossings}(A)-\ell(f))/2}.
\end{equation}
\end{corollary}

\subsection{Affine Words, Subwords, and Twists}

We extend the definitions of Section \ref{subsec:words} to the affine setting. A \textit{word} is any finite sequence $\mathbf{g}=(s_{i_1},s_{i_2},\dots,s_{i_m})$ of elements in $\hat{T}_n := \{s_0,s_1,\dots,s_{n-1}\}$, with repetition allowed. If $g=s_{i_1}s_{i_2}\dots s_{i_m}$, we call $\mathbf{g}$ a $g$-word, and if $m=\ell(g)$, we say it is reduced. A subword of $\mathbf{g}$ is a sequence $\mathbf{f}=(f_1,f_2,\dots,f_m)$ where $f_j\in\{s_{i_j},e\}$ for all $j$. For any such sequence, we set $\mathbf{f}_{(j)}=f_1f_2\dots f_j$. If $\mathbf{f}_{(m)}=f$, we refer to $\mathbf{f}$ as a $f$-subword of $\mathbf{g}$.

\begin{definition}
For $f\in \tilde{S}_n$, we say that a $f$-subword $\mathbf{f}$ of $\mathbf{g}$ is \textit{distinguished} if $\mathbf{f}_{(j)}\leq \mathbf{f}_{(j-1)}s_{i_j}$ for all $j$. We write $\Acal_{f,\mathbf{g}}$ for the set of distinguished $f$-subwords of $\mathbf{g}$. Additionally, for any $f$-subword $\mathbf{f}$ of $\mathbf{g}$, we write
\[
	e_{\mathbf{f}}:=\#\{j\in[m]\,|\,f_j=e\}.
\]
We write $\Acal^r_{f,\mathbf{g}}:=\{\mathbf{f}\in\Acal_{f,\mathbf{g}}\,|\,e_{\mathbf{f}}=r\}$.
\end{definition}

To each affine Deogram $A\in \affdeo{f}{P}$, we may associate a distinguished subword $\mathbf{f}\in \Acal_{f,\mathbf{\tau_P}}$, so that the indices $j$ such that $f_j=s_{i_j}$ correspond to the crossings in $A$. As before, this correspondence is easily seen to be bijective, and we note that $e_{\mathbf{f}}$ equals the number of elbows inside any representative of our affine Deogram $A$.

\begin{definition}
We say that a subword $\mathbf{f}$ of a word $\mathbf{g}$ is $h$-distinguished if $h\mathbf{f}_{(j)}\leq h\mathbf{f}_{(j-1)}s_{i_j}$ for each $j\in[m]$. We write $\Acal_{f,\mathbf{g}}^{(h)}$ for the set of $h$-distinguished $f$-subwords of $\mathbf{g}$. Similarly to the previous case, we write $\Acal_{f,\mathbf{g}}^{(h),r}\subseteq \Acal_{f,\mathbf{g}}^{(h)}$ for the subset of elements $\mathbf{f}$ such that $e_{\mathbf{f}}=r$.
\end{definition}

Fix $P\in\Pkn$ and let $B_P$ be the set of boxes of a fundamental domain of the strip determined by $P$, partially ordered by $b\preceq b'$ if $b$ lies weakly northwest of $b'$. A \emph{reading order} of $B_P$ is a linear extension of $\preceq$. Reading the boxes of $B_P$ in a reading order and recording, for each box, the simple generator indexed by the two strands meeting in it, produces a reduced word $\mathbf{\tau_P}$ in $\hat T_n$ with $k(n-k)$ letters. Two reading orders differ by a sequence of transpositions of $\preceq$-incomparable boxes, so the corresponding words differ by commutation moves, and the resulting sets $\Acal_{f,\mathbf{\tau_P}}$ are canonically identified. We fix one reading order once and for all, and abbreviate $\Acal^{r}_{f,P}:=\Acal^{r}_{f,\mathbf{\tau_P}}$ and $\Acal^{(h),r}_{f,P}:=\Acal^{(h),r}_{f,\mathbf{\tau_P}}$; likewise $\affdeo{f}{P}^{(h)}$ denotes the set of fillings corresponding to $\Acal^{(h)}_{f,P}$, i.e.\ the set of \emph{$h$-twisted} $(f,P)$-affine Deograms.

\subsection{Correspondence between Deograms and Affine Deograms}

Let $\tau_{k,n} \in \Saffk$ be the affine permutation defined by
\[
	\tau_{k,n}(i) = \begin{cases} i+n, & 1 \leq i \leq k, \\ i, & k+1 \leq i \leq n, \end{cases}
\]
where the indices are taken mod $n$; equivalently, $\tau_{k,n}=t_{[k]}$.

Throughout, words and subwords for elements of $\Saffk$ are taken with respect to the bijection
\[
	\iota_k:\Saffk\xrightarrow{\ \sim\ }\Saffz,\qquad \iota_k(f)=\Lambda^{-k}f,
\]
which satisfies $\ell(\iota_k(f))=\ell(f)$ and $\iota_k(fu)=\iota_k(f)\,u$ for all $u\in\Saffz$; in particular $\iota_k$ preserves right descents and is compatible with right multiplication by simple generators. By a \emph{reduced word for $f\in\Saffk$} we mean a reduced word for $\iota_k(f)$, and by an \emph{$f$-subword} of a word $\mathbf g$ we mean a subword whose product equals $\iota_k(f)$. (The identification $f\mapsto f\Lambda^{-k}$ of Section~\ref{sec:background} induces the same Bruhat order, since the two differ by conjugation by $\Lambda^{k}$, a length-preserving automorphism of $\Saffz$; but that automorphism sends $s_i\mapsto s_{i-k}$ and is therefore not compatible with right multiplication.)

Let $J=\{s_1,\dots,s_{n-1}\}$, so that $W_J:=\langle J\rangle\cong S_n$, and write $(\Saffk)^J=\{x\in\Saffk\mid xs>x\ \text{for all } s\in J\}$ for the set of minimal-length representatives of the cosets $xW_J$. For every $k$-Grassmannian permutation $w\in S_n$, define $x_w=w^{-1}\tau_{k,n}$ and let $\mathbf{x_w}$ be a reduced word for $x_w$.

\begin{lemma}\label{lemma:xwhelper}
	Let $w \in S_n$ be $k$-Grassmannian with $x_w\neq e$. Then the right descent set of $x_w$ is exactly $\{s_0\}$. In particular $x_w \in (\Saffk)^{J}$, i.e.\ $x_w$ is the unique element of minimal length in the coset $x_wW_J$.
\end{lemma}
\begin{proof}
Recall $x_w^{-1}=\tau_{k,n}^{-1}w$, so that
\[
	x_w^{-1}(i) = \begin{cases} w(i)-n, & 1\le i\le k,\\ w(i), & k+1\le i\le n,\end{cases}
	\qquad x_w^{-1}(n+1)=w(1).
\]
We first show $x_ws_0<x_w$, which is equivalent to $x_w^{-1}(n)>x_w^{-1}(n+1)$, i.e.\ to $w(n)>w(1)$. Since $w$ is $k$-Grassmannian, $w(1)$ is the smallest value in $\{w(1),\dots,w(k)\}$ and $w(n)$ is the largest value in $\{w(k+1),\dots,w(n)\}$. If $w(n)<w(1)$, then every element of $\{w(k+1),\dots,w(n)\}$ is less than every element of $\{w(1),\dots,w(k)\}$, forcing $w=(n-k+1,\dots,n,1,\dots,n-k)$ and hence $x_w=\Lambda^k$, i.e.\ $x_w=e$ in $\Saffz$, contrary to assumption. Thus $w(n)>w(1)$ and $s_0$ is a right descent.

Now let $1\le i\le n-1$; we claim $x_w^{-1}(i)<x_w^{-1}(i+1)$, so that $x_ws_i>x_w$. For $1\le i\le k-1$ and for $k+1\le i\le n-1$ this is immediate from $w(i)<w(i+1)$. For $i=k$ we have $x_w^{-1}(k)=w(k)-n\le 0<w(k+1)=x_w^{-1}(k+1)$.

Hence the right descent set of $x_w$ is $\{s_0\}$, and $x_w\in(\Saffk)^J$ by \cite[Proposition~2.4.4]{bjorner}.
\end{proof}

\begin{corollary}\label{cor:addition}
	For any $v, w \in S_n$ with $w$ $k$-Grassmannian, we have $\ell(x_w v) = \ell(x_w) + \ell(v)$.
\end{corollary}
\begin{proof}
	By Lemma~\ref{lemma:xwhelper}, $x_w \in (\Saffk)^J$. Since $v \in S_n \cong (\Saffk)_{J}$, the product $x_w v$ is a parabolic decomposition. By \cite[Proposition~2.4.4]{bjorner}, we conclude that $\ell(x_w v) = \ell(x_w) + \ell(v)$.
\end{proof}

\begin{lemma}\label{lemma:useallbox}
	Let $v, w \in S_n$ be such that $v \leq w$ and $w$ is $k$-Grassmannian. Then for every $x_wv$-subword $\mathbf{u}$ of $\mathbf{x_ww}$, we have $u_j = s_{i_j}$ for $1 \leq j \leq \ell(x_w)$.
\end{lemma}
\begin{proof}
Set $m=\ell(x_w)$, let $a=\mathbf u_{(m)}$ be the product of the letters selected from $\mathbf{x_w}$, and let $b$ be the product of the letters selected from $\mathbf{w}$, so that $ab=x_wv$. Every letter of $\mathbf{w}$ lies in $J$, so $b\in W_J$ and hence
\[
	a=(x_wv)b^{-1}\in x_wvW_J=x_wW_J.
\]
On the other hand $a$ is the product of a subword of the reduced word $\mathbf{x_w}$, so $a\le x_w$ and $\ell(a)\le \#\{j\le m\mid u_j\neq e\}\le m$. By Lemma~\ref{lemma:xwhelper}, $x_w$ is the unique element of minimal length in $x_wW_J$, so $\ell(a)\ge m$. Therefore $\ell(a)=m$, which together with $a\le x_w$ forces $a=x_w$ and forces every one of the first $m$ letters to be selected: $u_j=s_{i_j}$ for $1\le j\le m$.
\end{proof}

\begin{lemma}\label{lemma:xww}
	Let $w\in S_n$ be $k$-Grassmannian and let $P_w\in\Pkn$ be the lattice path with $(P_w)_U=w([k])=\{w(1),\dots,w(k)\}$ (equivalently, the boundary path of $\lambda_w$). Then
	\[
		x_ww=\tau_{P_w},
	\]
	and the concatenation $\mathbf{x_w}\mathbf{w}$ is a reduced word for $\tau_{P_w}$, of length $k(n-k)$.
\end{lemma}
\begin{proof}
	Conjugation acts on the elements $t_I$ by $u^{-1}t_Iu=t_{u(I)}$, so $x_ww=w^{-1}t_{[k]}w=t_{w([k])}=\tau_{P_w}$. By Corollary~\ref{cor:addition} the concatenation has $\ell(x_w)+\ell(w)=\ell(x_ww)$ letters, hence is reduced, and $\ell(\tau_{P_w})=k(n-k)$.
\end{proof}

From this, we obtain the following bijection; see Figure~\ref{fig:affdeocorrespondence}.

\begin{corollary}\label{cor:affdeocorrespondence}
	Let $v,w\in S_n$ with $v\le w$ and $w$ $k$-Grassmannian, and let $f=f_{v,w}\in\Bkn$. Then for any $r$, prepending the all-crossings filling indexed by $\mathbf{x_w}$ gives a bijection
	\[
		\mathcal{D}^r_{v,\mathbf{w}} \;\xrightarrow{\ \sim\ }\; \mathcal{A}^r_{x_wv,\,\mathbf{x_w}\mathbf{w}}=\mathcal{A}^r_{f,\,\mathbf{\tau_{P_w}}},
	\]
	which restricts to a bijection $\deomax{f}\to\affdeomax{f}{P_w}$. Moreover $(P_w)_U=w([k])=\sGrI_1(f)$.
\end{corollary}
\begin{proof}
	By Lemma~\ref{lemma:useallbox} every $f$-subword of $\mathbf{x_w}\mathbf{w}$ selects all letters of $\mathbf{x_w}$, and these letters impose no condition. Since $x_w\in(\Saffk)^J$ we have $\ell(x_wu)=\ell(x_w)+\ell(u)$ for all $u\in W_J$, so $u\mapsto x_wu$ is an isomorphism of Bruhat posets $W_J\to x_wW_J$; hence the $x_w$-twisted distinguished condition on a subword of $\mathbf w$ agrees with the distinguished condition in $S_n$.
\end{proof}

\begin{figure}
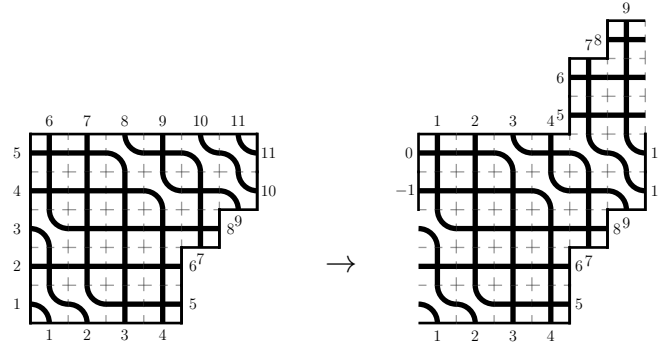

\centering
\begin{tabular}{ccc}
\deog{0/0,0/1,0/2,0/3,0/4,4/0,4/1,5/2,6/3,6/4,0/0,0/1,0/2,0/3,0/4,4/0,4/1,5/2,6/3,6/4}{0/5,1/5,2/5,3/5,4/5,5/5,0/0,1/0,2/0,3/0,4/2,5/3}{-0.350000000000000/0.500000000000000/1,-0.350000000000000/1.50000000000000/2,-0.350000000000000/2.50000000000000/3,-0.350000000000000/3.50000000000000/4,-0.350000000000000/4.50000000000000/5,0.500000000000000/5.35000000000000/6,1.50000000000000/5.35000000000000/7,2.50000000000000/5.35000000000000/8,3.50000000000000/5.35000000000000/9,4.50000000000000/5.35000000000000/10,5.50000000000000/5.35000000000000/11,0.500000000000000/-0.300000000000000/1,1.50000000000000/-0.300000000000000/2,2.50000000000000/-0.300000000000000/3,3.50000000000000/-0.300000000000000/4,4.30000000000000/0.500000000000000/5,4.30000000000000/1.50000000000000/6,4.50000000000000/1.70000000000000/7,5.30000000000000/2.50000000000000/8,5.50000000000000/2.70000000000000/9,6.30000000000000/3.50000000000000/10,6.30000000000000/4.50000000000000/11}{2/4,4/4,5/4,3/3,5/3,0/2,0/0,1/0}{0/4,1/4,3/4,0/3,1/3,2/3,4/3,1/2,2/2,3/2,4/2,0/1,1/1,2/1,3/1,2/0,3/0}

&
$\rightarrow$
&

\deog{4/0,4/5,4/1,4/6,5/2,5/7,6/3,6/4,4/0,4/5,4/1,4/6,5/2,5/7,6/3,6/4,0/4,0/3}{0/0,0/5,1/0,1/5,2/0,2/5,3/0,3/5,4/2,4/7,5/3,5/8}{0.500000000000000/-0.350000000000000/1,0.500000000000000/5.35000000000000/1,1.50000000000000/-0.350000000000000/2,1.50000000000000/5.35000000000000/2,2.50000000000000/-0.350000000000000/3,2.50000000000000/5.35000000000000/3,3.50000000000000/-0.350000000000000/4,3.50000000000000/5.35000000000000/4,4.25000000000000/ 0.500000000000000/5, 3.75000000000000/5.50000000000000/5, 4.25000000000000/1.50000000000000/6,3.75000000000000/6.50000000000000/6,4.50000000000000/1.65000000000000/7,4.50000000000000/7.35000000000000/7,5.25000000000000/2.50000000000000/8,4.75000000000000/7.50000000000000/8,5.50000000000000/2.65000000000000/9,5.50000000000000/8.35000000000000/9,6.35000000000000/3.50000000000000/10,6.35000000000000/4.50000000000000/11,-0.25000000000000/4.50000000000000/0,-0.35000000000000/3.50000000000000/-1}{2/4,4/4,5/4,3/3,5/3,0/2,0/0,1/0}{0/4,1/4,3/4,0/3,1/3,2/3,4/3,1/2,2/2,3/2,4/2,5/7,0/1,1/1,2/1,3/1,4/6,5/6,2/0,3/0,4/5,5/5}
\end{tabular}

\caption{Example of the correspondence between Deograms and affine Deograms for $f = [8,5,7,9,12,17,10,13,11]$.}
\label{fig:affdeocorrespondence}
\end{figure}

\subsection{Deographs and Decoupling}

The goal of this subsection is to prove Theorem~\ref{thm:decoupling_intro}. To do so, we associate a finite multigraph to each affine Deogram.

\begin{definition}
Fix $f\in\Bkn$ and $P\in\Pkn$. For any affine Deogram $A\in\affdeo{f}{P}$, define the multigraph $G(A)=(V,E)$ as follows. Its vertex set is $V=[n]$, and each elbow of $A$ in a representative region contributes one edge joining the labels modulo $n$ of the two strands forming that elbow. We call $G(A)$ the \emph{Deograph} of $A$. See Figure~\ref{fig:deograph}.
\end{definition}

\begin{figure}

\begin{tabular}{cc}
\deog{2/0,2/5,5/1,5/6,6/2,6/3,6/4,2/0,2/5,5/1,5/6,6/2,6/3,6/4,0/4,0/3,0/2}{0/0,0/5,1/0,1/5,2/1,2/6,3/1,3/6,4/1,4/6,5/2,5/7}{0.500000000000000/-0.350000000000000/1,0.500000000000000/5.35000000000000/1,1.50000000000000/-0.350000000000000/2,1.50000000000000/5.35000000000000/2,2.25000000000000/0.500000000000000/3,1.75000000000000/5.50000000000000/3,2.50000000000000/0.650000000000000/4,2.50000000000000/6.35000000000000/4,3.50000000000000/0.650000000000000/5,3.50000000000000/6.35000000000000/5,4.50000000000000/0.650000000000000/6,4.50000000000000/6.35000000000000/6,5.25000000000000/1.50000000000000/7,4.75000000000000/6.50000000000000/7,5.50000000000000/1.65000000000000/8,5.50000000000000/7.35000000000000/8,6.35000000000000/2.50000000000000/9,-.35000000000000/2.50000000000000/-2,6.35000000000000/3.50000000000000/10,-.35000000000000/3.50000000000000/-1,6.35000000000000/4.50000000000000/11,-.35000000000000/4.50000000000000/0}{0/4,2/4,0/3,2/3,5/3,0/2,1/2,3/2,4/1,0/1,4/5,3/5,4/4,4/2}{1/4,3/4,5/4,1/3,3/3,4/3,2/2,5/2,2/1,3/1,5/6,0/0,1/0,2/5,5/5,1/1}&

\begin{tikzpicture}[thick,baseline={(0,-1)}]
\def \number {11}
\def \radius {17mm}
\def \degree {360/\number}

\node[draw,circle,inner sep = 2pt] at ({\degree * -(1 -1)}:\radius) (1) {$1$};
\node[draw,circle,inner sep = 2pt] at ({\degree * -(2 -1)}:\radius) (2) {$2$};
\node[draw,circle,inner sep = 2pt] at ({\degree * -(3 -1)}:\radius) (3) {$3$};
\node[draw,circle,inner sep = 2pt] at ({\degree * -(4 -1)}:\radius) (4) {$4$};
\node[draw,circle,inner sep = 2pt] at ({\degree * -(5 -1)}:\radius) (5) {$5$};
\node[draw,circle,inner sep = 2pt] at ({\degree * -(6 -1)}:\radius) (6) {$6$};
\node[draw,circle,inner sep = 2pt] at ({\degree * -(7 -1)}:\radius) (7) {$7$};
\node[draw,circle,inner sep = 2pt] at ({\degree * -(8 -1)}:\radius) (8) {$8$};
\node[draw,circle,inner sep = 2pt] at ({\degree * -(9 -1)}:\radius) (9) {$9$};
\node[draw,circle,inner sep = 1pt] at ({\degree * -(10 -1)}:\radius) (10) {$10$};
\node[draw,circle,inner sep = 1pt] at ({\degree * -(11 -1)}:\radius) (11) {$11$};

\draw (1) -- (11);
\draw (1) to[out=60,in=-40] (11);
\draw (1) -- (4);
\draw (1) -- (8);
\draw (2) -- (10);
\draw (2) -- (3);
\draw (3) to[out=150,in=20] (5);
\draw (3) -- (9);
\draw (3) -- (4);
\draw (4) -- (5);
\draw (5) -- (6);
\draw (7) to[out=30,in=230] (9);
\draw (9) -- (10);
\draw (10) -- (11);

\end{tikzpicture}
\end{tabular}
\caption{An $(f,P)$-affine Deogram with its associated Deograph for $f=[8, 5, 7, 9, 11, 14, 12, 10, 17, 13, 15]$ and $P=\text{RRURRRURUUU}$.}
\label{fig:deograph}
\end{figure}
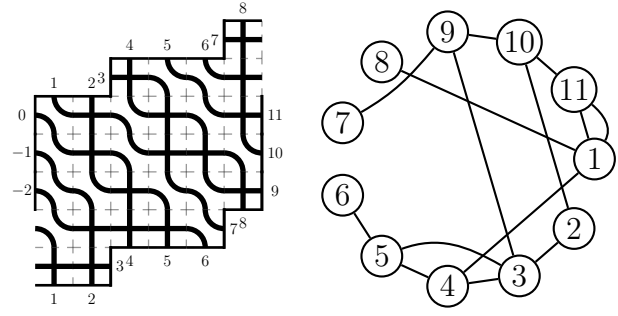

We prove the following fact, which will be useful for analyzing the combinatorial moves on affine Deograms in Section \ref{subsec:moves}.

\begin{proposition}\label{prop:forest}
For $f\in\Bkn$, $P\in\Pkn$, and $A\in\affdeomax{f}{P}$, the Deograph $G(A)$ is a forest. Moreover, its connected components are precisely the supports of the cycles of $\overline f$.
\end{proposition}

\begin{proof}
Let $m$ be the number of elbows of $A$ in a representative region. Since $A$ is maximal, $m=n-c(\overline f).$

Order the elbows according to the reading order of the boxes, starting from the boundary path $P$. Begin with the all-crossing filling and replace the tiles at these positions by elbows, one at a time, in this order. Immediately before a given elbow is inserted, all boxes preceding it agree with $A$. Thus the two strands entering that box have the same labels as the two strands forming the corresponding elbow of $A$.

Replacing the crossing at that box by an elbow changes the strand permutation by the transposition of these two labels; modulo $n$, this is the transposition of their residues. The all-crossing filling has strand permutation $\tau_P=t_{P_U}$, whose reduction modulo $n$ is the identity. Hence $\overline f$ lies in the subgroup of $S_n$ generated by the transpositions corresponding to the edges of $G(A)$.

Let $d$ be the number of connected components of $G(A)$. Every such transposition preserves each connected component, and hence $\overline f$ preserves each connected component. Consequently, every cycle of $\overline f$ is contained in a connected component of $G(A)$, so $d\leq c(\overline f)$.

On the other hand, $G(A)$ has $n$ vertices and $m$ edges (parallel edges included), so it has at least $n-m$ connected components. Therefore $d\geq n-m=c(\overline f)$, whence $d=c(\overline f)$. Equality in the bound $d\ge n-m$ holds precisely when the multigraph is a forest, so $G(A)$ is a forest.

Finally, every cycle of $\overline f$ is contained in a connected component, and the numbers of cycles and connected components agree. Thus each connected component is exactly the support of one cycle of $\overline f$.
\end{proof}

From this, we obtain the following corollary immediately.

\begin{corollary}\label{cor:foresthelp}
Fix $f\in\Bkn$, $P\in\Pkn$, and $A\in \affdeomax{f}{P}$. If the paths labeled $i$ and $j$ modulo $n$ form an elbow, then $i$ and $j$ are in the same cycle of $\overline{f}$.
\end{corollary}

\begin{lemma}\label{lem:restricted-path-compatible}
Use the notation of Theorem~\ref{thm:decoupling_intro}. If $\affdeomax{f}{P}\neq\varnothing$, then $P_j$ has exactly $k_j$ up-steps for every $j\in[r]$.
\end{lemma}

\begin{proof}
Let $A\in\affdeomax{f}{P}$. By Corollary~\ref{cor:foresthelp}, every elbow of $A$ joins two strands whose labels belong to the same set $S_j$; equivalently, every box in which two strands with labels in different sets meet is a crossing.

Fix $j$ and let $u_j$ be the number of up-steps of $P_j$. Delete from $A$ all strands whose labels do not belong to $S_j$, suppress the corresponding boundary steps, and standardize the remaining labels. Since all boxes between strands of different sets are crossings, the result is a periodic filling $B$ of the strip determined by $P_j$, with strand permutation $f_j$.

It remains to see that any periodic filling $B$ of the strip determined by a path $Q$ with $n'$ steps, of which $u$ are up-steps, has strand permutation of average $u$. Let $g$ be that strand permutation; by the geometry of the strip, $a\le g(a)\le a+n'$ for all $a$, so no two strands with congruent labels meet in a box. Passing from the all-crossing filling of the $Q$-strip to $B$ by replacing crossings with elbows one at a time therefore multiplies the strand permutation by affine transpositions $t_{a,b}$ with $a\not\equiv b \pmod{n'}$, each of which has average $0$. Since $\tau_Q=t_{Q_U}$ has average $u$, so does $g$.

Applying this to $B$ gives $k_j=\av(f_j)=u_j$, i.e.\ $P_j\in\mathcal P_{k_j,n_j}$.
\end{proof}

\begin{definition}\label{def:coloring}
Suppose $\overline{f} = (a_1^{(1)}\dots a_{n_1}^{(1)})\dots(a_1^{(r)}\dots a_{n_r}^{(r)})$ and let $S_j$ be as in Theorem~\ref{thm:decoupling_intro}. We define the coloring map
\[
	c_{f,P}: \affdeomax{f}{P} \to \prod_{j=1}^r \affdeomax{f|_{S_j}}{P|_{S_j}}
\]
as follows. For any $A\in \affdeomax{f}{P}$, color the wires in the following fashion:
\begin{enumerate}
\item Choose any elbow, and color both wires the same color. If wire $a$ forms an elbow with colored wire $b$, color $a$ the same color as $b$.
\item Continue this process until every elbow is monochromatic or uncolored, and then choose a new uncolored elbow and start this process again with a new color until every elbow is colored.
\item Color any remaining uncolored straight path with a new color until all wires have a color.
\item For each color, restrict $A$ to boxes with wires of the same color to obtain a product of affine Deograms. The result is $c_{f,P}(A)$. See Figure \ref{fig:coloring}.
\end{enumerate}
\end{definition}

\begin{figure}
\begin{tabular}{ccc}

\scalebox{\sclbx}{
\begin{tikzpicture}[xscale=0.5,yscale=0.5,baseline = {(0,1)}]
\foreach \x/\y in {3/0,3/3,5/1,5/2,0/1,0/2}{
  \draw[line width=1pt] (\x,\y-0.035) -- (\x,\y+1.035);
}
\foreach \x/\y in {0/0,0/3,1/0,1/3,2/0,2/3,3/1,3/4,4/1,4/4}{
  \draw[line width=1pt] (\x,\y) -- (\x+1,\y);
}
\foreach \x/\y/\z in {0.500000000000000/-0.2500000000000/1,0.500000000000000/3.250000000000/1,1.50000000000000/-0.2500000000000/2,1.50000000000000/3.250000000000/2,2.50000000000000/-0.2500000000000/3,2.50000000000000/3.250000000000/3,3.250000000000/0.500000000000000/4,2.8000000000000/3.50000000000000/4,3.50000000000000/0.650000000000000/5,3.50000000000000/4.250000000000/5,4.50000000000000/0.650000000000000/6,4.50000000000000/4.250000000000/6,5.250000000000/1.50000000000000/7,5.250000000000/2.50000000000000/8,-0.25000000000000/2.50000000000000/0,-0.5000000000000/1.50000000000000/-1}{
  \node[scale=0.50,inner sep=1pt] at (\x,\y) {$\z$};
}
\foreach \x/\y in {0/0}{
  \draw[line width=0.5pt,dashed,opacity=\gridop] (\x,\y) -- (\x+1,\y);
  \draw[line width=0.5pt,dashed,opacity=\gridop] (\x+1,\y) -- (\x+1,\y+1);
  \draw[blue,line width=2pt,opacity=\op] (\x+0.5,\y) to[out=90,in=0] (\x,\y+0.5);
  \draw[blue,line width=2pt,opacity=\op] (\x+0.5,\y+1) to[out=-90,in=180] (\x+1,\y+0.5);
}
\foreach \x/\y in {2/1,4/1}{
  \draw[line width=0.5pt,dashed,opacity=\gridop] (\x,\y) -- (\x+1,\y);
  \draw[line width=0.5pt,dashed,opacity=\gridop] (\x+1,\y) -- (\x+1,\y+1);
  \draw[red,line width=2pt,opacity=\op] (\x+0.5,\y) to[out=90,in=0] (\x,\y+0.5);
  \draw[red,line width=2pt,opacity=\op] (\x+0.5,\y+1) to[out=-90,in=180] (\x+1,\y+0.5);
}
\foreach \x/\y in {1/2,3/2}{
  \draw[line width=0.5pt,dashed,opacity=\gridop] (\x,\y) -- (\x+1,\y);
  \draw[line width=0.5pt,dashed,opacity=\gridop] (\x+1,\y) -- (\x+1,\y+1);
  \draw[green!80!black,line width=2pt,opacity=\op] (\x+0.5,\y) to[out=90,in=0] (\x,\y+0.5);
  \draw[green!80!black,line width=2pt,opacity=\op] (\x+0.5,\y+1) to[out=-90,in=180] (\x+1,\y+0.5);
}
\foreach \x/\y in {2/0,4/3}{
  \draw[line width=0.5pt,dashed,opacity=\gridop] (\x,\y) -- (\x+1,\y);
  \draw[line width=0.5pt,dashed,opacity=\gridop] (\x+1,\y) -- (\x+1,\y+1);
  \draw[red,line width=2pt,opacity=\op] (\x+0.5,\y) -- (\x+0.5,\y+1);
  \draw[blue,line width=2pt,opacity=\op] (\x,\y+0.5) -- (\x+1,\y+0.5);
}
\foreach \x/\y in {0/1}{
  \draw[line width=0.5pt,dashed,opacity=\gridop] (\x,\y) -- (\x+1,\y);
  \draw[line width=0.5pt,dashed,opacity=\gridop] (\x+1,\y) -- (\x+1,\y+1);
  \draw[blue,line width=2pt,opacity=\op] (\x+0.5,\y) -- (\x+0.5,\y+1);
  \draw[red,line width=2pt,opacity=\op] (\x,\y+0.5) -- (\x+1,\y+0.5);
}
\foreach \x/\y in {2/2,4/2}{
  \draw[line width=0.5pt,dashed,opacity=\gridop] (\x,\y) -- (\x+1,\y);
  \draw[line width=0.5pt,dashed,opacity=\gridop] (\x+1,\y) -- (\x+1,\y+1);
  \draw[red,line width=2pt,opacity=\op] (\x+0.5,\y) -- (\x+0.5,\y+1);
  \draw[green!80!black,line width=2pt,opacity=\op] (\x,\y+0.5) -- (\x+1,\y+0.5);
}
\foreach \x/\y in {1/1,3/1}{
  \draw[line width=0.5pt,dashed,opacity=\gridop] (\x,\y) -- (\x+1,\y);
  \draw[line width=0.5pt,dashed,opacity=\gridop] (\x+1,\y) -- (\x+1,\y+1);
  \draw[green!80!black,line width=2pt,opacity=\op] (\x+0.5,\y) -- (\x+0.5,\y+1);
  \draw[red,line width=2pt,opacity=\op] (\x,\y+0.5) -- (\x+1,\y+0.5);
}
\foreach \x/\y in {0/2}{
  \draw[line width=0.5pt,dashed,opacity=\gridop] (\x,\y) -- (\x+1,\y);
  \draw[line width=0.5pt,dashed,opacity=\gridop] (\x+1,\y) -- (\x+1,\y+1);
  \draw[blue,line width=2pt,opacity=\op] (\x+0.5,\y) -- (\x+0.5,\y+1);
  \draw[green!80!black,line width=2pt,opacity=\op] (\x,\y+0.5) -- (\x+1,\y+0.5);
}
\foreach \x/\y in {1/0,3/3}{
  \draw[line width=0.5pt,dashed,opacity=\gridop] (\x,\y) -- (\x+1,\y);
  \draw[line width=0.5pt,dashed,opacity=\gridop] (\x+1,\y) -- (\x+1,\y+1);
  \draw[green!80!black,line width=2pt,opacity=\op] (\x+0.5,\y) -- (\x+0.5,\y+1);
  \draw[blue,line width=2pt,opacity=\op] (\x,\y+0.5) -- (\x+1,\y+0.5);
}
\end{tikzpicture}
}
&
$\xrightarrow{c_{f,P}}$
&
\scalebox{\sclbx}{
\begin{tikzpicture}[xscale=0.5,yscale=0.5,baseline = {(0,-0.5)}]
\foreach \x/\y in {0/0,1/0}{
  \draw[line width=1pt] (\x,\y-0.035) -- (\x,\y+1.035);
}
\foreach \x/\y in {0/0,0/1}{
  \draw[line width=1pt] (\x,\y) -- (\x+1,\y);
}
\foreach \x/\y/\z in {0.500000000000000/-0.2500000000000/1,0.500000000000000/1.250000000000/1,1.250000000000/0.500000000000000/4,-0.5000000000000/0.50000000000000/-4}{
  \node[scale=0.50,inner sep=1pt] at (\x,\y) {$\z$};
}
\foreach \x/\y in {0/0}{
  \draw[line width=0.5pt,dashed,opacity=\gridop] (\x,\y) -- (\x+1,\y);
  \draw[line width=0.5pt,dashed,opacity=\gridop] (\x+1,\y) -- (\x+1,\y+1);
  \draw[blue,line width=2pt,opacity=\op] (\x+0.5,\y) to[out=90,in=0] (\x,\y+0.5);
  \draw[blue,line width=2pt,opacity=\op] (\x+0.5,\y+1) to[out=-90,in=180] (\x+1,\y+0.5);
}

\end{tikzpicture}
}
\scalebox{\sclbx}{
\begin{tikzpicture}[xscale=0.5,yscale=0.5,baseline = {(0,3/10)}]
\foreach \x/\y in {0/0,2/0}{
  \draw[line width=1pt] (\x,\y-0.035) -- (\x,\y+1.035);
}
\foreach \x/\y in {0/0,0/1,1/0,1/1}{
  \draw[line width=1pt] (\x,\y) -- (\x+1,\y);
}
\foreach \x/\y/\z in {0.500000000000000/-0.2500000000000/3,0.500000000000000/1.250000000000/3,1.500000000000000/-0.2500000000000/6,1.500000000000000/1.250000000000/6,2.250000000000/0.500000000000000/7,-0.5000000000000/0.50000000000000/-1}{
  \node[scale=0.50,inner sep=1pt] at (\x,\y) {$\z$};
}
\foreach \x/\y in {0/0,1/0}{
  \draw[line width=0.5pt,dashed,opacity=\gridop] (\x,\y) -- (\x+1,\y);
  \draw[line width=0.5pt,dashed,opacity=\gridop] (\x+1,\y) -- (\x+1,\y+1);
  \draw[red,line width=2pt,opacity=\op] (\x+0.5,\y) to[out=90,in=0] (\x,\y+0.5);
  \draw[red,line width=2pt,opacity=\op] (\x+0.5,\y+1) to[out=-90,in=180] (\x+1,\y+0.5);
}

\end{tikzpicture}
}
\scalebox{\sclbx}{
\begin{tikzpicture}[xscale=0.5,yscale=0.5,baseline = {(0,1)}]
\foreach \x/\y in {0/0,2/0}{
  \draw[line width=1pt] (\x,\y-0.035) -- (\x,\y+1.035);
}
\foreach \x/\y in {0/0,0/1,1/0,1/1}{
  \draw[line width=1pt] (\x,\y) -- (\x+1,\y);
}
\foreach \x/\y/\z in {0.500000000000000/-0.2500000000000/2,0.500000000000000/1.250000000000/2,1.500000000000000/-0.2500000000000/5,1.500000000000000/1.250000000000/5,2.250000000000/0.500000000000000/8,-0.25000000000000/0.50000000000000/0}{
  \node[scale=0.50,inner sep=1pt] at (\x,\y) {$\z$};
}
\foreach \x/\y in {0/0,1/0}{
  \draw[line width=0.5pt,dashed,opacity=\gridop] (\x,\y) -- (\x+1,\y);
  \draw[line width=0.5pt,dashed,opacity=\gridop] (\x+1,\y) -- (\x+1,\y+1);
  \draw[green!80!black,line width=2pt,opacity=\op] (\x+0.5,\y) to[out=90,in=0] (\x,\y+0.5);
  \draw[green!80!black,line width=2pt,opacity=\op] (\x+0.5,\y+1) to[out=-90,in=180] (\x+1,\y+0.5);
}
\end{tikzpicture}
}
\end{tabular}

\caption{Example of the coloring map.}

\label{fig:coloring}
\end{figure}
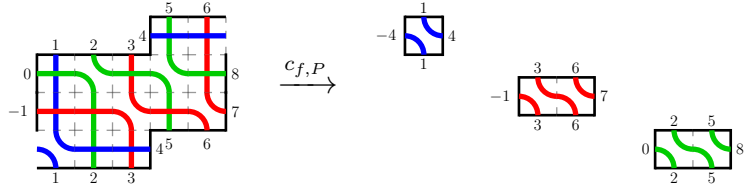

To prove Theorem~\ref{thm:decoupling_intro}, we show $c_{f,P}$ is a well-defined bijection.

\begin{proof}[Proof of Theorem~\ref{thm:decoupling_intro}]
If $P_j$ does not have exactly $k_j$ up-steps for some $j$, then Lemma~\ref{lem:restricted-path-compatible} gives $\affdeomax{f}{P}=\varnothing$. We may therefore assume that $P_j\in\mathcal P_{k_j,n_j} \qquad\text{for every }j\in[r]$.

We first verify that the coloring map of Definition~\ref{def:coloring} has the stated codomain. Let $A\in\affdeomax{f}{P}$ and write $c_{f,P}(A)=(A_1,\dots,A_r)$. By construction, $A_j$ has boundary path $P_j$ and strand permutation $f_j$. Moreover, deleting strands of other colors does not change the number of crossings between any fixed pair of remaining strands. Hence the distinguished condition for $A$ implies the distinguished condition for $A_j$.

By Proposition~\ref{prop:forest}, the component of $G(A)$ with vertex set $S_j\cap[n]$ is a tree on $n_j$ vertices. It therefore has $n_j-1$ edges. These edges are precisely the elbows of $A_j$, so $A_j$ has $n_j-1$ elbows. Since $\overline{f_j}$ is a single cycle, $n_j-1=n_j-c(\overline{f_j}),$
and therefore $A_j\in\affdeomax{f_j}{P_j}$. Thus $c_{f,P}$ is well-defined.

We now construct its inverse. Given $(A_1,\dots,A_r) \in \prod_{j=1}^r\affdeomax{f_j}{P_j}$, first replace the labels of $A_j$ so that its strands are labeled by the elements of $S_j$. Interleave these diagrams along the boundary path $P$ as follows. Scan the boxes of the ambient periodic strip in reading order. At each box, the labels of the two incoming strands are already determined.

If the two labels belong to different sets $S_j$, place a crossing. If they both belong to $S_j$, place the tile prescribed by $A_j$ after the strands of all other colors are deleted. Denote the resulting periodic filling by $M_{f,P}(A_1,\dots,A_r)$.

By construction, the restriction of $M_{f,P}(A_1,\dots,A_r)$ to the strands labeled by $S_j$ is $A_j$, so its strand permutation agrees with $f$ on $S_j$ for every $j$. Since the sets $S_j$ partition $\Z$, the strand permutation of $M_{f,P}(A_1,\dots,A_r)$ equals $f$.

Every elbow lies within one of the diagrams $A_j$. For the two strands forming such an elbow, their crossing history in the full diagram is the same as their crossing history in $A_j$. Hence the resulting filling is distinguished. It has
\[
\sum_{j=1}^r(n_j-1) = n-r = n-c(\overline f)
\]
elbows, and is therefore maximal. Thus $M_{f,P}:\prod_{j=1}^r\affdeomax{f_j}{P_j} \longrightarrow \affdeomax{f}{P}$ is well-defined.

It is immediate from the construction that $c_{f,P}\circ M_{f,P}=\operatorname{id}$.
Conversely, if $A\in\affdeomax{f}{P}$, then Corollary~\ref{cor:foresthelp} implies that every interaction between strands belonging to different cycle supports is a crossing. Therefore $A$ is recovered uniquely from its restrictions $c_{f,P}(A)$, and hence $M_{f,P}\circ c_{f,P}=\operatorname{id}$.
\end{proof}

\begin{remark}
Normalize the identification $S_j\cong\Z$ so that $1\in\Z$ corresponds to $\min(S_j\cap\Z_{\ge1})$. Since $S_j$ is $f$-stable, for $a\in S_j$ the condition $a<1\le f(a)$ is preserved by this identification, so
\[
	\sGrI_1(f_j)=\sGrI_1(f)\cap S_j
\]
after standardizing. As the up-steps of $P_{\sGrI_1(f)}$ are indexed by $\sGrI_1(f)$, this gives
\[
\left.P_{\sGrI_1(f)}\right|_{S_j} = P_{\sGrI_1(f_j)},
\]
and in particular the hypothesis of \Cref{thm:decoupling_intro}~(2) holds for $P=P_{\sGrI_1(f)}$. Combining Theorem~\ref{thm:decoupling_intro} with Corollary~\ref{cor:affdeocorrespondence} gives a combinatorial proof of \Cref{thm:geometric}~\labelcref{decouplepart}, originally proved in \cite{gallam2024}.
\end{remark}

\subsection{Combinatorial Moves on Affine Deograms}\label{subsec:moves}

We wish to prove Theorem~\ref{thm:eqmoves} combinatorially, through the use of three moves: \textit{Foata}, \textit{Box Addition/Removal}, and \textit{Zipper}.

We first construct the \textit{Foata} bijection on affine Deograms, which is an important tool for both Box Addition/Removal and Zipper. Let $\Acal^{(u)}_{f,\tau_P}$ be the set of $u$-distinguished $f$-subwords of $\tau_P$.

\begin{lemma}\label{lemma:twistedinjection} For any $f \in \Bkn$ and $i\in[n]$, we have an injection
\[
	F_i: \affdeo{f}{P}^{(s_i)} \to\affdeo{f}{P}.
\]
\end{lemma}

\begin{proof}
Let $A\in\affdeo{f}{P}^{(s_i)} $ and define $\mathbf{S}_A := (S_1,S_2,\dots,S_r)$, with $S_j\in\{E,C\}$, the sequence of elbows and crossings of $i$ and $i+1$ in $A$ (ordering from northwest to southeast). Notice that $S_1=C$ for every $A\in \affdeo{f}{P}^{(s_i)}$. 

We define $F_i(A)$ as follows: 
\begin{enumerate}[label=\arabic*.]
\item Add a bar in front of each $S_j$ such that $S_j=C$. 
\item In each section, reverse the string. Label the resulting string $\tilde{\mathbf{S}}_A$.
\item Construct an affine Deogram $\tilde{A}$ identical to $A$, but with the new sequence of crossings and elbows for the paths labeled $i$ and $i+1$ corresponding to $\tilde{\mathbf{S}}_A$.
\end{enumerate}

Then, $\tilde{A}\in\affdeo{f}{P}$, as for any elbow in $A$, if $j_1$ is below $j_2$, then either $j_1 = i+1$ and $j_2=i$ or $j_1<j_2$. The map on $\mathbf{S}_A$ ensures that we swap the condition for elbows labeled by $i$ and $i+1$. Set $F_i(A)=\tilde{A}$. From the construction, it is clear that $F_i$ is injective.
\end{proof}

We observe that for any $A\in \affdeo{f}{P}^{(s_i)}$, $\mathbf{S}_{F_i(A)}$ has $S_r=C$.

\begin{corollary}\label{cor:twistedbij}
If $f>s_if$, $F_i$ is a bijection.
\end{corollary}
\begin{proof}
Since $f>s_if$, $\#\{S_j\in \mathbf{S}_A\,|\,S_j=C\}$ is odd, so $S_r=C$ for every $A\in\affdeo{f}{P}$.
\end{proof}

\begin{remark}
We call this map the Foata bijection due to its similarity to the maps $\gamma_x$ introduced by Foata \cite{foata}, used to bijectively prove that the major index and inversion number are equidistributed over $S_n$.
\end{remark}

In the case $f < s_if$, we have a stronger injection, now from a disjoint union.
\begin{corollary}\label{cor:twistedbijcase2}
If $f < s_if$, we have an injection
\[
	U_i: \affdeo{f}{P}^{(s_i)} \sqcup \affdeo{s_if}{P} \to \affdeo{f}{P}.
\]
\end{corollary}
\begin{proof}
By Lemma \ref{lemma:twistedinjection}, we know $F_i: \affdeo{f}{P}^{(s_i)} \to  \affdeo{f}{P}$ is an injection, and we see that for all $A\in \affdeo{f}{P}^{(s_i)}$, $F_i(A)$ has $S_r = C$. 

We construct a map $G_i: \affdeo{s_if}{P} \to \affdeo{f}{P}$. As $s_if>f$, we note that for all $A\in\affdeo{s_if}{P}$, $\#\{S_j\in \mathbf{S}_A\,|\,S_j=C\}$ is odd, and in particular, $S_r = C$. We define $G_i$ by $A\mapsto \tilde{A}$, where $\mathbf{S}_{\tilde{A}} = (S_1,S_2,\dots,S_{r-1},E)$ and $\tilde{A}$ is otherwise identical to $A$. By the same argument as Lemma \ref{lemma:twistedinjection}, $\tilde{A}$ remains distinguished and thus $\tilde{A}\in\affdeo{f}{P}$, with $S_r=E$.
\end{proof}

We first state a lemma to help aid our analysis for the future combinatorial bijections.

\begin{lemma}\label{lemma:cornerbox}
	Let $P\in\Pkn$ with $P_i=\mathrm R$, $P_{i+1}=\mathrm U$. Then the box of $B_P$ at the corner formed by the steps $i,i+1$ is $\preceq$-maximal, and the two strands meeting in it are $i$ and $i+1$; hence $\mathbf{\tau_P}$ may be chosen with last letter $s_i$. Dually, if $P_i=\mathrm U$ and $P_{i+1}=\mathrm R$, then this box is $\preceq$-minimal and $\mathbf{\tau_P}$ may be chosen with first letter $s_i$. Removing this box identifies $B_P\setminus\{b\}$ with $B_{s_iP}\setminus\{b'\}$, where $b'$ is the corresponding corner box of $s_iP$.
\end{lemma}

We may now define Box Addition/Removal on affine Deograms. See Figure \ref{fig:boxaddremove} for an example.

\begin{lemma}\label{lemma:boxchangesbij}
For $f,s_if,fs_i,s_ifs_i\in \Bkn$ with $s_if>f$ and $P,s_iP\in\Pkn$ with $P_i = R$ and $P_{i+1}=U$, we have, for every $r$, bijections
\[
	\Acal^r_{f,P} \cong\begin{cases} \Acal^{r}_{s_ifs_i,s_iP}, & fs_i<f;\\[2pt]  \Acal^{r-1}_{s_if,s_iP} \sqcup \Acal^{r}_{s_ifs_i,s_iP},& fs_i>f.\end{cases}
\]
\end{lemma}

\begin{proof}
Throughout this proof, we let $\mathbf{f}$ denote a distinguished $f$-subword in $\mathbf{\tau_P}$ and $\mathbf{sfs}$ denote a (possibly $s_i$-twisted) distinguished $s_ifs_i$-subword in $\mathbf{\tau_{s_iP}}$.

Suppose $fs_i<f$. Then any $\mathbf{f}\in\Acal^r_{f,P}$, cannot satisfy $\mathbf{f}_{(k(n-k)-1)}= \mathbf{f}_{(k(n-k))}=f$, so it must satisfy $\mathbf{f}_{(k(n-k)-1)}=fs_i$. This gives a bijection $\Acal^r_{f,P}\cong\Acal^{(s_i),r}_{s_ifs_i,s_iP}$, where we require $\mathbf{sfs}_{(1)}=s_i$. However, applying Corollary \ref{cor:twistedbij}, as we must have $s_ifs_i>fs_i$ as well, gives the desired bijection.

Suppose $fs_i>f$. Then, $s_ifs_i>s_if>f$ and $s_ifs_i>fs_i>f$. Note that for every $\mathbf{f}\in\Acal^r_{f,P}$, either $\mathbf{f}_{(k(n-k)-1)} = f$, in which case $\mathbf{f}_{(k(n-k))} = f$, or $\mathbf{f}_{(k(n-k)-1)} = fs_i$, in which case $\mathbf{f}_{(k(n-k))} = f$ as well. This provides bijections $\Acal^r_{f,P}\cong \Acal^{(s_i),r-1}_{s_if,s_iP}\sqcup \Acal^{(s_i),r}_{s_ifs_i,s_iP}$, where we require $\mathbf{sf}_{(1)}=s_i$ and $\mathbf{sfs}_{(1)}=s_i$. However, since $s_ifs_i>fs_i$ and $s_if>f$, applying Corollary \ref{cor:twistedbij} gives the desired bijection. 
\end{proof}

\begin{figure}
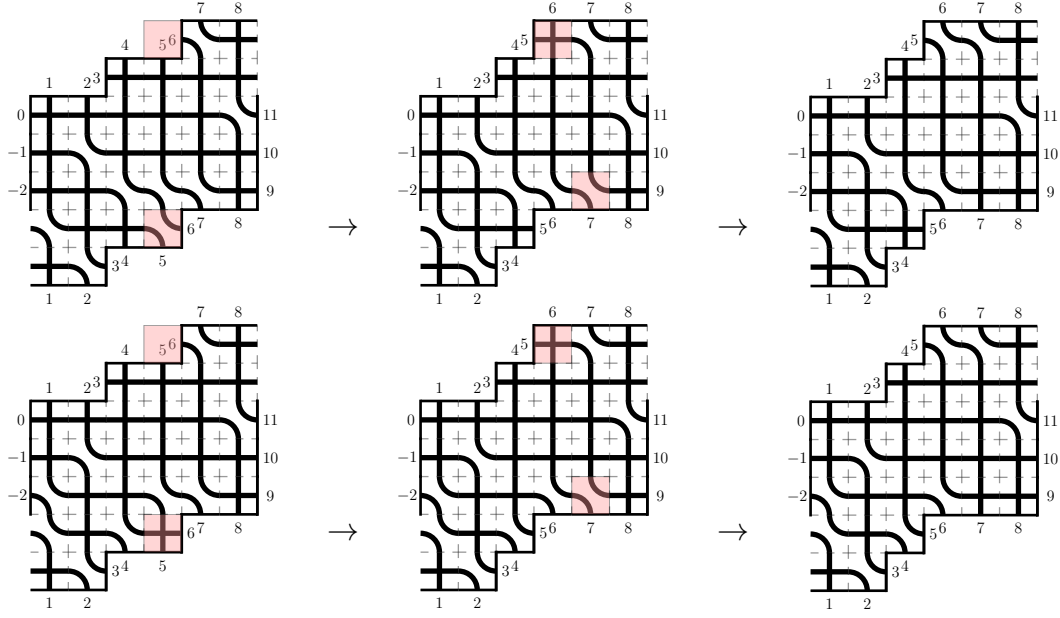

\centering

\begin{tabular}{ccccc}

\deogcolored{2/0,2/5,4/1,4/6,6/2,6/3,6/4,2/0,2/5,4/1,4/6,6/2,6/3,6/4,0/4,0/2,0/3}{0/0,0/5,1/0,1/5,2/1,2/6,3/1,3/6,4/2,4/7,5/2,5/7}{0.500000000000000/-0.350000000000000/1,0.500000000000000/5.35000000000000/1,1.50000000000000/-0.350000000000000/2,1.50000000000000/5.35000000000000/2,2.25000000000000/0.500000000000000/3,1.75000000000000/5.50000000000000/3,2.50000000000000/0.650000000000000/4,2.50000000000000/6.35000000000000/4,3.50000000000000/0.650000000000000/5,3.50000000000000/6.35000000000000/5,4.25000000000000/1.50000000000000/6,3.75000000000000/6.50000000000000/6,4.50000000000000/1.65000000000000/7,4.50000000000000/7.35000000000000/7,5.50000000000000/1.65000000000000/8,5.50000000000000/7.35000000000000/8,6.35000000000000/2.50000000000000/9,6.35000000000000/3.50000000000000/10,6.35000000000000/4.50000000000000/11,-.25000000000000/4.50000000000000/0,-.35000000000000/3.50000000000000/-1,-.35000000000000/2.50000000000000/-2}{5/4,1/3,2/2,3/2,4/2,0/1,3/1,1/0,4/6}{0/4,1/4,2/4,3/4,4/4,0/3,2/3,3/3,4/3,5/3,0/2,1/2,5/2,1/1,2/1,5/6,0/0,2/5,3/5,4/5,5/5}{3/1,3/6}

&
$\rightarrow$

&

\deogcolored{2/0,2/5,3/1,3/6,6/2,6/3,6/4,2/0,2/5,3/1,3/6,6/2,6/3,6/4,0/4,0/2,0/3}{0/0,0/5,1/0,1/5,2/1,2/6,3/2,3/7,4/2,4/7,5/2,5/7}{0.500000000000000/-0.350000000000000/1,0.500000000000000/5.35000000000000/1,1.50000000000000/-0.350000000000000/2,1.50000000000000/5.35000000000000/2,2.25000000000000/0.500000000000000/3,1.75000000000000/5.50000000000000/3,2.50000000000000/0.650000000000000/4,2.50000000000000/6.35000000000000/4,3.25000000000000/1.50000000000000/5,2.75000000000000/6.50000000000000/5,3.50000000000000/1.65000000000000/6,3.50000000000000/7.35000000000000/6,4.50000000000000/1.65000000000000/7,4.50000000000000/7.35000000000000/7,5.50000000000000/1.65000000000000/8,5.50000000000000/7.35000000000000/8,6.35000000000000/2.50000000000000/9,6.35000000000000/3.50000000000000/10,6.35000000000000/4.50000000000000/11,-.25000000000000/4.50000000000000/0,-.35000000000000/3.50000000000000/-1,-.35000000000000/2.50000000000000/-2}{5/4,1/3,2/2,3/2,4/2,0/1,1/0,4/6}{0/4,1/4,2/4,3/4,4/4,0/3,2/3,3/3,4/3,5/3,0/2,1/2,5/2,1/1,2/1,5/6,0/0,2/5,3/5,4/5,5/5,3/6}{4/2,3/6}

&
$\rightarrow$

&

\deog{2/0,2/5,3/1,3/6,6/2,6/3,6/4,2/0,2/5,3/1,3/6,6/2,6/3,6/4,0/4,0/2,0/3}{0/0,0/5,1/0,1/5,2/1,2/6,3/2,3/7,4/2,4/7,5/2,5/7}{0.500000000000000/-0.350000000000000/1,0.500000000000000/5.35000000000000/1,1.50000000000000/-0.350000000000000/2,1.50000000000000/5.35000000000000/2,2.25000000000000/0.500000000000000/3,1.75000000000000/5.50000000000000/3,2.50000000000000/0.650000000000000/4,2.50000000000000/6.35000000000000/4,3.25000000000000/1.50000000000000/5,2.75000000000000/6.50000000000000/5,3.50000000000000/1.65000000000000/6,3.50000000000000/7.35000000000000/6,4.50000000000000/1.65000000000000/7,4.50000000000000/7.35000000000000/7,5.50000000000000/1.65000000000000/8,5.50000000000000/7.35000000000000/8,6.35000000000000/2.50000000000000/9,6.35000000000000/3.50000000000000/10,6.35000000000000/4.50000000000000/11,-.25000000000000/4.50000000000000/0,-.35000000000000/3.50000000000000/-1,-.35000000000000/2.50000000000000/-2}{5/4,1/3,2/2,3/2,0/1,1/0,4/6,3/6}{0/4,1/4,2/4,3/4,4/4,0/3,2/3,3/3,4/3,5/3,0/2,1/2,5/2,1/1,2/1,5/6,0/0,2/5,3/5,4/5,5/5,4/2}

\\

\deogcolored{2/0,2/5,4/1,4/6,6/2,6/3,6/4,2/0,2/5,4/1,4/6,6/2,6/3,6/4,0/4,0/2,0/3}{0/0,0/5,1/0,1/5,2/1,2/6,3/1,3/6,4/2,4/7,5/2,5/7}{0.500000000000000/-0.350000000000000/1,0.500000000000000/5.35000000000000/1,1.50000000000000/-0.350000000000000/2,1.50000000000000/5.35000000000000/2,2.25000000000000/0.500000000000000/3,1.75000000000000/5.50000000000000/3,2.50000000000000/0.650000000000000/4,2.50000000000000/6.35000000000000/4,3.50000000000000/0.650000000000000/5,3.50000000000000/6.35000000000000/5,4.25000000000000/1.50000000000000/6,3.75000000000000/6.50000000000000/6,4.50000000000000/1.65000000000000/7,4.50000000000000/7.35000000000000/7,5.50000000000000/1.65000000000000/8,5.50000000000000/7.35000000000000/8,6.35000000000000/2.50000000000000/9,6.35000000000000/3.50000000000000/10,6.35000000000000/4.50000000000000/11,-.25000000000000/4.50000000000000/0,-.35000000000000/3.50000000000000/-1,-.35000000000000/2.50000000000000/-2}{5/4,1/3,0/2,3/2,4/2,0/1,2/1,1/0,4/6}{0/4,1/4,2/4,3/4,4/4,0/3,2/3,3/3,4/3,5/3,1/2,2/2,5/2,1/1,3/1,5/6,0/0,2/5,3/5,4/5,5/5}{3/1,3/6}
&
$\rightarrow$
&

\deogcolored{2/0,2/5,3/1,3/6,6/2,6/3,6/4,2/0,2/5,3/1,3/6,6/2,6/3,6/4,0/4,0/2,0/3}{0/0,0/5,1/0,1/5,2/1,2/6,3/2,3/7,4/2,4/7,5/2,5/7}{0.500000000000000/-0.350000000000000/1,0.500000000000000/5.35000000000000/1,1.50000000000000/-0.350000000000000/2,1.50000000000000/5.35000000000000/2,2.25000000000000/0.500000000000000/3,1.75000000000000/5.50000000000000/3,2.50000000000000/0.650000000000000/4,2.50000000000000/6.35000000000000/4,3.25000000000000/1.50000000000000/5,2.75000000000000/6.50000000000000/5,3.50000000000000/1.65000000000000/6,3.50000000000000/7.35000000000000/6,4.50000000000000/1.65000000000000/7,4.50000000000000/7.35000000000000/7,5.50000000000000/1.65000000000000/8,5.50000000000000/7.35000000000000/8,6.35000000000000/2.50000000000000/9,6.35000000000000/3.50000000000000/10,6.35000000000000/4.50000000000000/11,-.25000000000000/4.50000000000000/0,-.35000000000000/3.50000000000000/-1,-.35000000000000/2.50000000000000/-2}{5/4,1/3,0/2,3/2,4/2,0/1,2/1,1/0,4/6}{0/4,1/4,2/4,3/4,4/4,0/3,2/3,3/3,4/3,5/3,1/2,2/2,5/2,1/1,3/6,5/6,0/0,2/5,3/5,4/5,5/5}{4/2,3/6}

&
$\rightarrow$
&

\deog{2/0,2/5,3/1,3/6,6/2,6/3,6/4,2/0,2/5,3/1,3/6,6/2,6/3,6/4,0/4,0/2,0/3}{0/0,0/5,1/0,1/5,2/1,2/6,3/2,3/7,4/2,4/7,5/2,5/7}{0.500000000000000/-0.350000000000000/1,0.500000000000000/5.35000000000000/1,1.50000000000000/-0.350000000000000/2,1.50000000000000/5.35000000000000/2,2.25000000000000/0.500000000000000/3,1.75000000000000/5.50000000000000/3,2.50000000000000/0.650000000000000/4,2.50000000000000/6.35000000000000/4,3.25000000000000/1.50000000000000/5,2.75000000000000/6.50000000000000/5,3.50000000000000/1.65000000000000/6,3.50000000000000/7.35000000000000/6,4.50000000000000/1.65000000000000/7,4.50000000000000/7.35000000000000/7,5.50000000000000/1.65000000000000/8,5.50000000000000/7.35000000000000/8,6.35000000000000/2.50000000000000/9,6.35000000000000/3.50000000000000/10,6.35000000000000/4.50000000000000/11,-.25000000000000/4.50000000000000/0,-.35000000000000/3.50000000000000/-1,-.35000000000000/2.50000000000000/-2}{5/4,1/3,0/2,3/2,3/6,0/1,2/1,1/0,4/6}{0/4,1/4,2/4,3/4,4/4,0/3,2/3,3/3,4/3,5/3,1/2,2/2,5/2,1/1,4/2,5/6,0/0,2/5,3/5,4/5,5/5}

\end{tabular}
\caption{Visualization of the Box Addition/Removal move for $f = [5,10,13,6,7,9,12,11,15,14,19]$, $i=5$, and $r=9$. The boxes highlighted in red are those which change from one step to the next.}
\label{fig:boxaddremove}
\end{figure}

We have a similar bijection in the case $P_i=U$ and $P_{i+1}=R$.

\begin{lemma}\label{lemma:boxchangesbijcase2}
For $f,s_if,fs_i,s_ifs_i\in \Bkn$ with $fs_i>f$ and $P,s_iP\in\Pkn$ with $P_i = U$ and $P_{i+1}=R$, we have, for every $r$, bijections
\[
	\Acal^r_{f,P} \cong\begin{cases} \Acal^{r}_{s_ifs_i,s_iP}, & s_if<f;\\[2pt]  \Acal^{r-1}_{fs_i,s_iP} \sqcup \Acal^{r}_{s_ifs_i,s_iP},& s_if>f.\end{cases}
\]
\end{lemma}
\begin{proof}

The first case is also the first case of Lemma \ref{lemma:boxchangesbij}. Suppose $s_if>f$. We split on cases depending on $\mathbf{f}_1=e$ or $\mathbf{f}_1 = s_i$. In either case, we remove the box and append a crossing, to obtain maps $\Acal^r_{f,P}\to \Acal^{r-1}_{fs_i,s_iP}\sqcup \Acal^{(s_i),r}_{s_ifs_i,s_iP}$. As $fs_i>f$, we also have $s_ifs_i>s_if$, and we must have $\mathbf{fs}_{k(n-k)}=s_i$ and $\mathbf{sfs}_{k(n-k)}=s_i$. This implies that our map is bijective. Since $s_ifs_i>fs_i$, we then apply Corollary \ref{cor:twistedbij} to obtain $\Acal^r_{f,P}\to \Acal^{r-1}_{fs_i,s_iP}\sqcup \Acal^{r}_{s_ifs_i,s_iP}$.
\end{proof}

\begin{figure}
\centering
\setlength{\tabcolsep}{1pt}
\begin{tabular}{ccc|ccc}
$s_is_{i+1}s_i$ & $\leftrightarrow$ & $s_{i+1}s_is_{i+1}$ & $s_is_{i+1}s_i$ & $\leftrightarrow$ & $s_{i+1}s_is_{i+1}$ \\\hline

\scalebox{\ybscl}{
\begin{tikzpicture}[xscale=0.5,yscale=0.5,baseline = {(0,0.2)}]
\node[anchor = west] at (-2,0.5) {$a$};
\node[anchor = north] at (-.5,2) {$b$};
\node[anchor = north] at (.5,1.8) {$c$};
\node[anchor = south] at (-.5,-2) {$1$};
\node[anchor = south] at (.5,-2) {$2$};
\node[anchor = east] at (2,0.5) {$3$};
\crosscross
\bottomcross
\end{tikzpicture}
}
&
$\leftrightarrow$

&

\scalebox{\ybscl}{
\begin{tikzpicture}[xscale=0.5,yscale=0.5,baseline = {(0,0.2)}]
\node[anchor = west] at (-2,0.5) {$a$};
\node[anchor = north] at (-.5,3) {$b$};
\node[anchor = north] at (.5,2.8) {$c$};
\node[anchor = south] at (-.5,-1) {$1$};
\node[anchor = south] at (.5,-1) {$2$};
\node[anchor = east] at (2,0.5) {$3$};
\crosscross
\topcross
\end{tikzpicture}
}
&

\scalebox{\ybscl}{
\begin{tikzpicture}[xscale=0.5,yscale=0.5,baseline = {(0,0.2)}]
\node[anchor = west] at (-2,0.5) {$a$};
\node[anchor = north] at (-.5,2) {$b$};
\node[anchor = north] at (.5,1.8) {$c$};
\node[anchor = south] at (-.5,-2) {$1$};
\node[anchor = south] at (.5,-2) {$2$};
\node[anchor = east] at (2,0.5) {$3$};
\crosscross
\bottomelbow
\end{tikzpicture}
}

&

 $\leftrightarrow$ 

&

\scalebox{\ybscl}{
\begin{tikzpicture}[xscale=0.5,yscale=0.5,baseline = {(0,0.2)}]
\node[anchor = west] at (-2,0.5) {$a$};
\node[anchor = north] at (-.5,3) {$b$};
\node[anchor = north] at (.5,2.8) {$c$};
\node[anchor = south] at (-.5,-1) {$1$};
\node[anchor = south] at (.5,-1) {$2$};
\node[anchor = east] at (2,0.5) {$3$};
\crosscross
\topelbow
\end{tikzpicture}
}

\\

\scalebox{\ybscl}{
\begin{tikzpicture}[xscale=0.5,yscale=0.5,baseline = {(0,0.2)}]
\node[anchor = west] at (-2,0.5) {$a$};
\node[anchor = north] at (-.5,2) {$b$};
\node[anchor = north] at (.5,1.8) {$c$};
\node[anchor = south] at (-.5,-2) {$1$};
\node[anchor = south] at (.5,-2) {$2$};
\node[anchor = east] at (2,0.5) {$3$};
\elbowcross
\bottomcross
\end{tikzpicture}
}

&

 $\leftrightarrow$ 

&

\scalebox{\ybscl}{
\begin{tikzpicture}[xscale=0.5,yscale=0.5,baseline = {(0,0.2)}]
\node[anchor = west] at (-2,0.5) {$a$};
\node[anchor = north] at (-.5,3) {$b$};
\node[anchor = north] at (.5,2.8) {$c$};
\node[anchor = south] at (-.5,-1) {$1$};
\node[anchor = south] at (.5,-1) {$2$};
\node[anchor = east] at (2,0.5) {$3$};
\crosselbow
\topcross
\end{tikzpicture}
}

&

\scalebox{\ybscl}{
\begin{tikzpicture}[xscale=0.5,yscale=0.5,baseline = {(0,0.2)}]
\node[anchor = west] at (-2,0.5) {$a$};
\node[anchor = north] at (-.5,2) {$b$};
\node[anchor = north] at (.5,1.8) {$c$};
\node[anchor = south] at (-.5,-2) {$1$};
\node[anchor = south] at (.5,-2) {$2$};
\node[anchor = east] at (2,0.5) {$3$};
\elbowelbow
\bottomelbow
\end{tikzpicture}
}

&
$\leftrightarrow$
&

\scalebox{\ybscl}{
\begin{tikzpicture}[xscale=0.5,yscale=0.5,baseline = {(0,0.2)}]
\node[anchor = west] at (-2,0.5) {$a$};
\node[anchor = north] at (-.5,3) {$b$};
\node[anchor = north] at (.5,2.8) {$c$};
\node[anchor = south] at (-.5,-1) {$1$};
\node[anchor = south] at (.5,-1) {$2$};
\node[anchor = east] at (2,0.5) {$3$};
\elbowelbow
\topelbow
\end{tikzpicture}
}

\\
\scalebox{\ybscl}{
\begin{tikzpicture}[xscale=0.5,yscale=0.5,baseline = {(0,0.2)}]
\node[anchor = west] at (-2,0.5) {$a$};
\node[anchor = north] at (-.5,2) {$b$};
\node[anchor = north] at (.5,1.8) {$c$};
\node[anchor = south] at (-.5,-2) {$1$};
\node[anchor = south] at (.5,-2) {$2$};
\node[anchor = east] at (2,0.5) {$3$};
\crosselbow
\bottomcross
\end{tikzpicture}
}

&
$\leftrightarrow$
&

\scalebox{\ybscl}{
\begin{tikzpicture}[xscale=0.5,yscale=0.5,baseline = {(0,0.2)}]
\node[anchor = west] at (-2,0.5) {$a$};
\node[anchor = north] at (-.5,3) {$b$};
\node[anchor = north] at (.5,2.8) {$c$};
\node[anchor = south] at (-.5,-1) {$1$};
\node[anchor = south] at (.5,-1) {$2$};
\node[anchor = east] at (2,0.5) {$3$};
\elbowcross
\topcross
\end{tikzpicture}
}

&

\scalebox{\ybscl}{
\begin{tikzpicture}[xscale=0.5,yscale=0.5,baseline = {(0,0.2)}]
\node[anchor = west] at (-2,0.5) {$a$};
\node[anchor = north] at (-.5,2) {$b$};
\node[anchor = north] at (.5,1.8) {$c$};
\node[anchor = south] at (-.5,-2) {$1$};
\node[anchor = south] at (.5,-2) {$2$};
\node[anchor = east] at (2,0.5) {$3$};
\elbowcross
\bottomelbow
\end{tikzpicture}
}
&
 $\leftrightarrow$ 
&

\scalebox{\ybscl}{
\begin{tikzpicture}[xscale=0.5,yscale=0.5,baseline = {(0,0.2)}]
\node[anchor = west] at (-2,0.5) {$a$};
\node[anchor = north] at (-.5,3) {$b$};
\node[anchor = north] at (.5,2.8) {$c$};
\node[anchor = south] at (-.5,-1) {$1$};
\node[anchor = south] at (.5,-1) {$2$};
\node[anchor = east] at (2,0.5) {$3$};
\elbowelbow
\topcross
\end{tikzpicture}
}

\text{ or } 

\scalebox{\ybscl}{
\begin{tikzpicture}[xscale=0.5,yscale=0.5,baseline = {(0,0.2)}]
\node[anchor = west] at (-2,0.5) {$a$};
\node[anchor = north] at (-.5,3) {$b$};
\node[anchor = north] at (.5,2.8) {$c$};
\node[anchor = south] at (-.5,-1) {$1$};
\node[anchor = south] at (.5,-1) {$2$};
\node[anchor = east] at (2,0.5) {$3$};
\elbowcross
\topelbow
\end{tikzpicture}
}

\\

\scalebox{\ybscl}{\begin{tikzpicture}[xscale=0.5,yscale=0.5,baseline = {(0,0.2)}]
\node[anchor = west] at (-2,0.5) {$a$};
\node[anchor = north] at (-.5,2) {$b$};
\node[anchor = north] at (.5,1.8) {$c$};
\node[anchor = south] at (-.5,-2) {$1$};
\node[anchor = south] at (.5,-2) {$2$};
\node[anchor = east] at (2,0.5) {$3$};
\elbowelbow
\bottomcross
\end{tikzpicture}
}

\text{ or }

\scalebox{\ybscl}{
\begin{tikzpicture}[xscale=0.5,yscale=0.5,baseline = {(0,0.2)}]
\node[anchor = west] at (-2,0.5) {$a$};
\node[anchor = north] at (-.5,2) {$b$};
\node[anchor = north] at (.5,1.8) {$c$};
\node[anchor = south] at (-.5,-2) {$1$};
\node[anchor = south] at (.5,-2) {$2$};
\node[anchor = east] at (2,0.5) {$3$};
\crosselbow
\bottomelbow
\end{tikzpicture}
}

&
$\leftrightarrow$
&

\scalebox{\ybscl}{\begin{tikzpicture}[xscale=0.5,yscale=0.5,baseline = {(0,0.2)}]
\node[anchor = west] at (-2,0.5) {$a$};
\node[anchor = north] at (-.5,3) {$b$};
\node[anchor = north] at (.5,2.8) {$c$};
\node[anchor = south] at (-.5,-1) {$1$};
\node[anchor = south] at (.5,-1) {$2$};
\node[anchor = east] at (2,0.5) {$3$};
\crosselbow
\topelbow
\end{tikzpicture}
}

&

\end{tabular}

\caption{Local moves for the braid relation.}
\end{figure}
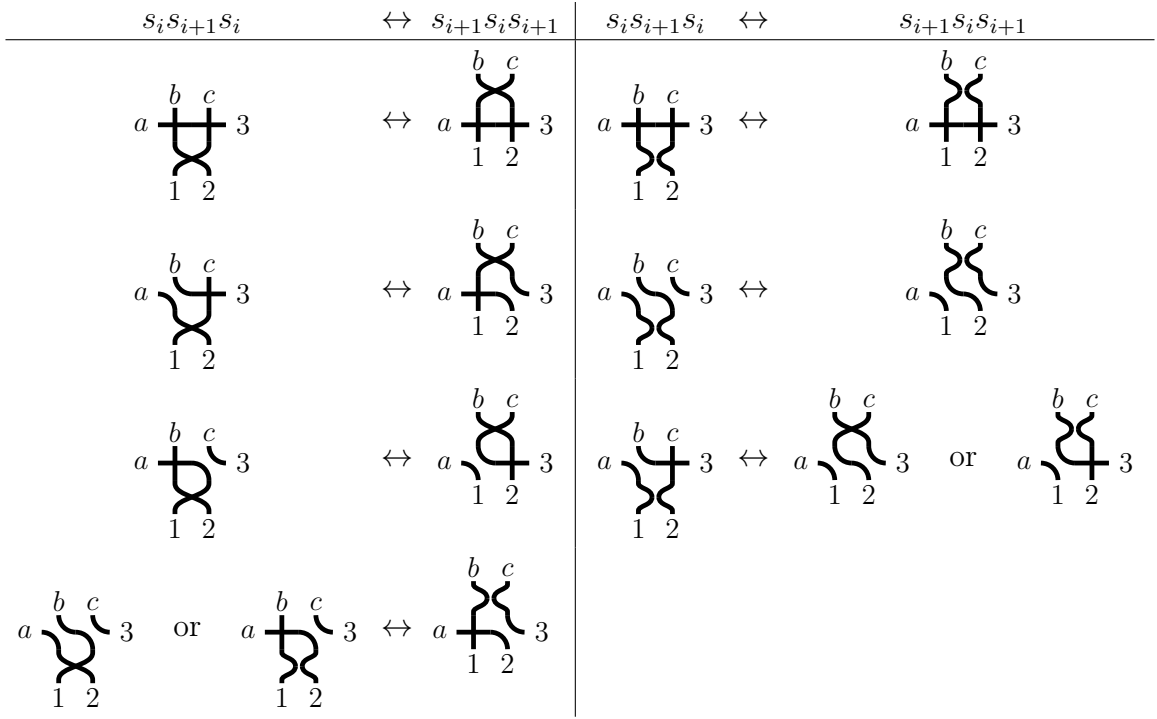

We define the Zipper bijection using the local moves coming from the braid relation $s_{i+1}s_is_{i+1}\mapsto s_is_{i+1}s_i$. 
\begin{lemma}\label{lemma:localmoves}
If $\mathbf{w}$ and $\mathbf{u}$ are reduced words of $w$ connected by braid moves, then for any $v\leq w$, we have a bijection $\mathcal{A}^r_{v,\mathbf{w}}\to \mathcal{A}^r_{v,\mathbf{u}}$.
\end{lemma}

\begin{proof}
Let $\mathbf{w}$ and $\mathbf{u}$ be two reduced words of $w$ connected by one braid move. We may assume $\mathbf{w} = (\dots,s_{i+1},s_i,s_{i+1},\dots)$ and $\mathbf{u} = (\dots,s_i,s_{i+1},s_i,\dots)$.

There are only two non-trivial cases; we highlight the proof for one as it is analogous for the other. The distinguished subword $\mathbf{v}=(\dots,e,s_i,e,\dots)$ may be mapped to $\mathbf{v'}=(\dots,s_i,e,e,\dots)$ or $\mathbf{v''}=(\dots,e,e,s_i,\dots)$, where $\mathbf{v}$ is a subword of $\mathbf{w}$ and $\mathbf{v'},\mathbf{v''}$ are subwords of $\mathbf{u}$. If $\mathbf{v'}$ is not distinguished, $v'_{(j+1)}=v'_{(j+2)}> v'_{(j+1)}s_i$. So, $\mathbf{v''}$ must be distinguished as we have
\[
v''_{(j)}=v''_{(j-1)}=v'_{(j+1)}s_i < v'_{(j+1)}=v''_{(j-1)}s_i.
\]
If instead $\mathbf{v'}$ is distinguished, $v'_{(j+1)}=v'_{(j+2)}< v'_{(j+1)}s_i$. By the same logic as above, $\mathbf{v''}$ is not distinguished as
\[
v''_{(j)}=v''_{(j-1)}=v'_{(j+1)}s_i > v'_{(j+1)}=v''_{(j-1)}s_i.
\]
\end{proof}

Using Lemma \ref{lemma:localmoves}, in the case $s_iP=P$, we may add \textit{phantom} crossings and elbows to our affine Deograms in the $i$th and $i+1$st indices and apply our local braid moves to obtain the following bijection.

\begin{lemma}\label{lemma:zipperfull}
For $fs_i,s_if,f\in \Bkn$ with $s_if,fs_i>f$ and $s_iP=P$, we have a bijection
\[
	\Phi: \Acal^r_{fs_i,P} \sqcup \Acal^r_{f,P} \to\Acal^{(s_i),r}_{s_if,P} \sqcup \Acal^r_{f,P}.
\]
\end{lemma}

\begin{corollary}\label{cor:zipperbij}
For $fs_i,s_if,f\in \Bkn$ with $s_if,fs_i>f$ and $s_iP=P$, we have a bijection
\[
	\Acal^r_{fs_i,P} \to  \Acal^r_{s_if,P}.
\]
\end{corollary}

\begin{proof}
Using Lemma \ref{lemma:zipperfull}, for any $A\in\Acal^r_{fs_i,P}$ we apply $\Phi$. If $\Phi(A)\in \Acal^{(s_i),r}_{s_if,P}$, then as $s_if>f$, we apply Corollary \ref{cor:twistedbij}. Otherwise, $\Phi(A)\in \Acal^r_{f,P}$, in which case we repeatedly apply $\Phi$ until $\Phi^m(A)\in \Acal^{(s_i),r}_{s_if,P}$ for some positive $m$.
\end{proof}

\begin{lemma}\label{lemma:doubleelbow}
Every pair of consecutive columns for a maximal affine Deogram has at most one pair of elbows in the same row.
\end{lemma}
\begin{proof}
If not, then $G(A)$ would contain a cycle, contradicting Proposition \ref{prop:forest}.
\end{proof}

\begin{corollary}
For $fs_i,s_if,f\in \Bkn$ with $s_if,fs_i>f$ and for maximal Deograms $A\in \Acal_{fs_i,P}$, we have
\[
	\Phi^m(A) \in  \Acal^{(s_i)}_{s_if,P}
\]
for some positive $m\leq2$, with $\Phi$ as defined in Lemma \ref{lemma:zipperfull}.
\end{corollary}
\begin{proof}
This follows directly from Lemma \ref{lemma:doubleelbow} and some straightforward analysis of the local moves.
\end{proof}

In the case $f>fs_i$, we may never add a phantom elbow as the wires that end at $i$ and $i+1$ cross an odd number of times in every affine Deogram. This gives us the following bijection.

\begin{lemma}\label{lemma:zipperfull2}
For $fs_i,s_if,f\in \Bkn$ with $s_if>f>fs_i$, we have a bijection
\[
	\Psi: \Acal^r_{fs_i,P} \to \Acal^{(s_i),r}_{s_if,P} \sqcup \Acal^r_{f,P}.
\]
\end{lemma}

\begin{corollary}\label{cor:zipperbij2}
For $fs_i,s_if,f\in \Bkn$ with $s_if>f>fs_i$, we have a bijection
\[
	\Acal^r_{fs_i,P} \to \Acal^{r}_{s_if,P} \sqcup \Acal^r_{f,P}.
\]
\end{corollary}
\begin{proof}
For $\Psi(A)\in \Acal^{(s_i),r}_{s_if,P}$, with $\Psi$ as in Lemma \ref{lemma:zipperfull2}, then, since $s_if>f$, we apply Corollary \ref{cor:twistedbij}. 
\end{proof}

\begin{proof}[Proof of Theorem~\ref{thm:eqmoves}]
If $P_i=R$ and $P_{i+1}=U$, we apply Lemma~\ref{lemma:boxchangesbij}. If $P_i = U$ and $P_{i+1}=R$, we apply Lemma~\ref{lemma:boxchangesbijcase2}. If $s_iP=P$, we apply Corollaries~\ref{cor:zipperbij} and \ref{cor:zipperbij2}.
\end{proof}

\section{Anti-Distinguished to Distinguished}\label{sec:antitodist}

The goal of this section is to prove the following proposition.

\begin{proposition}\label{prop:antitodist}
For any $f\in \Bkn$, $P\in\Pkn$, and any reduced word $\mathbf{f}$ of $f$, we have a bijection
\[
	\phi_{\mathbf{f}}:  \antiaffdeo{f}{P}\to\affdeo{f}{P}.
\]
\end{proposition}

We first give a precise definition of an anti-distinguished subword. For any $f$-subword $\mathbf{f}$ of $\mathbf{g} = (s_{i_1},\dots,s_{i_m})$ and $1\leq j \leq m$, let $\mathbf{f}_{\lra{j}} = f_j \dots f_m$.

\begin{definition}
For $f\in \tilde{S}_n$, we say that a $f$-subword $\mathbf{f}$ of $\mathbf{g}$ is \textit{anti-distinguished} if $\mathbf{f}_{\lra{j}}\leq s_{i_j}\mathbf{f}_{\lra{j+1}}$ for all $1\leq j\leq m-1$. 
\end{definition}

We make the following key observation.

\begin{lemma}\label{lemma:antikey}
Let $f\in \Bkn$, $P\in \Pkn$, and $\mathbf{\tau_P}$ be a reduced word for $\tau_P$. Then an $f$-subword $\mathbf{f}$ of $\mathbf{\tau_P}$ is anti-distinguished if and only if it is $f^{-1}$-distinguished.

For any $f\in\Bkn$ and $P\in\Pkn$, we have
\[
	\affdeo{f}{P}^{(f^{-1})} = \antiaffdeo{f}{P}.
\]
\end{lemma}
\begin{proof}
Let $\mathbf{f}$ be an $f$-subword of $\mathbf{\tau_P}$. Then $f = \mathbf{f}_{(j-1)}\mathbf{f}_{\lra{j}}$ for all $1\leq j \leq m$, and we obtain the chain of inequalities
\begin{align*}
\mathbf{f}_{\lra{j}}&\leq s_{i_j}\mathbf{f}_{\lra{j+1}},\\
\mathbf{f}_{\lra{j}}^{-1} & \leq \mathbf{f}_{\lra{j+1}}^{-1}s_{i_j},\\
f^{-1}\mathbf{f}_{(j-1)} & \leq f^{-1}\mathbf{f}_{(j)}s_{i_j}.
\end{align*} 
A simple case check for $f_j\in\{e,s_{i_j}\}$ shows this is equivalent to $f^{-1}\mathbf{f}_{(j)} \leq f^{-1}\mathbf{f}_{(j-1)}s_{i_j}$.
\end{proof}

\begin{remark}
Lemma~\ref{lemma:antikey} also has a simple wiring-diagram interpretation. To test whether an $f$-subword is $f^{-1}$-distinguished, prepend a reduced wiring diagram for $f^{-1}$ above the northern boundary and apply the usual distinguished rule to the concatenation. The total product is then $f^{-1}f=e$, so every pair of wires crosses an even number of times in the concatenated diagram. Hence, for any elbow, the parity of the crossings above the elbow agrees with the parity of the crossings below it. Thus the $f^{-1}$-distinguished condition, read from the northern boundary, is equivalent to the anti-distinguished condition, read from the southern boundary.
\end{remark}

To prove the anti-distinguished to distinguished bijection, we require a stronger version of Lemma \ref{lemma:twistedinjection}. The proof is nearly identical.

\begin{lemma}\label{lemma:gentwistedinjection} For any $f,g \in \Bkn$ and $i\in[n]$ with $s_ig>g$, we have an injection
\[
	F_i: \affdeo{f}{P}^{(s_ig)} \to\affdeo{f}{P}^{(g)}.
\]
\end{lemma}

\begin{proof}
Let $A\in\affdeo{f}{P}^{(s_ig)} $ and define $\mathbf{S}_A := (S_1,S_2,\dots,S_r)$, with $S_j\in\{E,C\}$, the sequence of elbows and crossings of $s_ig(i)$ and $s_ig(i+1)$ in $A$ (ordering from northwest to southeast). Since $s_ig>g$, we have $S_1=C$ for every $A\in \affdeo{f}{P}^{(s_ig)}$. 

We then proceed with the following operations: 
\begin{enumerate}[label=\arabic*.]
\item Add a bar in front of each $S_j$ such that $S_j=C$. 
\item In each section, reverse the string. Label the resulting string $\tilde{\mathbf{S}}_A$.
\item Construct an affine Deogram $\tilde{A}$ identical to $A$, but with the new sequence of crossings and elbows for the paths labeled $s_ig(i)$ and $s_ig(i+1)$ corresponding to $\tilde{\mathbf{S}}_A$.
\end{enumerate}

Then, $\tilde{A}\in\affdeo{f}{P}^{(g)}$, as for any elbow in $A$ with labeling \labelelbow, we have 
\[
	\begin{cases} j_1>j_2, & \text{ if } g^{-1}s_i(j_1)>g^{-1}s_i(j_2), \\ j_1<j_2, & \text{ otherwise.}\end{cases}
\]
The map on $\mathbf{S}_A$ ensures that we swap the condition for elbows labeled by $s_ig(i)$ and $s_ig(i+1)$ and since $i$ and $i+1$ are consecutive integers, we do not change the condition for other $(j_1,j_2)$ pairs. Set $F_i(A)=\tilde{A}$. From the construction, it is clear that $F_i$ is injective.
\end{proof}

\begin{corollary}\label{cor:gentwistedbij}
If $gf>s_igf$, $F_i$ is a bijection.
\end{corollary}
\begin{proof}
Since $gf>s_igf$, $\#\{S_j\in \mathbf{S}_A\,|\,S_j=C\}$ is odd, so $S_r=C$ for every $A\in\affdeo{f}{P}^{(g)}$.
\end{proof}

\begin{figure}
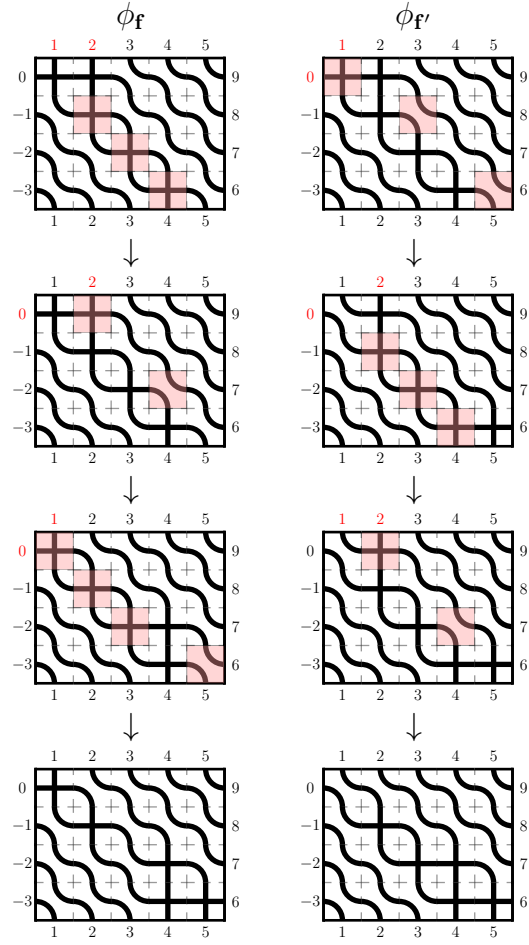

\def\arraystretch{1.3}
\begin{tabular}{cc}

$\phi_{\mathbf{f}}$ & $\phi_{\mathbf{f}'}$

\\

\deogcolored{0/0,0/1,0/2,0/3,5/0,5/1,5/2,5/3,0/0,0/1,0/2,0/3,5/0,5/1,5/2,5/3}{0/4,1/4,2/4,3/4,4/4,0/0,1/0,2/0,3/0,4/0}{-0.350000000000000/0.500000000000000/-3,-0.350000000000000/1.50000000000000/-2,-0.350000000000000/2.50000000000000/-1,-0.350000000000000/3.50000000000000/0,0.500000000000000/4.35000000000000/\textcolor{red}{1}, 1.50000000000000/4.35000000000000/\textcolor{red}{2},2.50000000000000/4.35000000000000/3,3.50000000000000/4.35000000000000/4,4.50000000000000/4.35000000000000/5,0.500000000000000/-0.300000000000000/1,1.50000000000000/-0.300000000000000/2,2.50000000000000/-0.300000000000000/3,3.50000000000000/-0.300000000000000/4,4.50000000000000/-0.300000000000000/5,5.30000000000000/0.500000000000000/6,5.30000000000000/1.50000000000000/7,5.30000000000000/2.50000000000000/8,5.30000000000000/3.50000000000000/9}{3/3,4/3,2/2,4/2,1/1,4/1,0/0,4/0,2/3,0/2,3/2,0/1,3/1,1/0,2/0}{0/3,1/3,1/2,2/1,3/0}{1/2,2/1,3/0}

&

\deogcolored{0/0,0/1,0/2,0/3,5/0,5/1,5/2,5/3,0/0,0/1,0/2,0/3,5/0,5/1,5/2,5/3}{0/4,1/4,2/4,3/4,4/4,0/0,1/0,2/0,3/0,4/0}{-0.350000000000000/0.500000000000000/-3,-0.350000000000000/1.50000000000000/-2,-0.350000000000000/2.50000000000000/-1,-0.350000000000000/3.50000000000000/\textcolor{red}{0},0.500000000000000/4.35000000000000/\textcolor{red}{1},1.50000000000000/4.35000000000000/2,2.50000000000000/4.35000000000000/3,3.50000000000000/4.35000000000000/4,4.50000000000000/4.35000000000000/5,0.500000000000000/-0.300000000000000/1,1.50000000000000/-0.300000000000000/2,2.50000000000000/-0.300000000000000/3,3.50000000000000/-0.300000000000000/4,4.50000000000000/-0.300000000000000/5,5.30000000000000/0.500000000000000/6,5.30000000000000/1.50000000000000/7,5.30000000000000/2.50000000000000/8,5.30000000000000/3.50000000000000/9}{3/3,4/3,2/2,4/2,1/1,4/1,0/0,4/0,2/3,0/2,3/2,0/1,3/1,1/0,2/0}{0/3,1/3,1/2,2/1,3/0}{0/3,4/0,2/2}\\

$\downarrow$ & $\downarrow$\\

\deogcolored{0/0,0/1,0/2,0/3,5/0,5/1,5/2,5/3,0/0,0/1,0/2,0/3,5/0,5/1,5/2,5/3}{0/4,1/4,2/4,3/4,4/4,0/0,1/0,2/0,3/0,4/0}{-0.350000000000000/0.500000000000000/-3,-0.350000000000000/1.50000000000000/-2,-0.350000000000000/2.50000000000000/-1,-0.350000000000000/3.50000000000000/\textcolor{red}{0},0.500000000000000/4.35000000000000/1,1.50000000000000/4.35000000000000/\textcolor{red}{2},2.50000000000000/4.35000000000000/3,3.50000000000000/4.35000000000000/4,4.50000000000000/4.35000000000000/5,0.500000000000000/-0.300000000000000/1,1.50000000000000/-0.300000000000000/2,2.50000000000000/-0.300000000000000/3,3.50000000000000/-0.300000000000000/4,4.50000000000000/-0.300000000000000/5,5.30000000000000/0.500000000000000/6,5.30000000000000/1.50000000000000/7,5.30000000000000/2.50000000000000/8,5.30000000000000/3.50000000000000/9}{3/3,4/3,2/2,4/2,1/1,4/1,0/0,4/0,2/3,0/2,3/2,0/1,3/1,1/0,2/0}{0/3,1/3,1/2,2/1,3/0}{1/3,3/1}

&

\deogcolored{0/0,0/1,0/2,0/3,5/0,5/1,5/2,5/3,0/0,0/1,0/2,0/3,5/0,5/1,5/2,5/3}{0/4,1/4,2/4,3/4,4/4,0/0,1/0,2/0,3/0,4/0}{-0.350000000000000/0.500000000000000/-3,-0.350000000000000/1.50000000000000/-2,-0.350000000000000/2.50000000000000/-1,-0.350000000000000/3.50000000000000/\textcolor{red}{0},0.500000000000000/4.35000000000000/1,1.50000000000000/4.35000000000000/\textcolor{red}{2},2.50000000000000/4.35000000000000/3,3.50000000000000/4.35000000000000/4,4.50000000000000/4.35000000000000/5,0.500000000000000/-0.300000000000000/1,1.50000000000000/-0.300000000000000/2,2.50000000000000/-0.300000000000000/3,3.50000000000000/-0.300000000000000/4,4.50000000000000/-0.300000000000000/5,5.30000000000000/0.500000000000000/6,5.30000000000000/1.50000000000000/7,5.30000000000000/2.50000000000000/8,5.30000000000000/3.50000000000000/9}{3/3,4/3,4/2,1/1,4/1,0/0,2/3,0/2,3/2,0/1,3/1,1/0,2/0,0/3,2/2}{2/1,1/3,1/2,3/0,4/0}{1/2,2/1,3/0}\\

$\downarrow$ & $\downarrow$\\

\deogcolored{0/0,0/1,0/2,0/3,5/0,5/1,5/2,5/3,0/0,0/1,0/2,0/3,5/0,5/1,5/2,5/3}{0/4,1/4,2/4,3/4,4/4,0/0,1/0,2/0,3/0,4/0}{-0.350000000000000/0.500000000000000/-3,-0.350000000000000/1.50000000000000/-2,-0.350000000000000/2.50000000000000/-1,-0.350000000000000/3.50000000000000/\textcolor{red}{0},0.500000000000000/4.35000000000000/\textcolor{red}{1},1.50000000000000/4.35000000000000/2,2.50000000000000/4.35000000000000/3,3.50000000000000/4.35000000000000/4,4.50000000000000/4.35000000000000/5,0.500000000000000/-0.300000000000000/1,1.50000000000000/-0.300000000000000/2,2.50000000000000/-0.300000000000000/3,3.50000000000000/-0.300000000000000/4,4.50000000000000/-0.300000000000000/5,5.30000000000000/0.500000000000000/6,5.30000000000000/1.50000000000000/7,5.30000000000000/2.50000000000000/8,5.30000000000000/3.50000000000000/9}{3/3,4/3,2/2,4/2,1/1,4/1,0/0,4/0,2/3,0/2,3/2,0/1,1/0,2/0,1/3}{0/3,1/2,2/1,3/0,3/1}{0/3,1/2,2/1,4/0}

&

\deogcolored{0/0,0/1,0/2,0/3,5/0,5/1,5/2,5/3,0/0,0/1,0/2,0/3,5/0,5/1,5/2,5/3}{0/4,1/4,2/4,3/4,4/4,0/0,1/0,2/0,3/0,4/0}{-0.350000000000000/0.500000000000000/-3,-0.350000000000000/1.50000000000000/-2,-0.350000000000000/2.50000000000000/-1,-0.350000000000000/3.50000000000000/0,0.500000000000000/4.35000000000000/\textcolor{red}{1}, 1.50000000000000/4.35000000000000/\textcolor{red}{2},2.50000000000000/4.35000000000000/3,3.50000000000000/4.35000000000000/4,4.50000000000000/4.35000000000000/5,0.500000000000000/-0.300000000000000/1,1.50000000000000/-0.300000000000000/2,2.50000000000000/-0.300000000000000/3,3.50000000000000/-0.300000000000000/4,4.50000000000000/-0.300000000000000/5,5.30000000000000/0.500000000000000/6,5.30000000000000/1.50000000000000/7,5.30000000000000/2.50000000000000/8,5.30000000000000/3.50000000000000/9}{3/3,4/3,4/2,1/1,4/1,0/0,2/3,0/2,3/2,0/1,3/1,1/0,2/0,0/3,2/2}{2/1,1/3,1/2,3/0,4/0}{1/3,3/1}\\

$\downarrow$ & $\downarrow$\\

\deog{0/0,0/1,0/2,0/3,5/0,5/1,5/2,5/3,0/0,0/1,0/2,0/3,5/0,5/1,5/2,5/3}{0/4,1/4,2/4,3/4,4/4,0/0,1/0,2/0,3/0,4/0}{-0.350000000000000/0.500000000000000/-3,-0.350000000000000/1.50000000000000/-2,-0.350000000000000/2.50000000000000/-1,-0.350000000000000/3.50000000000000/0,0.500000000000000/4.35000000000000/1,1.50000000000000/4.35000000000000/2,2.50000000000000/4.35000000000000/3,3.50000000000000/4.35000000000000/4,4.50000000000000/4.35000000000000/5,0.500000000000000/-0.300000000000000/1,1.50000000000000/-0.300000000000000/2,2.50000000000000/-0.300000000000000/3,3.50000000000000/-0.300000000000000/4,4.50000000000000/-0.300000000000000/5,5.30000000000000/0.500000000000000/6,5.30000000000000/1.50000000000000/7,5.30000000000000/2.50000000000000/8,5.30000000000000/3.50000000000000/9}{3/3,4/3,2/2,4/2,1/1,4/1,0/0,2/3,0/2,3/2,0/1,1/0,2/0,1/3,2/1}{0/3,1/2,3/0,3/1,4/0}

&

\deog{0/0,0/1,0/2,0/3,5/0,5/1,5/2,5/3,0/0,0/1,0/2,0/3,5/0,5/1,5/2,5/3}{0/4,1/4,2/4,3/4,4/4,0/0,1/0,2/0,3/0,4/0}{-0.350000000000000/0.500000000000000/-3,-0.350000000000000/1.50000000000000/-2,-0.350000000000000/2.50000000000000/-1,-0.350000000000000/3.50000000000000/0,0.500000000000000/4.35000000000000/1,1.50000000000000/4.35000000000000/2,2.50000000000000/4.35000000000000/3,3.50000000000000/4.35000000000000/4,4.50000000000000/4.35000000000000/5,0.500000000000000/-0.300000000000000/1,1.50000000000000/-0.300000000000000/2,2.50000000000000/-0.300000000000000/3,3.50000000000000/-0.300000000000000/4,4.50000000000000/-0.300000000000000/5,5.30000000000000/0.500000000000000/6,5.30000000000000/1.50000000000000/7,5.30000000000000/2.50000000000000/8,5.30000000000000/3.50000000000000/9}{3/3,4/3,4/2,1/1,4/1,0/0,2/3,0/2,3/2,0/1,1/3,1/0,2/0,0/3,2/2}{2/1,3/1,1/2,3/0,4/0}\\

\end{tabular}

\caption{The bijections $\phi_{\mathbf{f}}$ and $\phi_{\mathbf{f}'}$ where $\mathbf{f}=s_1s_0s_1$ and $\mathbf{f}'=s_0s_1s_0$ for $f = [5,4,7,8,9,10,11,12,15]$. The highlighted boxes are where the red-indexed wires meet and thus change in the next application of Lemma \ref{lemma:gentwistedinjection}.}

\label{fig:antitodistfig}

\end{figure}

\begin{proof}[Proof of Proposition \ref{prop:antitodist}]
By Lemma \ref{lemma:antikey}, we know $\antiaffdeo{f}{P}=\affdeo{f}{P}^{(f^{-1})}$. Let $\mathbf{f} = s_{i_1}\dots s_{i_{\ell}}$. Then, $f > s_{i_1}f$ and for any $2\leq r \leq \ell$, we have $s_{i_{r-1}}\dots s_{i_1}f>s_{i_r}\dots s_{i_1}f$. So, by repeatedly applying Corollary \ref{cor:gentwistedbij}, we obtain a sequence of bijections
\[
	\affdeo{f}{P}^{(f^{-1})} =  \affdeo{f}{P}^{(s_{i_{\ell}}s_{i_{\ell-1}}\dots s_{i_1})} \to  \affdeo{f}{P}^{(s_{i_{\ell-1}}\dots s_{i_1})}\to \dots\to \affdeo{f}{P}^{(s_{i_1})}\to \affdeo{f}{P}. 
\]
\end{proof}

\begin{proposition}
If $\mathbf{f}$ and $\mathbf{f}'$ are reduced words for $f\in\Bkn$ related by commutation moves, then $\phi_{\mathbf{f}}=\phi_{\mathbf{f}'}$.
\end{proposition}
\begin{proof}
It suffices to prove that if $\mathbf{f}=s_i s_j$ and $\mathbf{f}'=s_j s_i$ for $|i-j|>1$, then $\phi_{\mathbf{f}}=\phi_{\mathbf{f}'}$. However, as $F_i$ acts only on the two strands labeled $i$ and $i+1$, we note that any pair $k,\ell$ for which $k\neq i,i+1$ and $\ell\neq i,i+1$ is left unchanged. In particular, the pair $j,j+1$ is unchanged and thus the operation of $F_j$ is independent.
\end{proof}

\begin{remark}
If $\mathbf{f}$ and $\mathbf{f}'$ are reduced words for $f\in\Bkn$ related by commutation \textit{and} braid moves, we do not necessarily have $\phi_{\mathbf{f}}=\phi_{\mathbf{f}'}$; see Figure \ref{fig:antitodistfig}.
\end{remark}

\section{Point Count Invariance for Pl\"ucker Charts}\label{sec:invariance}

\begin{definition}
For each $P\in\Pkn$, we define
\[
	g_P(q) = \sum_{f\in\Bkn}R_{f,\tau_P}(q),
\]
with $R_{f,\tau_P}(q)$ defined as in \eqref{eq:rftau}.
\end{definition}

For $I\in\binom{[n]}k$, let $C_I = \{M\in\Grkn\,|\,\Delta_I(M)\neq0\}$.

\begin{proposition}
For any $P\in\Pkn$, $g_P(q)=q^{k(n-k)}$.
\end{proposition}

\begin{proof}
We have $\#C_{P_U}(\F_q) = q^{k(n-k)}$ and
\begin{align*}
	\#C_{P_U}(\F_q)=\#(\Grkn\cap C_{P_U})(\F_q) = \#\pn{\bigsqcup_{f\in\Bkn}\Pio_f\cap C_{P_U}}(\F_q) &= \sum_{f\in\Bkn}\#\pn{\Pio_f\cap C_{P_U}}(\F_q)\\
	&= \sum_{f\in\Bkn}R_{f,\tau_P}(q) = g_P(q).
\end{align*}
\end{proof}

The rest of this section is devoted to the combinatorial proof of the following proposition.
\begin{proposition}\label{prop:invofpath}
For any $P,P'\in\Pkn$, $g_P(q)=g_{P'}(q)$.
\end{proposition}

\begin{definition}
Let $A$ be a filling of a $k\times (n-k)$ rectangle with crossings \crossing and elbows \elbow. For any $P\in \Pkn$, we let $f_P(A)\in\Bkn$ be the affine permutation determined by filling $A$ inside path $P$. We also let
\[
	g_{P,A}(q) = (q-1)^{\#\text{elbows}(A)}q^{(\#\text{crossings}(A)-\ell(f_P(A)))/2}.
\]
Finally, if $A\in \affdeop{P}$, then for any integer $i\in[n]$, we let $F_{P,i}(A)$ denote the image of $A$ under the Foata map (defined in Lemma \ref{lemma:twistedinjection}) at index $i$ with up-right path $P$.
\end{definition}

We prove Proposition \ref{prop:invofpath} by showing that if $P,s_iP\in\Pkn$ with $P_i=R$ and $P_{i+1}=U$, then $g_P(q)=g_{s_iP}(q)$. To do so, we partition $\affdeop{P}$ and $\affdeop{s_iP}$ into disjoint sets $\{S_i\}_{i\in I}$, $\{T_i\}_{i\in I}$, respectively, such that
\[
	\sum_{A\in S_i}g_{P,A}(q)=\sum_{A\in T_i} g_{s_iP,A}(q).
\]
Before we describe our partition, we make a straightforward observation.
\begin{lemma}\label{lemma:changingbox}
Let $A\in\affdeop{P}$ and let $C_i(A)\in \{E,C\}$ denote the filling of the changing box of $A$. Then,
\begin{enumerate}
\item if $C_i(A)=C$, then $F_{s_iP,i}(A)\in\affdeop{s_iP}$. Additionally,
\item if $C_i(A)=E$, then $A\in\affdeop{s_iP}$.
\end{enumerate} 
\end{lemma}
\begin{proof}
By the combinatorial description of the distinguished condition, if $A\in\affdeop{P}$ and $C_i(A)=C$, then $A\in\affdeop{P}^{(s_i)}$. By construction, $F_{s_iP,i}(A)\in\affdeop{s_iP}$. If $C_i(A)=E$, then the distinguished condition is not affected as we removed a distinguished elbow from the southeastern box and added a northwestern elbow immediately under the boundary.
\end{proof}

We describe our partition using the following sets. These sets partition fillings according to whether they survive the path swap $P\mapsto s_iP$ directly, whether they require a Foata move, and whether the associated length of the underlying bounded affine permutation changes by 0 or $\pm2$.

\begin{definition}
For any $P\in\Pkn$, define the following sets:
\begin{align*}
	A_P&=\{A\in\affdeop{s_iP}\,|\, A\not\in\affdeop{P}, A\not\in F_{s_iP,i}(\affdeop{P})\},\\
	B_P &= \{A\in\affdeop{s_iP}\,|\,A\in F_{s_iP,i}(\affdeop{P}) \text{ and } \ell(f_P(A))+2=\ell(f_{s_iP}(A))\},\\
	C_P &= \{A\in\affdeop{s_iP}\,|\,A\in F_{s_iP,i}(\affdeop{P})  \text{ and } \ell(f_P(A))=\ell(f_{s_iP}(A))+2\},\\
	D_P &= \{A\in\affdeop{s_iP}\,|\,A\in F_{s_iP,i}(\affdeop{P})  \text{ and } \ell(f_P(A))=\ell(f_{s_iP}(A))\},\\
	G_P &= \{A\in\affdeop{s_iP}\,|\, A\in \affdeop{P} \text{ and } C_i(A)=E\},\\
	V_P &=\{A\in\affdeop{P}\,|\,C_i(A)=C \text{ and } \ell(f_P(A))+2=\ell(f_{s_iP}(A))\},\\
	W_P &=\{A\in\affdeop{P}\,|\,C_i(A)=C \text{ and } \ell(f_P(A))=\ell(f_{s_iP}(A))+2\},\\
	X_P &=\{A\in\affdeop{P}\,|\,C_i(A)=C, C_i(F_{s_iP,i}(A))=E, \text{ and } \ell(f_P(A))=\ell(f_{s_iP}(A))\},\\
	Y_P &= \{A\in\affdeop{P}\,|\,C_i(A)=C, C_i(F_{s_iP,i}(A))=C, \text{ and } \ell(f_P(A))=\ell(f_{s_iP}(A))\},\\
	Z_P &= \{A\in\affdeop{P}\,|\,C_i(A)=E\}.
\end{align*}
\end{definition}

\begin{corollary}\label{cor:partition}
For $P\in\Pkn$ with $P_i = R$ and $P_{i+1}=U$, we have the partitions
\begin{align*}
	\affdeop{P} &= V_P\sqcup W_P\sqcup X_P\sqcup Y_P\sqcup Z_P,\\
	\affdeop{s_iP} &=A_P\sqcup B_P\sqcup C_P\sqcup D_P\sqcup G_P.
\end{align*}
\end{corollary}
\begin{proof}
This is clear from Lemma \ref{lemma:changingbox} and noting that either $\ell(f)=\ell(s_ifs_i)$ or $\ell(f)=\ell(s_ifs_i)\pm2$ for any $f\in\Bkn$.
\end{proof}

We now demonstrate the relationships between these sets, beginning with the most straightforward. See Figure \ref{fig:invofpathfig} for a visualization.

\begin{figure}
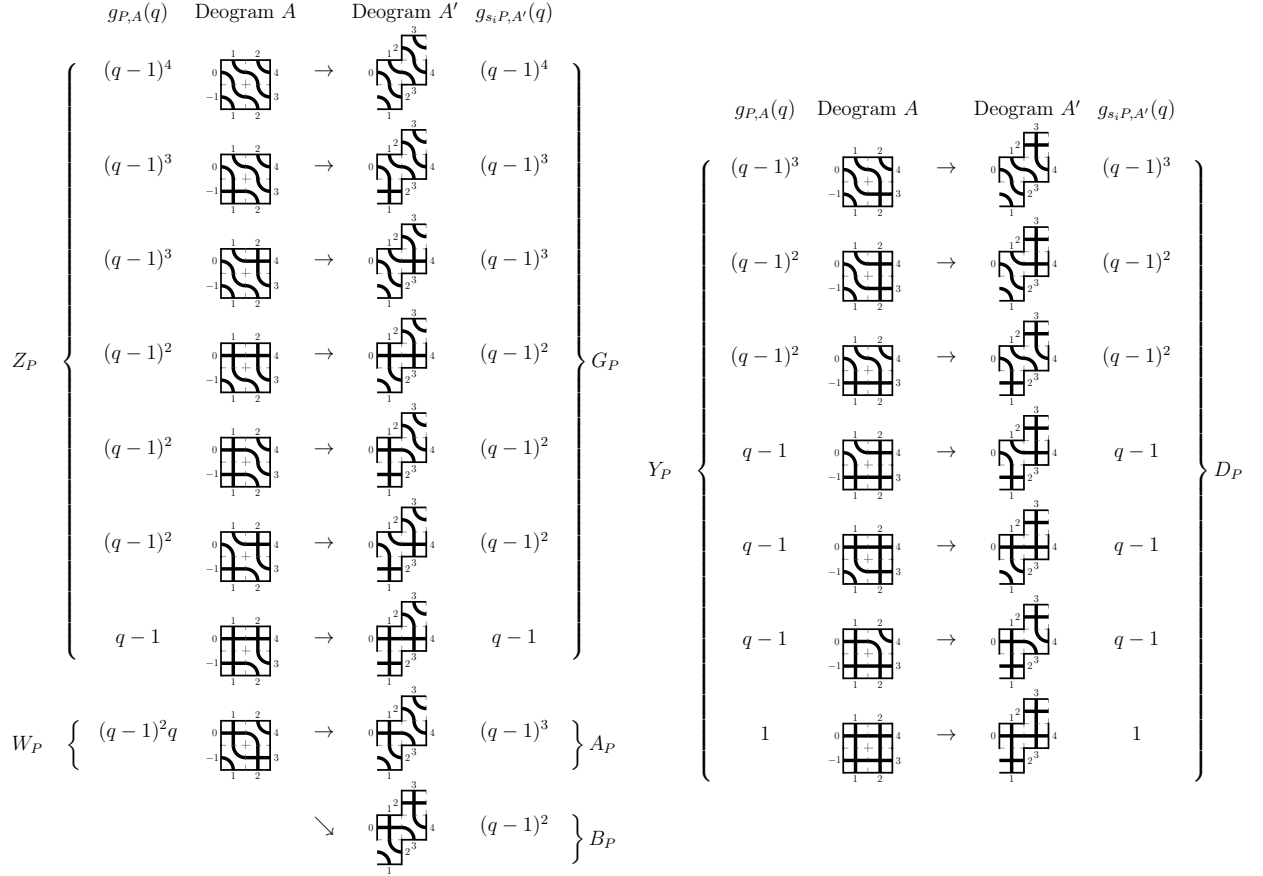

\centering

\begin{minipage}[h]{0.49\textwidth}
\centering
\scalebox{0.65}
{

\begin{tabular}{lcccccl}

& $g_{P,A}(q)$ & Deogram $A$ & & Deogram $A'$ & $g_{s_iP,A'}(q)$\\

\multirow{1}{*}{\hspace{1em}$Z_P\left.\begin{array}{l}
                \\\\\\\\\\\\\\\\\\\\\\\\\\\\\\\\\\\\\\\\\\\\\\\\\\
                \end{array}\right\lbrace$} 
                &

$(q-1)^4$
&

\deog{2/0,0/0,2/1,0/1}{0/0,0/2,1/0,1/2}{0.500000000000000/-0.250000000000000/1,0.500000000000000/2.25000000000000/1,1.50000000000000/-0.250000000000000/2,1.50000000000000/2.25000000000000/2,2.25000000000000/0.500000000000000/3,-0.35000000000000/0.50000000000000/-1,2.25000000000000/1.50000000000000/4,-0.25000000000000/1.50000000000000/0}{0/1,1/1,0/0,1/0}{}

&
$\rightarrow$

&

\deog{1/0,1/2,2/1,0/1}{0/0,0/2,1/1,1/3}{0.500000000000000/-0.250000000000000/1,0.500000000000000/2.25000000000000/1,1.25000000000000/0.500000000000000/2,0.750000000000000/2.50000000000000/2,1.50000000000000/0.750000000000000/3,1.50000000000000/3.25000000000000/3,2.25000000000000/1.50000000000000/4,-0.25000000000000/1.50000000000000/0}{1/2,0/1,1/1,0/0}{}

&

$(q-1)^4$
 & \multirow{1}{*}{\hspace{-1em}$\left.\begin{array}{l}
                \\\\\\\\\\\\\\\\\\\\\\\\\\\\\\\\\\\\\\\\\\\\\\\\\\
                \end{array}\right\rbrace G_P$} 

\\

&

$(q-1)^3$
&

\deog{2/0,0/0,2/1,0/1}{0/0,0/2,1/0,1/2}{0.500000000000000/-0.250000000000000/1,0.500000000000000/2.25000000000000/1,1.50000000000000/-0.250000000000000/2,1.50000000000000/2.25000000000000/2,2.25000000000000/0.500000000000000/3,-0.35000000000000/0.50000000000000/-1,2.25000000000000/1.50000000000000/4,-0.25000000000000/1.50000000000000/0}{0/1,1/1,1/0}{0/0}

&
$\rightarrow$
&

\deog{1/0,1/2,2/1,0/1}{0/0,0/2,1/1,1/3}{0.500000000000000/-0.250000000000000/1,0.500000000000000/2.25000000000000/1,1.25000000000000/0.500000000000000/2,0.750000000000000/2.50000000000000/2,1.50000000000000/0.750000000000000/3,1.50000000000000/3.25000000000000/3,2.25000000000000/1.50000000000000/4,-0.25000000000000/1.50000000000000/0}{1/2,0/1,1/1}{0/0}

&

$(q-1)^3$
\\
&
$(q-1)^3$
&

\deog{2/0,0/0,2/1,0/1}{0/0,0/2,1/0,1/2}{0.500000000000000/-0.250000000000000/1,0.500000000000000/2.25000000000000/1,1.50000000000000/-0.250000000000000/2,1.50000000000000/2.25000000000000/2,2.25000000000000/0.500000000000000/3,-0.35000000000000/0.50000000000000/-1,2.25000000000000/1.50000000000000/4,-0.25000000000000/1.50000000000000/0}{0/1,0/0,1/0}{1/1}
&
$\rightarrow$
&
\deog{1/0,1/2,2/1,0/1}{0/0,0/2,1/1,1/3}{0.500000000000000/-0.250000000000000/1,0.500000000000000/2.25000000000000/1,1.25000000000000/0.500000000000000/2,0.750000000000000/2.50000000000000/2,1.50000000000000/0.750000000000000/3,1.50000000000000/3.25000000000000/3,2.25000000000000/1.50000000000000/4,-0.25000000000000/1.50000000000000/0}{1/2,0/1,0/0}{1/1}

&

$(q-1)^3$

\\
&

$(q-1)^2$
&
\deog{2/0,0/0,2/1,0/1}{0/0,0/2,1/0,1/2}{0.500000000000000/-0.250000000000000/1,0.500000000000000/2.25000000000000/1,1.50000000000000/-0.250000000000000/2,1.50000000000000/2.25000000000000/2,2.25000000000000/0.500000000000000/3,-0.35000000000000/0.50000000000000/-1,2.25000000000000/1.50000000000000/4,-0.25000000000000/1.50000000000000/0}{0/0,1/0}{0/1,1/1}
&
$\rightarrow$
&

\deog{1/0,1/2,2/1,0/1}{0/0,0/2,1/1,1/3}{0.500000000000000/-0.250000000000000/1,0.500000000000000/2.25000000000000/1,1.25000000000000/0.500000000000000/2,0.750000000000000/2.50000000000000/2,1.50000000000000/0.750000000000000/3,1.50000000000000/3.25000000000000/3,2.25000000000000/1.50000000000000/4,-0.25000000000000/1.50000000000000/0}{1/2,0/0}{1/1,0/1}

&

$(q-1)^2$
\\
&
$(q-1)^2$
&
\deog{2/0,0/0,2/1,0/1}{0/0,0/2,1/0,1/2}{0.500000000000000/-0.250000000000000/1,0.500000000000000/2.25000000000000/1,1.50000000000000/-0.250000000000000/2,1.50000000000000/2.25000000000000/2,2.25000000000000/0.500000000000000/3,-0.35000000000000/0.50000000000000/-1,2.25000000000000/1.50000000000000/4,-0.25000000000000/1.50000000000000/0}{1/1,1/0}{0/1,0/0}
&
$\rightarrow$
&
\deog{1/0,1/2,2/1,0/1}{0/0,0/2,1/1,1/3}{0.500000000000000/-0.250000000000000/1,0.500000000000000/2.25000000000000/1,1.25000000000000/0.500000000000000/2,0.750000000000000/2.50000000000000/2,1.50000000000000/0.750000000000000/3,1.50000000000000/3.25000000000000/3,2.25000000000000/1.50000000000000/4,-0.25000000000000/1.50000000000000/0}{1/2,1/1}{0/1,0/0}

&

$(q-1)^2$

\\
&
$(q-1)^2$
&
\deog{2/0,0/0,2/1,0/1}{0/0,0/2,1/0,1/2}{0.500000000000000/-0.250000000000000/1,0.500000000000000/2.25000000000000/1,1.50000000000000/-0.250000000000000/2,1.50000000000000/2.25000000000000/2,2.25000000000000/0.500000000000000/3,-0.35000000000000/0.50000000000000/-1,2.25000000000000/1.50000000000000/4,-0.25000000000000/1.50000000000000/0}{0/1,1/0}{0/0,1/1}
&
$\rightarrow$
&

\deog{1/0,1/2,2/1,0/1}{0/0,0/2,1/1,1/3}{0.500000000000000/-0.250000000000000/1,0.500000000000000/2.25000000000000/1,1.25000000000000/0.500000000000000/2,0.750000000000000/2.50000000000000/2,1.50000000000000/0.750000000000000/3,1.50000000000000/3.25000000000000/3,2.25000000000000/1.50000000000000/4,-0.25000000000000/1.50000000000000/0}{1/2,0/1}{1/1,0/0}

&

$(q-1)^2$
\\
&
$q-1$
&
\deog{2/0,0/0,2/1,0/1}{0/0,0/2,1/0,1/2}{0.500000000000000/-0.250000000000000/1,0.500000000000000/2.25000000000000/1,1.50000000000000/-0.250000000000000/2,1.50000000000000/2.25000000000000/2,2.25000000000000/0.500000000000000/3,-0.35000000000000/0.50000000000000/-1,2.25000000000000/1.50000000000000/4,-0.25000000000000/1.50000000000000/0}{1/0}{0/0,1/1,0/1}
&
$\rightarrow$
&

\deog{1/0,1/2,2/1,0/1}{0/0,0/2,1/1,1/3}{0.500000000000000/-0.250000000000000/1,0.500000000000000/2.25000000000000/1,1.25000000000000/0.500000000000000/2,0.750000000000000/2.50000000000000/2,1.50000000000000/0.750000000000000/3,1.50000000000000/3.25000000000000/3,2.25000000000000/1.50000000000000/4,-0.25000000000000/1.50000000000000/0}{1/2}{0/0,0/1,1/1}

&

$q-1$

\\
\multirow{1}{*}{\hspace{1em}$W_P\left.\begin{array}{l}
                \\\\
                \end{array}\right\lbrace$} 
&
$(q-1)^2q$
&
\deog{2/0,0/0,2/1,0/1}{0/0,0/2,1/0,1/2}{0.500000000000000/-0.250000000000000/1,0.500000000000000/2.25000000000000/1,1.50000000000000/-0.250000000000000/2,1.50000000000000/2.25000000000000/2,2.25000000000000/0.500000000000000/3,-0.35000000000000/0.50000000000000/-1,2.25000000000000/1.50000000000000/4,-0.25000000000000/1.50000000000000/0}{0/0,1/1}{0/1,1/0}
&
$\rightarrow$
&
\deog{1/0,1/2,2/1,0/1}{0/0,0/2,1/1,1/3}{0.500000000000000/-0.250000000000000/1,0.500000000000000/2.25000000000000/1,1.25000000000000/0.500000000000000/2,0.750000000000000/2.50000000000000/2,1.50000000000000/0.750000000000000/3,1.50000000000000/3.25000000000000/3,2.25000000000000/1.50000000000000/4,-0.25000000000000/1.50000000000000/0}{1/2,1/1,0/0}{0/1}

&

$(q-1)^3$
 & \multirow{1}{*}{\hspace{-1em}$\left.\begin{array}{l}
                \\\\
                \end{array}\right\rbrace A_P$} 
\\

&

&

&
$\searrow$
&

\deog{1/0,1/2,2/1,0/1}{0/0,0/2,1/1,1/3}{0.500000000000000/-0.250000000000000/1,0.500000000000000/2.25000000000000/1,1.25000000000000/0.500000000000000/2,0.750000000000000/2.50000000000000/2,1.50000000000000/0.750000000000000/3,1.50000000000000/3.25000000000000/3,2.25000000000000/1.50000000000000/4,-0.25000000000000/1.50000000000000/0}{0/0,1/1}{0/1,1/2}

&

$(q-1)^2$
 & \multirow{1}{*}{\hspace{-1em}$\left.\begin{array}{l}
                \\\\
                \end{array}\right\rbrace B_P$} 
\end{tabular}
}
\end{minipage}
\hfill
\begin{minipage}[t]{0.49\textwidth}
\centering
\scalebox{0.65}{

\begin{tabular}{lcccccl}
& $g_{P,A}(q)$ & Deogram $A$ & & Deogram $A'$ & $g_{s_iP,A'}(q)$\\

\multirow{1}{*}{\hspace{1em}$Y_P\left.\begin{array}{l}
                \\\\\\\\\\\\\\\\\\\\\\\\\\\\\\\\\\\\\\\\\\\\\\\\\\\\
                \end{array}\right\lbrace$} 
     
&
$(q-1)^3$
&

\deog{2/0,0/0,2/1,0/1}{0/0,0/2,1/0,1/2}{0.500000000000000/-0.250000000000000/1,0.500000000000000/2.25000000000000/1,1.50000000000000/-0.250000000000000/2,1.50000000000000/2.25000000000000/2,2.25000000000000/0.500000000000000/3,-0.35000000000000/0.50000000000000/-1,2.25000000000000/1.50000000000000/4,-0.25000000000000/1.50000000000000/0}{0/1,1/1,0/0}{1/0}

&
$\rightarrow$
&

\deog{1/0,1/2,2/1,0/1}{0/0,0/2,1/1,1/3}{0.500000000000000/-0.250000000000000/1,0.500000000000000/2.25000000000000/1,1.25000000000000/0.500000000000000/2,0.750000000000000/2.50000000000000/2,1.50000000000000/0.750000000000000/3,1.50000000000000/3.25000000000000/3,2.25000000000000/1.50000000000000/4,-0.25000000000000/1.50000000000000/0}{0/0,1/1,0/1}{1/2}

&

$(q-1)^3$
&\multirow{1}{*}{\hspace{-1em}$\left.\begin{array}{l}
                \\\\\\\\\\\\\\\\\\\\\\\\\\\\\\\\\\\\\\\\\\\\\\\\\\\\
                \end{array}\right\rbrace D_P$} 
\\
&
$(q-1)^2$
&

\deog{2/0,0/0,2/1,0/1}{0/0,0/2,1/0,1/2}{0.500000000000000/-0.250000000000000/1,0.500000000000000/2.25000000000000/1,1.50000000000000/-0.250000000000000/2,1.50000000000000/2.25000000000000/2,2.25000000000000/0.500000000000000/3,-0.35000000000000/0.50000000000000/-1,2.25000000000000/1.50000000000000/4,-0.25000000000000/1.50000000000000/0}{0/1,0/0}{1/1,1/0}
&
$\rightarrow$
&
\deog{1/0,1/2,2/1,0/1}{0/0,0/2,1/1,1/3}{0.500000000000000/-0.250000000000000/1,0.500000000000000/2.25000000000000/1,1.25000000000000/0.500000000000000/2,0.750000000000000/2.50000000000000/2,1.50000000000000/0.750000000000000/3,1.50000000000000/3.25000000000000/3,2.25000000000000/1.50000000000000/4,-0.25000000000000/1.50000000000000/0}{0/0,0/1}{1/1,1/2}

&

$(q-1)^2$

\\
&
$(q-1)^2$
&
\deog{2/0,0/0,2/1,0/1}{0/0,0/2,1/0,1/2}{0.500000000000000/-0.250000000000000/1,0.500000000000000/2.25000000000000/1,1.50000000000000/-0.250000000000000/2,1.50000000000000/2.25000000000000/2,2.25000000000000/0.500000000000000/3,-0.35000000000000/0.50000000000000/-1,2.25000000000000/1.50000000000000/4,-0.25000000000000/1.50000000000000/0}{0/1,1/1}{0/0,1/0}
&
$\rightarrow$
&

\deog{1/0,1/2,2/1,0/1}{0/0,0/2,1/1,1/3}{0.500000000000000/-0.250000000000000/1,0.500000000000000/2.25000000000000/1,1.25000000000000/0.500000000000000/2,0.750000000000000/2.50000000000000/2,1.50000000000000/0.750000000000000/3,1.50000000000000/3.25000000000000/3,2.25000000000000/1.50000000000000/4,-0.25000000000000/1.50000000000000/0}{1/1,0/1}{1/2,0/0}

&

$(q-1)^2$

\\
&
$q-1$
&
\deog{2/0,0/0,2/1,0/1}{0/0,0/2,1/0,1/2}{0.500000000000000/-0.250000000000000/1,0.500000000000000/2.25000000000000/1,1.50000000000000/-0.250000000000000/2,1.50000000000000/2.25000000000000/2,2.25000000000000/0.500000000000000/3,-0.35000000000000/0.50000000000000/-1,2.25000000000000/1.50000000000000/4,-0.25000000000000/1.50000000000000/0}{0/1}{0/0,1/1,1/0}
&
$\rightarrow$
&

\deog{1/0,1/2,2/1,0/1}{0/0,0/2,1/1,1/3}{0.500000000000000/-0.250000000000000/1,0.500000000000000/2.25000000000000/1,1.25000000000000/0.500000000000000/2,0.750000000000000/2.50000000000000/2,1.50000000000000/0.750000000000000/3,1.50000000000000/3.25000000000000/3,2.25000000000000/1.50000000000000/4,-0.25000000000000/1.50000000000000/0}{0/1}{0/0,1/1,1/2}

&

$q-1$ 

\\
&
$q-1$
&
\deog{2/0,0/0,2/1,0/1}{0/0,0/2,1/0,1/2}{0.500000000000000/-0.250000000000000/1,0.500000000000000/2.25000000000000/1,1.50000000000000/-0.250000000000000/2,1.50000000000000/2.25000000000000/2,2.25000000000000/0.500000000000000/3,-0.35000000000000/0.50000000000000/-1,2.25000000000000/1.50000000000000/4,-0.25000000000000/1.50000000000000/0}{0/0}{0/1,1/1,1/0}
&
$\rightarrow$
&

\deog{1/0,1/2,2/1,0/1}{0/0,0/2,1/1,1/3}{0.500000000000000/-0.250000000000000/1,0.500000000000000/2.25000000000000/1,1.25000000000000/0.500000000000000/2,0.750000000000000/2.50000000000000/2,1.50000000000000/0.750000000000000/3,1.50000000000000/3.25000000000000/3,2.25000000000000/1.50000000000000/4,-0.25000000000000/1.50000000000000/0}{0/0}{1/1,0/1,1/2}

&

$q-1$
\\
&

$q-1$
&
\deog{2/0,0/0,2/1,0/1}{0/0,0/2,1/0,1/2}{0.500000000000000/-0.250000000000000/1,0.500000000000000/2.25000000000000/1,1.50000000000000/-0.250000000000000/2,1.50000000000000/2.25000000000000/2,2.25000000000000/0.500000000000000/3,-0.35000000000000/0.50000000000000/-1,2.25000000000000/1.50000000000000/4,-0.25000000000000/1.50000000000000/0}{1/1}{0/0,0/1,1/0}
&
$\rightarrow$
&

\deog{1/0,1/2,2/1,0/1}{0/0,0/2,1/1,1/3}{0.500000000000000/-0.250000000000000/1,0.500000000000000/2.25000000000000/1,1.25000000000000/0.500000000000000/2,0.750000000000000/2.50000000000000/2,1.50000000000000/0.750000000000000/3,1.50000000000000/3.25000000000000/3,2.25000000000000/1.50000000000000/4,-0.25000000000000/1.50000000000000/0}{1/1}{0/0,0/1,1/2}

&

$q-1$
\\
&
$1$
&
\deog{2/0,0/0,2/1,0/1}{0/0,0/2,1/0,1/2}{0.500000000000000/-0.250000000000000/1,0.500000000000000/2.25000000000000/1,1.50000000000000/-0.250000000000000/2,1.50000000000000/2.25000000000000/2,2.25000000000000/0.500000000000000/3,-0.35000000000000/0.50000000000000/-1,2.25000000000000/1.50000000000000/4,-0.25000000000000/1.50000000000000/0}{}{1/0,0/0,1/1,0/1}
&
$\rightarrow$
&

\deog{1/0,1/2,2/1,0/1}{0/0,0/2,1/1,1/3}{0.500000000000000/-0.250000000000000/1,0.500000000000000/2.25000000000000/1,1.25000000000000/0.500000000000000/2,0.750000000000000/2.50000000000000/2,1.50000000000000/0.750000000000000/3,1.50000000000000/3.25000000000000/3,2.25000000000000/1.50000000000000/4,-0.25000000000000/1.50000000000000/0}{}{0/1,0/0,1/1,1/2}

&

$1$

\end{tabular}
}

\end{minipage}
\caption{Example of Proposition \ref{prop:invofpath} with $P = \text{RRUU}$. In this case, $C_P=W_P=X_P=\emptyset$.}
\label{fig:invofpathfig}
\end{figure}

\begin{lemma}\label{lemma:easycollection}
For $P\in\Pkn$ with $P_i=R$ and $P_{i+1}=U$, we have a bijection $\zeta_P:Z_P\to G_P$ where for all $Z\in Z_P$,
\[
	g_{P,Z}(q)=g_{s_iP,\zeta_P(Z)}(q).
\]
\end{lemma}
\begin{proof}
By Lemma \ref{lemma:changingbox}, we see that, as a set of fillings, $Z_P=G_P$, so take $\zeta_P=\text{Id}_{Z_P}$.
\end{proof}

\begin{lemma}\label{lemma:secondeasycollection}
For $P\in\Pkn$ with $P_i=R$ and $P_{i+1}=U$, we have a bijection $\beta_P:Y_P\to D_P$ where for all $Y\in Y_P$,
\[
g_{P,Y}(q) = g_{s_iP,\beta_P(Y)}(q).
\]
\end{lemma}
\begin{proof}
Let $\beta_P = F_{s_iP,i}$. By definition of $Y_P$ and $D_P$, this is a bijection. The resulting equality of $g$ polynomials follows from the definition of the sets and the fact that $F_{s_iP,i}$ does not change the number of elbows, crossings, or the underlying affine permutation of the filling.
\end{proof}

\begin{lemma}\label{lemma:twoinone}
For any $P,s_iP\in\Pkn$ with $P_i=R$ and $P_{i+1}=U$, we have a bijection $\phi_P:B_P\to A_P$ such that, for any $B\in B_P$,
\begin{enumerate}
\item $\ell(f_{s_iP}(B))= \ell(f_{s_iP}(\phi_P(B)))+1$ and
\item $\#\text{elbows}(B)+1  = \#\text{elbows}(\phi_P(B))$.
\end{enumerate}
Furthermore, the map $F_{s_iP,i}:V_P\to B_P$ is a bijection.
\end{lemma}

\begin{proof}
The last claim follows by the definition of $V_P$ and $B_P$. Let $B\in B_P$ be given and let $V\in V_P$ be such that $B=F_{s_iP,i}(V)$. We create $\phi_P$ by constructing a bijection $\alpha_P:V_P\to A_P$.

Define $A$ to be the filling identical to $V$, with the exception that $C_i(A)=E$. We claim that $A\in A_P$, and once we show this we have our desired bijection.

Since $V\in V_P$, we have $\ell(f_P(V))+2=\ell(f_{s_iP}(V))$. Let $f = f_P(V)$. Then, $s_ifs_i=f_{s_iP}(V)$, and as $\ell(f)+2=\ell(s_ifs_i)$, we have $s_if,fs_i>f$. By construction of $A$, we have $f_P(A)=f_{s_iP}(A) = s_if$. As $s_if>f$, the wires labeled by $s_if^{-1}(i)$ and $s_if^{-1}(i+1)$ cross an odd number of times. Since $C_i(A)=E$, this implies $A\not\in\affdeop{P}$. Additionally, $A\not\in F_{s_iP,i}(\affdeop{P})$: the affected pair has the wrong initial parity for the inverse Foata construction because the changing box is an elbow rather than a crossing.

So, $A\in A_P$. We define $\alpha_P(V)=A$ and let $\phi_P=\alpha_P\circ(F_{s_iP,i})^{-1}$. The resulting properties of $\phi_P$ follow from the analysis above and recalling that $F_{s_iP,i}$ does not change the number of elbows, crossings, or the underlying affine permutation of the filling.
\end{proof}

\begin{corollary}\label{cor:ABVcor}
For any $P,s_iP\in\Pkn$ with $P_i=R$ and $P_{i+1}=U$, let $B=F_{s_iP,i}(V)\in B_P$ for some $V\in V_P$ and $A = \phi_P(B)\in A_P$. Then, we have
\[
	g_{s_iP,A}(q)+g_{s_iP,B}(q)=g_{P,V}(q).
\]
\end{corollary}

\begin{proof}
Let $n_e = \#\text{elbows}(A)$ and $n_c=\#\text{crossings}(A)$. Then by Lemma \ref{lemma:twoinone}, $\#\text{elbows}(B)=n_e-1$ and thus $\#\text{crossings}(B)=n_c+1$. Also as $\ell(f_{s_iP}(A))+1 = \ell(f_{s_iP}(B))$, letting $\ell :=  \ell(f_{s_iP}(A))$, we may write
\[
	g_{s_iP,A}(q)+g_{s_iP,B}(q) = (q-1)^{n_e}q^{(n_c-\ell)/2} + (q-1)^{n_e-1}q^{(n_c+1-(\ell+1))/2} = q\cdot q^{(n_c-\ell)/2}(q-1)^{n_e-1}.
\]
As $B =  F_{s_iP,i}(V)$ and $\ell(f_{s_iP}(A))+1=\ell(f_{s_iP}(B)) = \ell(f_{P}(V))+2$, we have
\[
	g_{P,V}(q) = (q-1)^{n_e-1}q^{(n_c+1-(\ell -1))/2} =  q\cdot q^{(n_c-\ell)/2}(q-1)^{n_e-1}.
\]
\end{proof}

\begin{lemma}\label{lemma:CWX}
For any $P,s_iP\in\Pkn$ with $P_i=R$ and $P_{i+1}=U$, we have a bijection $\psi_P:W_P\to X_P$ such that, for any $W\in W_P$,
\begin{enumerate}
\item $\ell(f_{P}(W))= \ell(f_{P}(\psi_P(W)))+1$ and
\item $\#\text{elbows}(W)+1  = \#\text{elbows}(\psi_P(W))$.
\end{enumerate}
Furthermore, the map $F_{s_iP,i}:W_P\to C_P$ is a bijection.
\end{lemma}

\begin{proof}
The last claim follows by the definition of $W_P$ and $C_P$. We construct the bijection $\psi_P:W_P\to X_P$ as follows:
\begin{enumerate}
\item Fix some $W\in W_P$. Consider the sequence $\mathbf{S}=\{S_1,\dots,S_r\}$ of elbows and crossings of the wires labeled $i$ and $i+1$ in $W$. Let $j$ be the smallest index such that $S_j=C$, and let $\mathcal{T}=\{S_{j+1},\dots,S_r\}$.
\item Apply the Foata map $F_i$ to $\mathcal{T}$.
\item Prepend the string $\{S_1,\dots,S_{j-1},E\}$ to $\mathcal{T}$. Call this $\widetilde{\mathbf{S}}$. 
\item Create an affine Deogram $X$ identical to $W$ with the sequences of elbows and crossings $\widetilde{\mathbf{S}}$ for the wires $i$ and $i+1$.
\end{enumerate}
Define $\psi_P(W)=X$. We show $X\in X_P$.

Firstly, it is clear that $X\in\affdeop{P}$ as, by construction, the wires $i$ and $i+1$ satisfy the distinguished condition, and as $W\in \affdeop{P}$, so do the other pairs of wires in $X$.

Letting $f=f_P(W)$, we see $s_ifs_i=f_{s_iP}(W)$, and as $W\in W_P$, we have $s_ifs_i<f$, implying $s_if,fs_i<f$. Specifically, $\ell(s_if)+1=\ell(fs_i)+1=\ell(f)$. We see that $f_P(X)=s_if$ and $f_{s_iP}(X)=fs_i$.

Finally, this operation is bijective as every $W\in W_P$ has the wires $i$ and $i+1$ crossing an odd number of times, and the string $\mathcal{T}$ will either be empty or begin with a crossing $C$. Thus, the inverse operation is simply described:
\begin{enumerate}
\item Fix some $X\in X_P$. Consider the sequence $\mathbf{S}=(S_1,\dots,S_r)$ of elbows and crossings of the wires labeled $i$ and $i+1$ in $X$. Let $j$ be the largest index such that $S_\ell=E$ for all $1 \leq \ell \leq j$, and let $\mathcal{T}=(S_{j+1},\dots,S_r)$.
\item Apply the inverse Foata map $F_i$ to $\mathcal{T}$.
\item Prepend the string $(S_1,\dots,S_{j-1},C)$ to $\mathcal{T}$. Call this $\widetilde{\mathbf{S}}$. 
\item Create an affine Deogram $W$ identical to $X$ with the sequences of elbows and crossings $\widetilde{\mathbf{S}}$ for the wires $i$ and $i+1$.
\end{enumerate}
Checking $W\in W_P$ is similar to the argument above.
\end{proof}

\begin{corollary}\label{cor:CWXcor}
For any $P,s_iP\in\Pkn$ with $P_i=R$ and $P_{i+1}=U$, let $W\in W_P$ and $X = \psi_P(W)\in X_P$. Then, letting $F_{s_iP,i}(W)=C\in C_P$, we have
\[
	g_{P,W}(q)+g_{P,X}(q)=g_{s_iP,C}(q).
\]
\end{corollary}

\begin{proof}
Let $n_e = \#\text{elbows}(W)$ and $n_c=\#\text{crossings}(W)$. Then by Lemma \ref{lemma:CWX}, $\#\text{elbows}(X)=n_e+1$ and thus $\#\text{crossings}(X)=n_c-1$. Also as $\ell(f_{P}(W)) = \ell(f_{P}(X))+1$, letting $\ell :=  \ell(f_{P}(W))$, we may write
\[
	g_{P,W}(q)+g_{P,X}(q) = (q-1)^{n_e}q^{(n_c-\ell)/2} + (q-1)^{n_e+1}q^{(n_c-1-(\ell-1))/2} = q\cdot q^{(n_c-\ell)/2}(q-1)^{n_e}.
\]
As $C =  F_{s_iP,i}(W)$ and $\ell(f_{s_iP}(C)) = \ell(f_{s_iP}(W))=\ell(f_P(W))-2$, we have
\[
	g_{s_iP,C}(q) = (q-1)^{n_e} q^{(n_c-(\ell-2))/2} = q\cdot q^{(n_c-\ell)/2}(q-1)^{n_e}.
\]
\end{proof}

\begin{proof}[Proof of Proposition \ref{prop:invofpath}.]
We partition $\affdeop{P}$ and $\affdeop{s_iP}$ as in Corollary \ref{cor:partition}. The combination of Lemmas \ref{lemma:easycollection} and \ref{lemma:secondeasycollection} and Corollaries \ref{cor:ABVcor} and \ref{cor:CWXcor} shows
\[
	g_P(q) = \sum_{A\in \affdeop{P}}g_{P,A}(q) = \sum_{A\in\affdeop{s_iP}}g_{s_iP,A}(q) = g_{s_iP}(q).
\]
\end{proof}

\section{Geometric Isomorphisms}\label{sec:geometry}

\subsection{Distinguished vs Anti-Distinguished, Geometrically}\label{subsec:antitodistgeo}

The goal of this subsection is to give a geometric proof of \Cref{thm:geometric}~\labelcref{part3}. The first observation is that anti-distinguished $v$-subwords in $\mathbf{w}$ are in bijection with distinguished $v^{-1}$-subwords of $\mathbf{w}^{\text{rev}}$.

\begin{proposition}\label{prop:antirev}
Let $v,w\in S_n$ and $\mathbf{w}$ be a reduced word for $w$. Then a $v$-subword $\mathbf{v}$ of $\mathbf{w}$ is anti-distinguished if and only if $\mathbf{v}^{\text{rev}}$ is a distinguished $v^{-1}$-subword of $\mathbf{w}^{\text{rev}}$.
\end{proposition}

\begin{proof}
Writing $\mathbf{w} = (s_{i_1},\dots,s_{i_m})$ and $\mathbf{w}^{\text{rev}}=(s_{i_m},\dots,s_{i_1})$, we have $\bv_{\lra{j}} = (\bv^{\text{rev}}_{(m-j+1)})^{-1}$ and $s_{i_j}\bv_{\lra{j+1}} = (\bv^{\text{rev}}_{(m-j)} s_{i_j})^{-1}$. Then, $\bv_{\lra{j}}  \leq s_{i_j}\bv_{\lra{j+1}}$ is equivalent to $ (\bv^{\text{rev}}_{(m-j+1)})^{-1} \leq  (\bv^{\text{rev}}_{(m-j)} s_{i_j})^{-1}$ which is also equivalent to $\bv^{\text{rev}}_{(m-j+1)}\leq  \bv^{\text{rev}}_{(m-j)} s_{i_j}$.
\end{proof}

Recalling our definition of Richardson varieties in \eqref{eq:richardson}, we also consider the left-sided quotient $B_-\backslash G$. 
\[
	\backvec{R}^{\circ}_{v,w} = B_- \backslash (B_- w B_- \cap B_- v B) \subset B_- \backslash G.
\]
The map $g\mapsto g^{-1}$ induces an isomorphism
\begin{equation}\label{eq:iota}
	\iota:\mathring{R}_{v,w} \xrightarrow{\sim} \backvec{R}^{\circ}_{v^{-1},w^{-1}},
\end{equation}
where $gB_-\mapsto B_-g^{-1}$.

\begin{proposition}[\cite{brown2006poisson,leclerc2016cluster}]\label{prop:chiral}
The map $\chi_v: \backvec{R}^{\circ}_{v,w}\to \mathring{R}_{v,w}$ obtained by identifying both sides with $N\dot{v}\cap \dot{v}N\cap B_-wB_-$ is an isomorphism.
\end{proposition}

\begin{proposition}
We have an isomorphism $\mathring{R}_{v,w}\xrightarrow{\sim} \mathring{R}_{v^{-1},w^{-1}}$. For $v\leq w$ with $w$ $k$-Grassmannian, this induces the equality $\#\deo{f}=\#\antideo{f}$.
\end{proposition}
\begin{proof}
The map $\chi_{v^{-1}}\circ \iota$ is an isomorphism via \eqref{eq:iota} and Proposition \ref{prop:chiral}. Then Theorem \ref{thm:deodharorig} coupled with $\ell(w)=\ell(w^{-1})$ and $\ell(v)=\ell(v^{-1})$ and the linear independence of the polynomials $(q-1)^{\ell(w)-\ell(v)-2a} q^{a}$ for any $0\leq a\leq \floor{(\ell(w)-\ell(v))/2}$ (as they have distinct degrees), along with Proposition \ref{prop:antirev}, implies the result.
\end{proof}

\subsection{Affine Richardson Proofs of the Chart Isomorphisms}\label{subsec:affine-richardson-charts}

The purpose of this subsection is to prove the chartwise statements in Corollaries~\ref{cor:ceq_geo}, \ref{cor:double1_geo}, and \ref{cor:double2_geo}. Unlike the preceding subsection, we do not need to track individual Pl\"ucker coordinates. Instead, we use Snider's identification
\[
    \Pio_{f,I}\cong \mathring{\Rcal}_f^{t_I}
\]
from Theorem~\ref{thm:snideriso}, together with the standard rank-one recursion for open Richardson varieties. We use the convention that both sides are empty when $f\not\leq t_I$.

We first recall the rank-one result in the form needed below. This is the one-letter case of Deodhar's decomposition; the same statement for Kac--Moody flag varieties follows from the Billig--Dyer decomposition \cite{deodhar,billig}.

\begin{proposition}\label{prop:rank-one-richardson} Let $u\leq w$ lie in the same component of an extended affine Weyl group, and let $s$ be a simple reflection. Write $\mathring{\Rcal}_u^w=\varnothing$ when $u\not\leq w$.
\begin{enumerate}
\item If $u<us$ and $w<ws$, then
\[
    \mathring{\Rcal}_u^w\cong \mathring{\Rcal}_{us}^{ws}.
\]
The same conclusion holds if $us<u$ and $ws<w$.
\item If $u<us$ and $ws<w$, then there is a closed subvariety $Z\subseteq\mathring{\Rcal}_u^w$ such that
\[
    Z\cong \mathring{\Rcal}_{us}^{ws}\times\C,
    \qquad
    \mathring{\Rcal}_u^w\setminus Z
       \cong \mathring{\Rcal}_u^{ws}\times\C^*.
\]
\end{enumerate}
The analogous statements hold for left multiplication, with $su,sw$ in place of $us,ws$.
\end{proposition}

The next elementary observation determines the left and right descents of $t_I$.

\begin{lemma}\label{lem:translation-descents}
For every $I\in\binom{[n]}k$,
\[
    s_it_Is_i=t_{s_iI}.
\]
Moreover:
\begin{enumerate}
\item if $i\in I$ and $i+1\notin I$, then
\[
    s_it_I<t_I<t_Is_i;
\]
\item if $i\notin I$ and $i+1\in I$, then
\[
    t_Is_i<t_I<s_it_I;
\]
\item if either both or neither of $i,i+1$ belong to $I$, then
\[
    t_I<s_it_I=t_Is_i,
    \qquad s_iI=I.
\]
\end{enumerate}
\end{lemma}

\begin{proof}
The conjugation identity follows directly from the definition of $t_I$. Choose compatible affine lifts. Then
\[
\begin{aligned}
 t_I(i)&=i+n\mathbf 1_{i\in I},
 &t_I(i+1)&=i+1+n\mathbf 1_{i+1\in I},\\
 t_I^{-1}(i)&=i-n\mathbf 1_{i\in I},
 &t_I^{-1}(i+1)&=i+1-n\mathbf 1_{i+1\in I}.
\end{aligned}
\]
The three cases now follow from the usual left- and right-descent criteria.
\end{proof}

\begin{corollary}\label{cor:ceq_geo}
Suppose $f,s_if,fs_i,s_ifs_i\in\Bkn$ and $\ell(s_if)=\ell(fs_i)$. Then for any $I\in\binom{[n]}k$, we have
\[
    \Pio_{s_if}\cap\{\Delta_I\neq0\}     \xrightarrow{\sim}    \Pio_{fs_i}\cap\{\Delta_{s_iI}\neq0\}.
\]
\end{corollary}

\begin{proof}
If $s_if=fs_i$, then $s_i(s_if)s_i=s_if$. Swapping columns $i$ and $i+1$ therefore preserves $\Pio_{s_if}=\Pio_{fs_i}$ and sends $\Delta_I$ to $\pm\Delta_{s_iI}$, so the claim is immediate. We may thus assume $s_if\neq fs_i$.

First suppose that both lengths equal $\ell(f)+1$. Then
\[
    \ell(s_ifs_i)=\ell(f)+2.
\]
Put
\[
    h=s_if,\qquad r=fs_i,\qquad g=s_ifs_i,     \qquad w=t_I,\qquad w'=t_{s_iI}=s_iws_i.
\]
By Theorem~\ref{thm:snideriso}, it suffices to compare $\mathring{\Rcal}_h^w$ and $\mathring{\Rcal}_r^{w'}$.

If $i\in I$ and $i+1\notin I$, then $h$ and $w$ both have a left $s_i$-descent. After removing these descents, $f=s_ih$ and $s_iw$ both have a right $s_i$-ascent. Proposition~\ref{prop:rank-one-richardson} gives
\[
    \mathring{\Rcal}_h^w       \cong \mathring{\Rcal}_f^{s_iw}       \cong \mathring{\Rcal}_r^{s_iws_i}       =\mathring{\Rcal}_r^{w'}.
\]

If both or neither of $i,i+1$ belong to $I$, then $w'=w$ and $s_iw=ws_i>w$. Since $h<g$ and $r<g$, moving first on the right and then on the left gives
\[
    \mathring{\Rcal}_h^w       \cong \mathring{\Rcal}_g^{ws_i}       \cong \mathring{\Rcal}_r^{s_iws_i}       =\mathring{\Rcal}_r^w.
\]
These are precisely the possibilities allowed by $i+1\in I\Rightarrow i\in I$.

Now suppose that both lengths equal $\ell(f)-1$. Put $f'=s_ifs_i$. Then $s_if'=fs_i$, $f's_i=s_if$, and the corresponding lengths for $f'$ are $\ell(f')+1$. Apply the first part to $f'$ and the chart $s_iI$, then invert the resulting isomorphism. The condition $i+1\in s_iI\Rightarrow i\in s_iI$ is exactly $i\in I\Rightarrow i+1\in I$.
\end{proof}

\begin{corollary}\label{cor:double1_geo}
Suppose $f,s_if,fs_i,s_ifs_i\in\Bkn$ and $\ell(s_ifs_i)=\ell(f)+2$. If $i\in I$ implies $i+1\in I,$ then $\Pio_{f,I}$ admits a decomposition
\[
    \Pio_{f,I}=Z_I\sqcup U_I,
\]
where $Z_I$ is closed, $U_I$ is its open complement, and
\[
    Z_I\cong \Pio_{s_ifs_i,s_iI}\times\C,     \qquad     U_I\cong \Pio_{s_if,s_iI}\times\C^*.
\]
\end{corollary}

\begin{proof}
Let $h=s_if$, $r=fs_i$, $g=s_ifs_i$, $w=t_I$, and $w'=t_{s_iI}=s_iws_i$.

Suppose first that $i\notin I$ and $i+1\in I$. Then $f<r$ and $ws_i<w$. The right-handed mixed case of Proposition~\ref{prop:rank-one-richardson} decomposes $\mathring{\Rcal}_f^w$ into a closed piece and its open complement with bases
\[
    \mathring{\Rcal}_r^{ws_i}     \qquad\text{and}\qquad     \mathring{\Rcal}_f^{ws_i},
\]
respectively. Since
\[
    r<s_ir=g,\qquad f<s_if=h,    \qquad ws_i<s_iws_i=w',
\]
the left-handed same-orientation isomorphism identifies these bases with $\mathring{\Rcal}_g^{w'}$ and $\mathring{\Rcal}_h^{w'}$.

If both or neither of $i,i+1$ belong to $I$, then $w'=w$ and $w<s_iw=ws_i$. First use the left-handed same-orientation isomorphism
\[
    \mathring{\Rcal}_f^w        \cong \mathring{\Rcal}_h^{s_iw}.
\]
The right-handed mixed decomposition of the latter has closed piece $\mathring{\Rcal}_g^w\times\C$ and open complement $\mathring{\Rcal}_h^w\times\C^*$.

Applying Theorem~\ref{thm:snideriso} in each case gives the asserted closed--open decomposition of $\Pio_{f,I}$.
\end{proof}

\begin{corollary}\label{cor:double2_geo}
Suppose $f,s_if,fs_i,s_ifs_i\in\Bkn$ and $\ell(s_ifs_i)=\ell(f)+2$. If $i+1\in I$ implies $i\in I,$ then $\Pio_{f,I}$ admits a decomposition
\[
    \Pio_{f,I}=Z_I'\sqcup U_I',
\]
where $Z_I'$ is closed, $U_I'$ is its open complement, and
\[
    Z_I'\cong \Pio_{s_ifs_i,s_iI}\times\C,     \qquad     U_I'\cong \Pio_{fs_i,s_iI}\times\C^*.
\]
\end{corollary}

\begin{proof}
Let $h=s_if$, $r=fs_i$, $g=s_ifs_i$, $w=t_I$, and $w'=t_{s_iI}=s_iws_i$.

Suppose first that $i\in I$ and $i+1\notin I$. Then $f<h$ and $s_iw<w$. The left-handed mixed case of Proposition~\ref{prop:rank-one-richardson} decomposes $\mathring{\Rcal}_f^w$ into a closed piece and its open complement with bases
\[
    \mathring{\Rcal}_h^{s_iw}     \qquad\text{and}\qquad     \mathring{\Rcal}_f^{s_iw},
\]
respectively. Since
\[
    h<hs_i=g,\qquad f<fs_i=r,     \qquad s_iw<s_iws_i=w',
\]
the right-handed same-orientation isomorphism identifies these bases with $\mathring{\Rcal}_g^{w'}$ and $\mathring{\Rcal}_r^{w'}$.

If both or neither of $i,i+1$ belong to $I$, then $w'=w$ and $w<ws_i=s_iw$. First use the right-handed same-orientation isomorphism
\[
    \mathring{\Rcal}_f^w        \cong \mathring{\Rcal}_r^{ws_i}.
\]
The left-handed mixed decomposition of the latter has closed piece $\mathring{\Rcal}_g^w\times\C$ and open complement $\mathring{\Rcal}_r^w\times\C^*$. Applying Theorem~\ref{thm:snideriso} completes the proof.
\end{proof}

\begin{remark}\label{rem:path-cases}
If $I=P_U$, then $i\in I$ exactly when the $i$th step of $P$ is an up-step. Thus Corollary~\ref{cor:double1_geo} applies when $s_iP=P$ or $(P_i,P_{i+1})=(R,U)$, whereas Corollary~\ref{cor:double2_geo} applies when $s_iP=P$ or $(P_i,P_{i+1})=(U,R)$. Over $\F_q$, the two decompositions give
\[
\begin{aligned}
 \#\Pio_{f,I}(\F_q)  &=q\,\#\Pio_{s_ifs_i,s_iI}(\F_q) +(q-1)\,\#\Pio_{s_if,s_iI}(\F_q),\\
 \#\Pio_{f,I}(\F_q)   &=q\,\#\Pio_{s_ifs_i,s_iI}(\F_q)     +(q-1)\,\#\Pio_{fs_i,s_iI}(\F_q),
\end{aligned}
\]
respectively. These are the geometric identities underlying the two
cases of Theorem~\ref{thm:eqmoves}.
\end{remark}

\begin{corollary}\label{cor:isoforprob}
Let $f,s_ifs_i\in\Bkn$ satisfy
\[
    \ell(f)=\ell(s_ifs_i).
\]
Then
\[
    \Pio_{f,\sGrI_{i+1}(f)} = \Pio_f \xrightarrow{\sim} \Pio_{s_ifs_i} = \Pio_{s_ifs_i,\sGrI_{i+1}(s_ifs_i)}.
\]
Moreover,
\[
    \sGrI_{i+1}(s_ifs_i) = t_{f^{-1}(i),f^{-1}(i+1)} \sGrI_{i+1}(f).
\]
\end{corollary}

\begin{proof}
The middle isomorphism follows from Theorem~\ref{thm:geometric}\ref{part4}. The first and last equalities follow from the source Grassmann necklace description of an open positroid variety: every Pl\"ucker coordinate indexed by an element of $\sGr(f)$ is nonzero on all of $\Pio_f$.

The final identity follows directly from the definition of the source Grassmann necklace. Conjugation by $s_i$ interchanges the two positions $f^{-1}(i)$ and $f^{-1}(i+1)$ relevant to the threshold at $i+1$, and leaves all other positions unchanged.
\end{proof}

\section{Open Problems and More}\label{sec:problems}

In this section, we collect a few problems that arise naturally from previous sections. The first problem comes directly from the definition of the source Grassmann necklace $\sGr(f)$.

\begin{problem}\label{prob:neck}
For any $f\in\mathbf{B}_{k,n}$ and $i\in[n]$, find a bijection
\begin{equation}\label{eq:neckbij}
	\affdeo{f}{P_{\sGrI_i(f)}} \to \affdeo{f}{P_{\sGrI_{i+1}(f)}}.
\end{equation}
\end{problem}

Corollary~\ref{cor:isoforprob} gives a geometric isomorphism between the open positroid varieties underlying the two affine Deogram decompositions below. The following problem asks for a direct combinatorial realization of this isomorphism.

\begin{problem}\label{prob:weak}
For any $f,s_ifs_i\in\mathbf{B}_{k,n}$ with $\ell(f)=\ell(s_ifs_i)$, find a bijection
\begin{equation}\label{eq:weakbij}
	\affdeo{f}{P_{\sGrI_{i+1}(f)}} \to \affdeo{s_ifs_i}{P_{t_{f^{-1}(i),f^{-1}(i+1)}\sGrI_{i+1}(f)}} = \affdeo{s_ifs_i}{P_{\sGrI_{i+1}(s_ifs_i)}}.
\end{equation}
\end{problem}

\begin{figure}
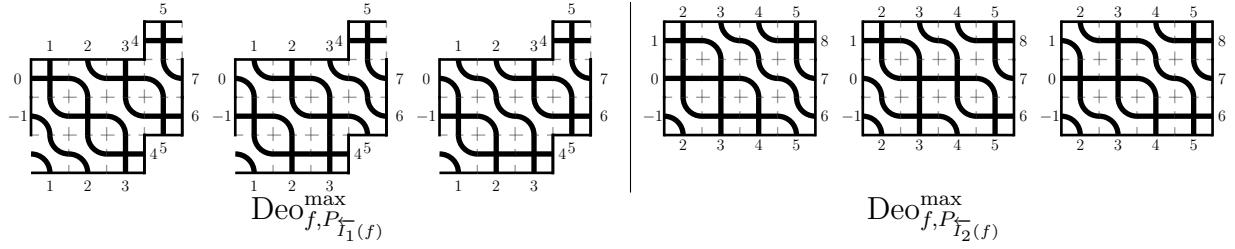


\setlength{\tabcolsep}{-4pt}
\begin{tabular}{ccc@{}c@{}ccc}

\deog{3/0,3/3,4/1,4/2,3/0,3/3,4/1,4/2,0/1,0/2}{0/0,0/3,1/0,1/3,2/0,2/3,3/1,3/4}{0.5/-0.35/1,0.5/3.35/1,1.5/-0.35/2,1.5/3.35/2,2.5/-0.35/3,2.5/3.35/3,3.25/0.5/4,2.75/3.5/4,3.5/0.65/5,3.5/4.35/5,4.35/1.5/6,-0.35/1.5/-1,4.35/2.5/7,-0.35/2.5/0}{1/2,3/2,0/1,2/1,0/0,1/0}{0/2,2/2,1/1,3/1,2/0,3/3}
&
\deog{3/0,3/3,4/1,4/2,3/0,3/3,4/1,4/2,0/1,0/2}{0/0,0/3,1/0,1/3,2/0,2/3,3/1,3/4}{0.5/-0.35/1,0.5/3.35/1,1.5/-0.35/2,1.5/3.35/2,2.5/-0.35/3,2.5/3.35/3,3.25/0.5/4,2.75/3.5/4,3.5/0.65/5,3.5/4.35/5,4.35/1.5/6,-0.35/1.5/-1,4.35/2.5/7,-0.35/2.5/0}{0/2,2/2,3/2,1/1,3/1,0/0}{1/2,0/1,2/1,1/0,2/0,3/3}
&
\deog{3/0,3/3,4/1,4/2,3/0,3/3,4/1,4/2,0/1,0/2}{0/0,0/3,1/0,1/3,2/0,2/3,3/1,3/4}{0.5/-0.35/1,0.5/3.35/1,1.5/-0.35/2,1.5/3.35/2,2.5/-0.35/3,2.5/3.35/3,3.25/0.5/4,2.75/3.5/4,3.5/0.65/5,3.5/4.35/5,4.35/1.5/6,-0.35/1.5/-1,4.35/2.5/7,-0.35/2.5/0}{0/2,1/2,3/2,1/1,2/1,0/0}{2/2,0/1,3/1,1/0,2/0,3/3}
&
\vline

&
\deog{0/1,0/2,0/3,4/1,4/2,4/3}{0/1,0/4,1/1,1/4,2/1,2/4,3/1,3/4}{4.25/3.5/8,-0.25/3.5/1,0.5/0.75/2,0.5/4.25/2,1.5/0.75/3,1.5/4.25/3,2.5/0.75/4,2.5/4.25/4,3.5/0.75/5,3.5/4.25/5,4.25/1.5/6,-0.3/1.5/-1,4.25/2.5/7,-0.25/2.5/0}{1/3,2/3,2/2,3/2,0/1,3/1}{0/3,3/3,0/2,1/2,1/1,2/1}
&
\deog{0/1,0/2,0/3,4/1,4/2,4/3}{0/1,0/4,1/1,1/4,2/1,2/4,3/1,3/4}{4.25/3.5/8,-0.25/3.5/1,0.5/0.75/2,0.5/4.25/2,1.5/0.75/3,1.5/4.25/3,2.5/0.75/4,2.5/4.25/4,3.5/0.75/5,3.5/4.25/5,4.25/1.5/6,-0.3/1.5/-1,4.25/2.5/7,-0.25/2.5/0}{1/3,2/3,0/2,3/2,0/1,2/1}{0/3,3/3,1/2,2/2,1/1,3/1}

&
\deog{0/1,0/2,0/3,4/1,4/2,4/3}{0/1,0/4,1/1,1/4,2/1,2/4,3/1,3/4}{4.25/3.5/8,-0.25/3.5/1,0.5/0.75/2,0.5/4.25/2,1.5/0.75/3,1.5/4.25/3,2.5/0.75/4,2.5/4.25/4,3.5/0.75/5,3.5/4.25/5,4.25/1.5/6,-0.3/1.5/-1,4.25/2.5/7,-0.25/2.5/0}{0/3,2/3,2/2,3/2,0/1,1/1}{1/3,3/3,0/2,1/2,2/1,3/1}
\\
\multicolumn{3}{c}{$\affdeomax{f}{P_{\sGrI_1(f)}}$} & & \multicolumn{3}{c}{$\affdeomax{f}{P_{\sGrI_2(f)}}$}
\end{tabular}
\caption{Example of Problem \ref{prob:neck} for $f = [3,5,6,8,7,10,12]$.}
\label{fig:cyclicshift}
\end{figure}

\begin{figure}
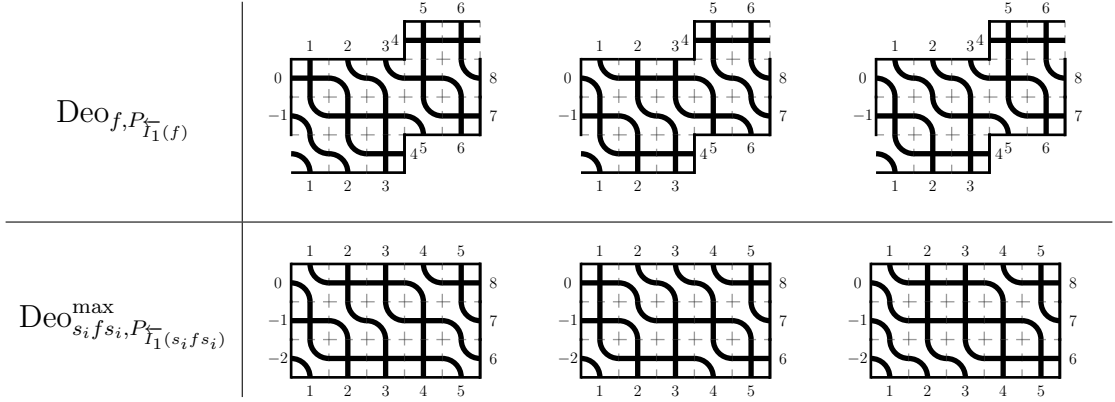

\bgroup
\renewcommand{\arraystretch}{4}
\begin{tabular}{c|ccc}

$\affdeo{f}{P_{\sGrI_1(f)}}$
&
\deog{3/0,3/3,5/1,0/1,0/2,5/2,3/0,3/3,5/1,5/2}{0/0,0/3,1/0,1/3,2/0,2/3,3/1,3/4,4/1,4/4}{0.5/-0.35/1,0.5/3.35/1,1.5/-0.35/2,1.5/3.35/2,2.5/-0.35/3,2.5/3.35/3,3.25/0.5/4,2.75/3.5/4,3.5/0.65/5,3.5/4.35/5,4.5/0.65/6,4.5/4.35/6,5.35/1.5/7,-0.35/1.5/-1,5.35/2.5/8,-0.35/2.5/0}{1/2,2/2,4/2,0/1,3/1,0/0,1/0}{0/2,3/2,1/1,2/1,4/1,2/0,3/3,4/3}

&
\deog{3/0,3/3,5/1,0/1,0/2,5/2,3/0,3/3,5/1,5/2}{0/0,0/3,1/0,1/3,2/0,2/3,3/1,3/4,4/1,4/4}{0.5/-0.35/1,0.5/3.35/1,1.5/-0.35/2,1.5/3.35/2,2.5/-0.35/3,2.5/3.35/3,3.25/0.5/4,2.75/3.5/4,3.5/0.65/5,3.5/4.35/5,4.5/0.65/6,4.5/4.35/6,5.35/1.5/7,-0.35/1.5/-1,5.35/2.5/8,-0.35/2.5/0}{0/2,3/2,4/2,1/1,2/1,4/1,0/0}{1/2,2/2,0/1,3/1,1/0,2/0,3/3,4/3}
&
\deog{3/0,3/3,5/1,0/1,0/2,5/2,3/0,3/3,5/1,5/2}{0/0,0/3,1/0,1/3,2/0,2/3,3/1,3/4,4/1,4/4}{0.5/-0.35/1,0.5/3.35/1,1.5/-0.35/2,1.5/3.35/2,2.5/-0.35/3,2.5/3.35/3,3.25/0.5/4,2.75/3.5/4,3.5/0.65/5,3.5/4.35/5,4.5/0.65/6,4.5/4.35/6,5.35/1.5/7,-0.35/1.5/-1,5.35/2.5/8,-0.35/2.5/0}{0/2,1/2,2/2,4/2,1/1,3/1,0/0}{3/2,0/1,2/1,4/1,1/0,2/0,3/3,4/3}

\\[4ex]\hline
$\affdeomax{s_ifs_i}{P_{\sGrI_1(s_ifs_i)}}$
&
\deog{5/0,5/1,5/2,5/0,5/1,5/2,0/0,0/1,0/2}{0/0,0/3,1/0,1/3,2/0,2/3,3/0,3/3,4/0,4/3}{0.5/-0.35/1,0.5/3.35/1,1.5/-0.35/2,1.5/3.35/2,2.5/-0.35/3,2.5/3.35/3,3.5/-0.35/4,3.5/3.35/4,4.5/-0.35/5,4.5/3.35/5,5.35/0.5/6,-0.35/0.5/-2,5.35/1.5/7,-0.35/1.5/-1,5.35/2.5/8,-0.35/2.5/0}{0/2,3/2,1/1,2/1,4/1,0/0,4/0}{1/2,2/2,4/2,0/1,3/1,1/0,2/0,3/0}
&
\deog{5/0,5/1,5/2,5/0,5/1,5/2,0/0,0/1,0/2}{0/0,0/3,1/0,1/3,2/0,2/3,3/0,3/3,4/0,4/3}{0.5/-0.35/1,0.5/3.35/1,1.5/-0.35/2,1.5/3.35/2,2.5/-0.35/3,2.5/3.35/3,3.5/-0.35/4,3.5/3.35/4,4.5/-0.35/5,4.5/3.35/5,5.35/0.5/6,-0.35/0.5/-2,5.35/1.5/7,-0.35/1.5/-1,5.35/2.5/8,-0.35/2.5/0}{1/2,2/2,3/2,1/1,4/1,0/0,3/0}{0/2,4/2,0/1,2/1,3/1,1/0,2/0,4/0}
&
\deog{5/0,5/1,5/2,5/0,5/1,5/2,0/0,0/1,0/2}{0/0,0/3,1/0,1/3,2/0,2/3,3/0,3/3,4/0,4/3}{0.5/-0.35/1,0.5/3.35/1,1.5/-0.35/2,1.5/3.35/2,2.5/-0.35/3,2.5/3.35/3,3.5/-0.35/4,3.5/3.35/4,4.5/-0.35/5,4.5/3.35/5,5.35/0.5/6,-0.35/0.5/-2,5.35/1.5/7,-0.35/1.5/-1,5.35/2.5/8,-0.35/2.5/0}{0/2,3/2,0/1,4/1,0/0,1/0,2/0}{1/2,2/2,4/2,1/1,2/1,3/1,3/0,4/0}

\end{tabular}
\egroup
\caption{Example of Problem \ref{prob:weak} for $f = [5,3,6,9,7,8,10,12]$ and $i=0$.}
\label{fig:weaksep}
\end{figure}

Fix $0<k<n$ with $\gcd(k,n)=1$. A $(k,n)$-Dyck path is a lattice path from $(0,0)$ to $(n-k,k)$ which never lies below the diagonal. Let $\Dyck_{k,n}$ denote the set of such paths. Let $f_{k,n}\in\Bkn$ be the bounded affine permutation given by $f_{k,n}(i)=i+k$. In \cite{gallam2024}, Galashin and Lam showed that the sets $\Dyck_{k,n}$ and $\deomax{f_{k,n}}$ are equinumerous. However, they did not do so combinatorially outside of the Catalan case $n=2k+1$ and left it as an open problem to find a bijective proof (see \cite[Problem 7.6]{gallam2024}). Theorem \ref{thm:problem} reduces the construction of such a bijection to Problems \ref{prob:neck} and \ref{prob:weak} to find a bijection between $\deomax{f_{k,n}}$ and $\Dyck_{k,n}$. To do so, we state some definitions from \cite{gallam2024}.

\begin{definition}
We say that $f\in \Bkn$ has a \textit{double crossing} at some $i\in\Z$ if $s_ifs_i<s_if<f$, $s_ifs_i<fs_i<f$, and $s_ifs_i,s_if,fs_i\in \Bkn$.
\end{definition}

In this case, we say $s_ifs_i$ is obtained from $f$ by a \textit{double move}.

\begin{definition}
Let $f\in\Bkn$, $i\in\Z$, and $f':=s_ifs_i$. If $\ell(f)=\ell(f')$ and $f'\in\Bkn$, we say that $f$ and $f'$ are related by a \textit{length-preserving simple conjugation}. We say that $f,g\in\Bkn$ are $c$-\textit{equivalent} and write $f \stackrel{c}{\sim} g$ if $f$ and $g$ can be related by a sequence of length-preserving simple conjugations.
\end{definition}

\begin{proof}[Proof of Theorem \ref{thm:problem}]
Using Corollary \ref{cor:affdeocorrespondence}, we biject $\deomax{f_{k,n}}\to \affdeomax{f_{k,n}}{P_{k,n}}$ where $P_{k,n}$ has $n-k$ right steps followed by $k$ up steps. We then apply the recursion constructed in \cite[Proposition 3.2]{gallam2024} on $\affdeomax{f_{k,n}}{P_{k,n}}$ using $c$-equivalence and double moves, and apply the recursion in \cite[Proposition 6.5]{gallam2024} on $\Dyck_{k,n}$. 

The decoupling bijection, Theorem~\ref{thm:decoupling_intro}, ensures that it suffices to restrict to cyclic bounded affine permutations. The double moves are encapsulated in Theorem \ref{thm:eqmoves}, along with some $c$-equivalence moves.

Throughout the process, we are finding bijections between (potentially unions of) sets of the form $\affdeomax{f}{P_{\sGrI_j(f)}}$.

If Problem \ref{prob:neck} is solved, whenever we reach a $c$-equivalence move that we do not have a bijection for, we apply \eqref{eq:neckbij}. Assuming $s_if>f$, we can always find an index $j$ such that either $i+1\in \sGrI_j(f)$ and $i\not\in \sGrI_j(f)$, or $s_i\sGrI_j(f) = \sGrI_j(f)$, since we may assume $\overline{f}$ has no fixed points due to Theorem~\ref{thm:decoupling_intro}. This allows us to perform the $c$-equivalence bijections of Theorem \ref{thm:eqmoves}.

If Problem \ref{prob:weak} is solved, we then have access to all $c$-equivalence moves, allowing us to move through the recursion of \cite{gallam2024}.
\end{proof}

We demonstrate a simple bijective proof in the case $k=2$. The case $k=n-2$ is analogous.

\begin{lemma}
Let $n\geq3$ be odd and $D\in\deomax{f_{2,n}}$. Every crossing of $D$ is adjacent exclusively to elbows.
\end{lemma}
\begin{proof}
Fix a crossing with vertical wire labeled by $i$ and horizontal wire labeled by $j$. There cannot be a vertically adjacent crossing as this would imply $f(i) = i$. There also cannot be a horizontally adjacent crossing as this would imply $f(j) > j+2$.
\end{proof}

Finally, since the bottom-left and top-right boxes must be filled with an elbow, we see we have a $2\times 2$ ``checkerboard'' filling of crossing-elbow pairs. However, since there are $n-2$ columns and we have $n-1$ elbows, there must be a unique odd-indexed column filled with two elbows. These are the $\frac1n\binom{n}{2} = \frac{n-1}2$ maximal Deograms in $\deomax{f_{2,n}}$. We can then define statistics on $D$.

\begin{definition}
Let $D\in\deomax{f_{2,n}}$ be given, and let $\text{col}(D)$ denote the column of the unique double elbow of $D$.

Define $\texttt{area}(D)$ as the number of elbows to the right of the top elbow in $\text{col}(D)$.

Define $\texttt{dinv}(D)$ as the number of elbows to the left of the bottom elbow in $\text{col}(D)$.
\end{definition}

\begin{theorem}\label{thm:direct2} Let $n\geq3$ be odd. We have a statistics-preserving bijection
\[
	\Phi:\deomax{f_{2,n}}\to\Dyck_{2,n}.
\]
Moreover, the map $\textbf{rot}:\deomax{f_{2,n}}\to \deomax{f_{2,n}}$ is a statistics-reversing bijection. That is, $(\texttt{area}(D),\texttt{dinv}(D)) = (\texttt{dinv}(\textbf{rot}(D)),\texttt{area}(\textbf{rot}(D)))$.
\end{theorem}
\begin{proof}

We define $\Phi$ as follows. Let $D\in\deomax{f_{2,n}}$ be given with $\text{col}(D)=2i+1$ for some $i\geq0$. Then $\Phi(D)$ is the Dyck path with the second up-step occurring at column $i+1$. This map is clearly bijective.

$\Phi$ is also statistics-preserving as each elbow not in $\text{col}(D)$ must alternate from bottom row to the top row. So, the number of elbows to the right of $\text{col}(D)$ is equal to the number of full boxes to the right of the up-step at column $i+1$ (up to the diagonal). In this case, the number of full boxes to the left of the up-step at column $i+1$ equals $\texttt{dinv}(\Phi(D))$, so equality is clear.

Finally $\textbf{rot}$ is a statistics-reversing bijection by the definitions of the statistics and of $\deomax{f_{2,n}}$.
\end{proof}

\begin{figure}

\begin{tabular}{ccc}

\deog{0/0,0/1,5/0,5/1,0/0,0/1,5/0,5/1}{0/2,1/2,2/2,3/2,4/2,0/0,1/0,2/0,3/0,4/0}{-0.35/0.5/1,-0.35/1.5/2,0.5/2.35/3,1.5/2.35/4,2.5/2.35/5,3.5/2.35/6,4.5/2.35/7,0.5/-0.3/1,1.5/-0.3/2,2.5/-0.3/3,3.5/-0.3/4,4.5/-0.3/5,5.3/0.5/6,5.3/1.5/7}{0/1,2/1,4/1,0/0,1/0,3/0}{1/1,3/1,2/0,4/0}

&

\deog{0/0,0/1,5/0,5/1,0/0,0/1,5/0,5/1}{0/2,1/2,2/2,3/2,4/2,0/0,1/0,2/0,3/0,4/0}{-0.35/0.5/1,-0.35/1.5/2,0.5/2.35/3,1.5/2.35/4,2.5/2.35/5,3.5/2.35/6,4.5/2.35/7,0.5/-0.3/1,1.5/-0.3/2,2.5/-0.3/3,3.5/-0.3/4,4.5/-0.3/5,5.3/0.5/6,5.3/1.5/7}{1/1,2/1,4/1,0/0,2/0,3/0}{0/1,3/1,1/0,4/0}

&

\deog{0/0,0/1,5/0,5/1,0/0,0/1,5/0,5/1}{0/2,1/2,2/2,3/2,4/2,0/0,1/0,2/0,3/0,4/0}{-0.35/0.5/1,-0.35/1.5/2,0.5/2.35/3,1.5/2.35/4,2.5/2.35/5,3.5/2.35/6,4.5/2.35/7,0.5/-0.3/1,1.5/-0.3/2,2.5/-0.3/3,3.5/-0.3/4,4.5/-0.3/5,5.3/0.5/6,5.3/1.5/7}{1/1,3/1,4/1,0/0,2/0,4/0}{0/1,2/1,1/0,3/0}
\\

$\downarrow$ & $\downarrow$ & $\downarrow$\\

\scalebox{\sclbx}{
\begin{tikzpicture}[xscale=0.5,yscale=0.5]

    \draw[black] (0,0) grid (5,2);

    \draw[red,line width=1.5pt] (0,0) -- (5,2);
    
    \draw[orange, thick,line width=3pt](0,0)--(0,2)--(5,2);
    
\end{tikzpicture}
}

&

\scalebox{\sclbx}{
\begin{tikzpicture}[xscale=0.5,yscale=0.5]

    \draw[black] (0,0) grid (5,2);

    \draw[red,line width=1.5pt] (0,0) -- (5,2);
    
    \draw[orange, thick,line width=3pt](0,0)--(0,1)--(1,1)--(1,2)--(5,2);
    
\end{tikzpicture}
}

&
\scalebox{\sclbx}{
\begin{tikzpicture}[xscale=0.5,yscale=0.5]

    \draw[black] (0,0) grid (5,2);

    \draw[red,line width=1.5pt] (0,0) -- (5,2);
    
    \draw[orange, thick,line width=3pt](0,0)--(0,1)--(2,1)--(2,2)--(5,2);
    
\end{tikzpicture}
}

\end{tabular}

\caption{Visualization of $\Phi:\deomax{f_{2,7}}\to\Dyck_{2,7}$. See Theorem \ref{thm:direct2}.}
\label{fig:directbij}
\end{figure}
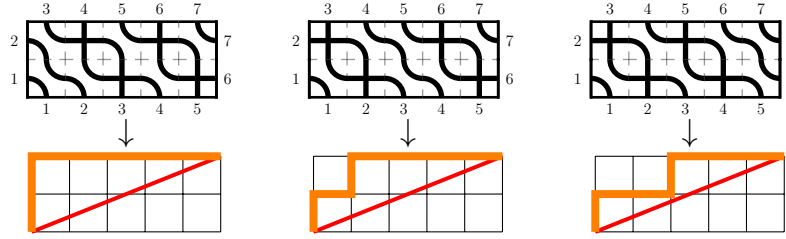

\bibliographystyle{alpha}
\bibliography{biblio}

\end{document}